\documentclass{amsart}
\usepackage{hyperref}

\usepackage{amsrefs}

\usepackage{amsmath,amssymb,amsthm,amsfonts}

\usepackage{color}

\newtheorem{theorem}{Theorem}[section]
\newtheorem{lemma}[theorem]{Lemma}
\newtheorem{proposition}[theorem]{Proposition}
\newtheorem{corollary}[theorem]{Corollary}
\newtheorem{remark}[theorem]{Remark}

\numberwithin{equation}{section}

\newcommand{\sourceG}{g}
\newcommand{\semiG}{\mathbb{A}}

 \newcommand{\rmi}{{\rm i}}

\newcommand{\R}{\mathbb{R}}
\newcommand{\Z}{\mathbb{Z}}
\newcommand{\Ndim}{n}

\newcommand{\id}{i}

\newcommand{\solU}{f}
\newcommand{\wN}{\ell}
\newcommand{\wNz}{\ell_0}
\newcommand{\PwNz}{\ell_0'}
\newcommand{\wE}{j}
\newcommand{\wM}{i}
\newcommand{\wJ}{j_*}
\newcommand{\wJt}{\tilde{j}}
\newcommand{\ALTsig}{\varrho}

\newcommand{\Am}{A}
\newcommand{\Ac}{A}

\newcommand{\vel}{v}
\newcommand{\spa}{x}
\newcommand{\fva}{\xi}
\newcommand{\dK}{K}

\newcommand{\NgE}{{K \ge 2\ksob}}
\newcommand{\DgE}{\Ndim \ge 3}

\newcommand{\CD}{\mathcal{D}}
\newcommand{\CE}{\mathcal{E}}
\newcommand{\CF}{\mathcal{F}}
\newcommand{\FM}{\mu}
\newcommand{\MM}{\sqrt{\mu}}

\newcommand{\FP}{\mathbf{P}}
\newcommand{\SB}{\mathbf{B}}
\newcommand{\FL}{L}
\newcommand{\FI}{\mathbf{I}}

\newcommand{\CI}{\mathcal{I}}

\newcommand{\CN}{N(\FL)}

\newcommand{\na}{\nabla}
\newcommand{\al}{\alpha}
\newcommand{\be}{\beta}
\newcommand{\la}{\lambda}

\newcommand{\pa}{\partial}
\newcommand{\ka}{\kappa}
\newcommand{\eps}{\epsilon}

\newcommand{\Ga}{\Gamma}

\newcommand{\eqdef}{\overset{\mbox{\tiny{def}}}{=}}
\newcommand{\threed}{{\mathbb R}^{\Ndim}}
\newcommand{\nsm}{|}

\newcommand{\ang}[1]{ \left< {#1} \right> }
\newcommand{\domain}{\mathbb{R}^\Ndim}
\newcommand{\ind}{ {\mathbf 1}}
\newcommand{\spacen}{N^{s,\gamma}}
\newcommand{\spaceL}{L^2_{\gamma+2s}}
\newcommand{\spaceELLn}{N^{s,\gamma}_\wN}
\newcommand{\ksob}{K^*_n}
\newcommand{\SMeps}{\eps_{\dK,\wN}}
\newcommand{\hilbertONE}{\mathcal{H}_{\vel}}
\newcommand{\hilbertTWO}{\mathcal{H}'_{\vel}}

\begin{document}

\title[The Boltzmann equation, Besov spaces, and decay rates in $\threed_{\spa}$]
{The Boltzmann equation, Besov spaces, and optimal time decay rates in $\threed_{\spa}$}

\author[V. Sohinger]{Vedran Sohinger}
\address{
University of Pennsylvania, Department of Mathematics, David Rittenhouse Lab, 209 South 33rd Street, Philadelphia, PA 19104-6395, USA}
\email{vedranso at math.upenn.edu}
\urladdr{http://www.math.upenn.edu/~vedranso/}
\thanks{V. S. was supported by a Simons Postdoctoral Fellowship.}

\author[R. M. Strain]{Robert M. Strain}
\email{strain at math.upenn.edu}
\urladdr{http://www.math.upenn.edu/~strain/}
\thanks{R.M.S. was partially supported by the NSF grants  DMS-1200747, DMS-0901463, and an Alfred P. Sloan Foundation Research Fellowship.}

\begin{abstract}
We prove that $k$-th order derivatives of perturbative classical solutions to the hard and soft potential Boltzmann equation (without the angular cut-off assumption) in the whole space, $\threed_x$ with $\DgE$, converge in large time to the global Maxwellian with the optimal decay rate of  $O\left(t^{-\frac{1}{2}\left(k+\ALTsig+\frac{\Ndim}{2}-\frac{\Ndim}{r}\right)}\right)$
in the $L^r_x(L^2_{\vel})$-norm for any $2\leq r\leq \infty$.   These results hold for any $\ALTsig \in (0, \Ndim/2]$ as long as initially $\| f_0\|_{\dot{B}^{-\ALTsig,\infty}_2 L^2_{\vel}} < \infty$.   In the hard potential case, we prove faster decay results in the sense that if $\| \FP f_0\|_{\dot{B}^{-\ALTsig,\infty}_2 L^2_{\vel}} < \infty$ and 
$\| \{\FI - \FP\} f_0\|_{\dot{B}^{-\ALTsig+1,\infty}_2 L^2_{\vel}} < \infty$ for $\ALTsig \in (\Ndim/2, (\Ndim+2)/2]$  then the solution decays the global Maxwellian in $L^2_\vel(L^2_x)$ with the optimal large time decay rate of $O\left(t^{-\frac{1}{2}\ALTsig}\right)$.
\end{abstract}



%
\setcounter{tocdepth}{1}
\maketitle
\tableofcontents

\thispagestyle{empty}





\section{Introduction and the main result}\label{first:sec}

The study of optimal large time decay rates in the whole space for perturbative solutions to non-linear dissipative partial differential equations with degenerate structure has received a substantial amount of attention in recent times, for example 
\cite{MR882376,MR2325837,ZhuStrain,ZhuStrain2, MR0609540, SdecaySOFT,2010arXiv1006.3605D,MR2357430,arXiv:0912.1742,MR677262,GuoWang2011,MR2764990,DYZ2011}.   For equations in which $L^2(\threed_{\spa})$ based norms can be propagated by the solution, it is common to make a smallness assumption on the $L^1(\threed_{\spa})$ norm of the initial data and combine this with $L^2(\threed_{\spa})$ type estimates in order to obtain large time decay estimates.  However it is often the case that propagating bounds on $L^1(\threed_{\spa})$ norms is difficult along the time evolution.  This can cause severe difficulties in applications because  one could improve existing theories by showing that an $L^1(\threed_{\spa})$ type norm is small or bounded after a finite but large time $T>0$, and then applying the aforementioned decay theory.  To overcome these types of difficulties, it is of great interest to prove 
 decay rates in an $L^2(\threed_{\spa})$ based space which is larger than $L^1(\threed_{\spa})$.  In this paper we accomplish this task for the non-cutoff Boltzmann equation in the homogeneous Besov-Lipschitz space $\dot{B}^{-\ALTsig,\infty}_2 \supset L^p(\threed_{\spa})$ where for $p\in [1,2]$ we use $\ALTsig = \frac{\Ndim}{p}-\frac{\Ndim}{2}$.  We remark that these spaces can be thought of as a physical choice since it is possible to obtain the $L^1(\threed_{\spa})$ embedding.   In the hard-potential case, we prove faster decay results in the 
spaces $\dot{B}^{-\ALTsig,\infty}_2$ for $\ALTsig \in (\frac{\Ndim}{2}, \frac{\Ndim+2}{2}]$.  
We  anticipate that our methods are applicable to a much wider class of degenerately dissipative equations.

For the non-cutoff Boltzmann equation, particularly for the soft potential case, there are two competing degenerate effects; so that this equation can be thought of as ``doubly degenerate''.  Firstly, for both the hard and the soft potentials, there is a degeneracy in the whole space because the macroscopic part of the solution is not a part of the dissipation.  Second there is a further degeneracy,  for the soft potentials, due to the weak velocity decay in the dissipation.  As described below, we develop new methods to overcome the combination of these difficulties in $\dot{B}^{-\ALTsig,\infty}_2 L^2_{\vel}$.

We study solutions to the {\em Boltzmann equation}, which is given by
  \begin{equation}
  \frac{\partial F}{\partial t} + v \cdot \nabla_{\spa} F = {\mathcal Q}(F,F),
  \quad
  F(0,x,\vel) = F_0(x,\vel).
  \label{BoltzFULL}
  \end{equation}
Here the unknown is $F=F(t,x,v)\ge 0$, which for $t\ge 0$ physically represents the density of particles in phase space. The spatial coordinates are $x\in\threed_{\spa}$, and velocities are $\vel\in\threed_{\vel}$ with 
$\DgE$.   The  {\em Boltzmann collision operator}, ${\mathcal Q}$,
is a bilinear operator which acts only on the velocity variables, $\vel$, instantaneously in $(t,x)$ as 
  \begin{equation*}
  {\mathcal Q} (G,F)(v) \eqdef 
  \int_{\mathbb{R}^n}  dv_* 
  \int_{\mathbb{S}^{n-1}}  d\sigma~ 
  B(v-v_*, \sigma) \, 
  \big[ G'_* F' - G_* F \big].
  \end{equation*} 
We use the standard shorthand $F = F(v)$, $G_* = G(v_*)$, $F' = F(v')$, $G_*^{\prime} = G(v'_*)$. 
In this expression,  $v$, $v_*$ and $v'$, $v' _*$  are 
the velocities of a pair of particles before and after collision.  They are connected through the formulas
  \begin{equation}
  v' = \frac{v+v_*}{2} + \frac{|v-v_*|}{2} \sigma, \qquad
  v'_* = \frac{v+v_*}{2} - \frac{|v-v_*|}{2} \sigma,
  \qquad \sigma \in \mathbb{S}^{n-1}.
\notag 
  \end{equation}
  We will discuss below in more detail the {\em Boltzmann collision kernel}, $B(v-v_*, \sigma)$.

 We will study the linearization of \eqref{BoltzFULL} around the Maxwellian equilibrium states 
\begin{equation}
F(t,x,v) = \FM(v)+\sqrt{\FM(v)} f(t,x,v),
\label{maxLIN}
\end{equation}
where without loss of generality the Maxwellian is given by
$$
\FM(v) \eqdef (2\pi)^{-n/2}e^{-|v|^2/2}.
$$
We use the homogeneous mixed Besov space $\dot{B}^{\ALTsig,\infty}_2 L^2_\vel $ with norm  
\begin{equation}\label{besovINFdef}
\| g\|_{\dot{B}^{\ALTsig,\infty}_2 L^2_\vel} 
\eqdef \sup_{j \in \Z} \left( 2^{\ALTsig j } \| \Delta_j g \|_{L^2(\threed_\spa \times \threed_\vel)} \right),
\quad 
\ALTsig \in \R.
\end{equation}
Here $\Delta_j$ are the standard Littlewood-Paley projections onto frequencies of order $2^j$ (in the spatial, $x$, variable only); they are defined in Section \ref{app.besov.h}.  We provide a discussion of more general Besov spaces in Section \ref{app.besov.h}.  
We suppose once and for all that $\dK$ is an integer satisfying $\NgE$, where $\ksob \eqdef \lfloor \frac{\Ndim}{2} +1 \rfloor$ is the smallest integer which is strictly greater than $\frac{\Ndim}{2}$.

Now given initial data  $\solU_0(x,\vel)$, to the Boltzmann equation \eqref{BoltzFULL} in the form \eqref{maxLIN}, we define an instant energy functional for the initial data as
\begin{equation}
\SMeps 
\approx
\sum_{|\al|+|\be|\leq \dK}\|w^{\wN-|\be|\rho}\pa^\al_\be \solU_0\|_{L^2_{\spa} L^2_{\vel}}^2.
\notag
\end{equation}
Above $\wN\ge 0$, $\alpha$ and $\beta$ are multi-indices, then $\pa^\al_\be$ is the corresponding high order spatial and velocity derivative, and $\rho>0$ is a parameter that is defined in the paragraph below \eqref{def.dNm}.   In particular  $\SMeps$ is defined again precisely in \eqref{def.id.jm}.   The current goal is to state our main theorems on time decay rates right away; therefore we postpone the definitions of the rest of our notation, and the statement of our existence and uniqueness Theorem \ref{thm.energy} (from \cite{gsNonCut0,SdecaySOFT}) until Section \ref{sec.notation} below.

Our main theorems are stated as follows:

\begin{theorem}\label{thm.main.decay}
Suppose that $\SMeps$ 
is sufficiently small with $\wN \ge \wNz$, where $\wNz$ 
is given by \eqref{wNz.def} below.  Consider the global solution $f(t,x,v)$ to the Boltzmann equation from Theorem \ref{thm.energy} with initial data $f_0(x,v)$.  

Fix $\ALTsig \in (0, \Ndim/2]$ and we consider any $k\in\{ 0,1,\ldots, \dK -1\}$.  Suppose additionally that $\| f_0\|_{\dot{B}^{-\ALTsig,\infty}_2 L^2_{\vel}} < \infty$.  Then uniformly for $t\ge 0$ we have
\begin{equation}
\label{Theorem1A}
\|  f(t) \|_{\dot{B}^{m,\infty}_2 L^2_{\vel}}^2 
\lesssim (1+t)^{-(m+\ALTsig)},
\quad \forall m \in [-\ALTsig, k].
\end{equation}
Furthermore 
\begin{equation}
\label{Theorem1B}
\sum_{k \le |\alpha| \le \dK} \| \partial^\alpha \solU(t) \|_{L^2_{\spa}L^2_{\vel}}^2
\lesssim (1+t)^{-(k+\ALTsig)}.
\end{equation}
\end{theorem}

Notice we have established the optimal $L^{p}_{\spa}$ decay rate  for $p\in [1,2)$ for all of the derivatives of order $k\in\{ 0,1,\ldots, \dK -1\}$ in the larger space $\dot{B}^{-\ALTsig,\infty}_2$ with $\ALTsig \in [0, \Ndim/2]$ and $\ALTsig= \frac{\Ndim}{p}-\frac{\Ndim}{2}$; and we only need to assume that initially $\| f_0\|_{\dot{B}^{-\ALTsig,\infty}_2 L^2_{\vel}} < \infty$.

\begin{corollary}\label{cor.decay.interp}
Fix any $2\leq r\leq \infty$ and $k\in\{ 0,1,\ldots, \dK -1\}$ satisfying the inequality $k< \dK-1 -\frac{\Ndim}{2}+\frac{\Ndim}{r}$.\footnote{We remark that the proof of Corollary \ref{cor.decay.interp} easily shows that if $2\leq r< \infty$ then we can allow  $k\le  \dK-1 -\frac{\Ndim}{2}+\frac{\Ndim}{r}$, and we only need to restrict to $k< \dK-1 -\frac{\Ndim}{2}$ when $r=\infty$.  }  Suppose all the conditions from Theorem \ref{thm.main.decay} hold.    
Then for $\ALTsig \in (0, \Ndim/2]$ with $\| f_0\|_{\dot{B}^{-\ALTsig,\infty}_2 L^2_{\vel}} < \infty$, we have the following  estimate
\begin{equation} \notag 
\sum_{|\alpha| = k } \| \partial^\alpha \solU(t)\|_{L^r_\spa L^2_{\vel}}^2
\lesssim(1+t)^{-k - \ALTsig-\frac{\Ndim}{2}+\frac{\Ndim}{r}},
\end{equation}
which holds uniformly over $t \ge 0$. 
\end{corollary}

A substantial difficulty in proving Theorem \ref{thm.main.decay} is the fact that we need to use nonlinear energy estimates in vector-valued mixed-norm spaces, such as $L^p_x \hilbertONE$, where $\hilbertONE$ is a separable Hilbert space in the $\vel$ variable. As we will see, such estimates arise naturally as a result of the definition of mixed-norm Besov spaces in \eqref{besovINFdef}, as well as from the definition of the 
sharp weighted geometric fractional  Sobolev norm in
\eqref{normdef}.   We overcome this difficulty by using an elaborate analysis of vector-valued Calder\'{o}n-Zygmund theory given in Section \ref{secAPP:INTERP}. We will give a more precise account of these ideas where they come up naturally below.

Regarding Corollary \ref{cor.decay.interp}, 
one can, in principle, use the methods described in the proof to obtain decay estimates in the stronger norm $L^2_{\vel}L^r_{\spa}$.  In order to do this, we would have to reverse the order of the norms in the interpolation estimates of Section \ref{secAPP:INTERP} (which is possible). We do not currently pursue this issue. 

We can furthermore analyze the full energy functional defined in \eqref{def.eNm}.
In the hard potential case \eqref{kernelP}, we have faster decay results.

\begin{theorem}\label{thm.extra.decay}
Suppose that $\SMeps$ 
is sufficiently small with $\wN \ge \wNz$, where $\wNz$ 
is given by \eqref{wNz.def} below.  Consider the global solution $f(t,x,v)$ to the Boltzmann equation from Theorem \ref{thm.energy} with initial data $f_0(x,v)$.  

Fix $\ALTsig \in (0, \Ndim/2]$, suppose additionally that $\| f_0\|_{\dot{B}^{-\ALTsig,\infty}_2 L^2_{\vel}} < \infty$.  Then 
\begin{equation}
\label{TheoremExtra}
\CE_{\dK,\wN}(t)
\lesssim (1+t)^{-\ALTsig},
\end{equation}
which holds uniformly for $t\ge 0$.  Here $\CE_{\dK,\wN}(t)$ is the full instant energy functional given as in  \eqref{def.eNm} and Theorem \ref{thm.energy} below.
Furthermore, in the hard potential case \eqref{kernelP}, for $\ALTsig \in (\Ndim/2, (\Ndim+2)/2]$, and $\FP$ defined in \eqref{form.p} below,   if 
$$
\| \FP f_0\|_{\dot{B}^{-\ALTsig,\infty}_2 L^2_{\vel}} + \| \{\FI - \FP\} f_0\|_{\dot{B}^{-\ALTsig+1,\infty}_2 L^2_{\vel}} < \infty,
$$ 
 then the solution also uniformly satisfies \eqref{TheoremExtra} with this $\ALTsig$.
\end{theorem}

As we will see, the proofs of Theorem \ref{thm.main.decay} and Theorem \ref{thm.extra.decay} are conceptually very different. The proof of Theorem \ref{thm.main.decay} relies on nonlinear energy estimates, whereas the proof of Theorem \ref{thm.extra.decay} is based on a comparison to the linear problem. The link with the analysis of the linear problem allows us to make a connection with the pioneering work of Ellis and Pinsky \cite{MR0609540}. After some efforts, their methods can be extended to our case and they can give us a more precise understanding of the linearized Boltzmann operator for small spatial frequencies. It turns out that this part of the linearized operator is the most difficult to control. Having obtained this control, we can obtain a gain of a factor $t^{-\frac{1}{2}}$, and we can deduce the optimal decay rates for an optimal range of parameters, in the sense that we will now explain.

When we say that these large time decay rates are ``optimal'' we mean that they are the same as those for the linear Boltzmann equation \eqref{BoltzLIN}, as seen in Theorem \ref{thm.decay.lin}.  The optimal rates in $L^r_\spa L^2_{\vel}$, from Corollary \ref{cor.decay.interp}  also hold for \eqref{BoltzLIN}.    These rates for $0$-th order derivatives also coincide with classical time-decay results for the Boltzmann equation \cite{MR882376,MR677262} with angular cut-off studied using spectral analysis.  We note that the method of  \cite{GuoWang2011} does not obtain the optimal $L^1$ decay rates.
Related recent results, concerning Besov spaces and the Boltzmann equation, which appeared after this work was complete can be found in \cite{ArsenioMasmoudi} and \cite{DuanLiuXu}.

The decay rates that we obtain are also consistent with the classical optimal large time decay rates for the heat equation; see for instance \cite{MR1938147}.    In particular it is well known that if $g_0(x)$ is a tempered distribution vanishing at infinity and satisfying $\| g_0\|_{\dot{B}^{-\ALTsig,\infty}_2(\threed_\spa)} < \infty$, then one further has
$$
\| g_0\|_{\dot{B}^{-\ALTsig,\infty}_2(\threed_\spa)}
\approx
\left\| t^{\ALTsig/2} \left\|  e^{t \Delta} g_0 \right\|_{L^2(\threed_{\spa})}  \right\|_{L^\infty_t((0, \infty) )}, \quad \text{for any $\ALTsig>0$.}
$$
See for instance \cite[Theorem 5.4]{MR1938147} where further references and more general results can be found.
Notice that the faster decay rates of higher derivatives for solutions to the heat equation can be easily obtained in the same way.

Notice that for the heat equation, and for the linear Boltzmann equation \eqref{BoltzLIN} in Theorem \ref{thm.decay.lin}, these decay results using initial conditions in the negative regularity Besov spaces (of order ``$-\ALTsig$'') hold for any $\ALTsig  > 0$.  For the non-linear problem  the restriction of $0<\ALTsig \le (\Ndim + 2)/2$ is also encountered in the large time optimal decay rates for the incompressible Navier-Stokes system; see  \cite{MR881519} and \cite[Ch. 26]{MR1938147}.  For  incompressible Navier-Stokes, it appears that we may not hope to go beyond $\ALTsig = (\Ndim + 2)/2$ without choosing special initial data \cite{MR1938147}. 
Thus the range $0<\ALTsig \le (\Ndim + 2)/2$ seems to us to represent a satisfying theory of decay rates in these spaces.

We are furthermore concerned in this paper primarily with obtaining the optimal large time convergence rates.  In that light we are not as concerned with optimizing the assumptions that we use on the regularity ($\dK \ge 2\ksob$) or the number of weights placed on the initial data ($\wN \ge \wNz$ with $\wNz$ 
from \eqref{wNz.def} below).

We obtain  decay for all derivatives of order $k\in\{ 0,1,\ldots, \dK -1\}$, where $\dK$  is the Sobolev regularity of the initial data in $\SMeps$ from \eqref{def.id.jm} and the existence theory in Theorem \ref{thm.energy}.  Our obstruction to obtaining the higher order decay of the highest order derivative $\dK$ comes from the estimates of the functionals $\CI^k(t)$ in Lemma \ref{lem.inter.f}, which fatally contain error terms including derivatives of order $k+1$ when controlling derivative energy estimates of order $k$.  

In the rest of this section, we will finish introducing the full model \eqref{BoltzFULL}, including the collision kernel, and then we discuss its geometric fractionally diffusive behavior.  The  Boltzmann collision kernel, $B(v-v_*, \sigma)$, will physically depend upon the {\em relative velocity} $|v-v_*|$ and on
the {\em deviation angle}  $\theta$ through the formula
$\cos \theta = ( v-v_* )\cdot \sigma/|v-v_*|$ where, without restriction, we can suppose by symmetry that  $B(v-v_*, \sigma)$
is supported on $\cos \theta \ge 0$.

\subsection*{The Collision Kernel}

Our assumptions are the following:
 \begin{itemize}
 \item 
We suppose that $B(v-v_*, \sigma)$ takes product form in its arguments as
\begin{equation}
B(v-v_*, \sigma) =\Phi( |v-v_*| ) \, b(\cos \theta).
\notag
\end{equation}
It generally holds that both $b$ and $\Phi$ are non-negative functions. 
 
  \item The angular function $t \mapsto b(t)$ is not locally integrable; for $c_b >0$ it satisfies
\begin{equation}
\frac{c_b}{\theta^{1+2s}} 
\le 
\sin^{\Ndim -2} \theta  ~ b(\cos \theta) 
\le 
\frac{1}{c_b\theta^{1+2s}},
\quad
s \in (0,1),
\quad
   \forall \, \theta \in \left(0,\frac{\pi}{2} \right].
   \label{kernelQ}
\end{equation}
 
  \item The kinetic factor $z \mapsto \Phi(|z|)$ satisfies for some $C_\Phi >0$
\begin{equation}
\Phi( |v-v_*| ) =  C_\Phi  |v-v_*|^\gamma , \quad \gamma  \ge -2s.
\label{kernelP}
\end{equation}
In the rest of this paper these will be called ``hard potentials.'' 

  \item Our results will also apply to the more singular situation
\begin{equation}
\Phi( |v-v_*| ) =  C_\Phi  |v-v_*|^\gamma , \quad  -2s > \gamma  > -\Ndim,
\quad \gamma + 2s > -\frac{\Ndim}{2}.
\label{kernelPsing}
\end{equation}
These will be called ``soft potentials'' throughout this paper.

 \end{itemize}
 
These collision kernels are physically motivated since they can be derived from a spherical intermolecular repulsive potential such as 
$
\phi(r)=r^{-(p-1)}
$
with $p \in (2,\infty)$ as shown by Maxwell in 1866.  In the physical dimension $(n=3)$,  $B$ satisfies the conditions above with $\gamma = (p-5)/(p-1)$ and $s = 1/(p-1)$;  see \cite{MR1942465}.

We linearize the Boltzmann equation \eqref{BoltzFULL} around \eqref{maxLIN}.  This grants an equation for the perturbation, $f(t,x,v)$, that is given by
\begin{gather}
 \partial_t f + v\cdot \nabla_{\spa} f + \FL (f)
=
\Ga (f,f),
\quad
f(0, x, v) = f_0(x,v),
\label{Boltz}
\end{gather}
where the {\it linearized Boltzmann operator}, $\FL$, is defined as
$$
 \FL(g)
 \eqdef  
- \FM^{-1/2}\mathcal{Q}(\FM ,\MM g)- \FM^{-1/2}\mathcal{Q}(\MM g,\FM),
$$
and the bilinear operator, $\Gamma$, is then
\begin{gather}
\Gamma (g,h)
\eqdef
 \FM^{-1/2}\mathcal{Q}(\MM g,\MM h).
\label{gamma0}
\end{gather}
The $(\Ndim+2)$-dimensional null space of $\FL$ is well known \cite{MR1379589}: 
\begin{equation}\label{nullLsp}
    \CN \eqdef {\rm span}\left\{\sqrt{\FM}, ~  \vel_1\sqrt{\FM}, ~\ldots, ~  \vel_\Ndim\sqrt{\FM}, ~   (|\vel|^2-\Ndim)\sqrt{\FM} \right\}.
\end{equation}
Now, for fixed $(t,x)$, we define the orthogonal projection from
$L^2_{\vel}$ to $\CN$ as
\begin{equation}
 \FP f=a^f(t,x)\sqrt{\FM}
    +\sum_{i=1}^{\Ndim}
    b^f_i(t,x)\vel_i\sqrt{\FM}+c^f(t,x) \frac{1}{\sqrt{2\Ndim}}(|\vel|^2-\Ndim)\sqrt{\FM}, \label{form.p}
\end{equation}
where the functions $a^f$, $b^f\eqdef [b_1,\cdots,b_n]$ and $c^f$ are defined by
\begin{equation}
\label{coef.p.def}
\begin{split}
     & a= \langle \MM, \solU\rangle= \langle \MM, \FP \solU\rangle,
  \\
    & b_i=\langle \vel_i \MM, \solU\rangle
=\langle \vel_i \MM,\FP \solU\rangle,
\\
   & c= \frac{1}{\sqrt{2\Ndim}}\langle (|\vel|^2-\Ndim) \MM, \solU\rangle
= \frac{1}{\sqrt{2\Ndim}}\langle (|\vel|^2-\Ndim) \MM, \FP \solU\rangle.
    \end{split}
\end{equation}
Then we can write $f =  \FP f + \{ \FI- \FP\} f$.  
It is a well-known fact \cite{MR1379589} that:
\begin{equation}
\label{ProjectionGamma}
\FP \Ga (f, f) = 0.
\end{equation}
We further define $[a,b,c]$ to be the vector with components $a,b,c$.  And $|[a,b,c]|$ is the standard Euclidean length of these vectors.

In the works \cite{gsNonCut0,gsNonCutA,gsNonCutEst},
Gressman and the second author have introduced into the Boltzmann theory the 
following sharp weighted geometric fractional  Sobolev norm:
\begin{equation} 
\nsm f\nsm_{\spacen}^2 
\eqdef 
\nsm f\nsm_{L^2_{\gamma+2s}}^2 + 
\int_{\mathbb{R}^n} dv \int_{\mathbb{R}^n} dv' ~
(w(\vel) w(\vel') )^{\frac{\gamma+2s+1}{2}}
 \frac{|f' - f|^2}{d(v,v')^{n+2s}} 
\ind_{d(v,v') \leq 1}. \label{normdef} 
\end{equation}
Generally, $\ind_{A}$ is the standard indicator function of the set $A$.
Now this space includes the weighted $L^2_\wN$ space, for $\wN\in\mathbb{R}$, with norm given by
$$
\nsm f\nsm_{L^2_{\wN}}^2 
\eqdef
\int_{\mathbb{R}^n} dv ~ 
w^{\wN}(\vel) 
~
|f(v)|^2.
$$
The weight function\footnote{We point out that our notation for \eqref{weigh} is different from the notation in the second author's previous papers \cite{gsNonCut0,SdecaySOFT} when $\gamma+2s<0$ from \eqref{kernelPsing}.}
 is defined as follows
\begin{equation}
w(\vel) 
\eqdef 
\ang{\vel}, \quad
\ang{v}
\eqdef \sqrt{1+|v|^2}.
\label{weigh}
\end{equation}
The fractional differentiation effects are measured 
using the anisotropic metric $d(v,v')$ on the ``lifted'' paraboloid (in $\R^{\Ndim+1}$) as
$$
d(v,v') \eqdef \sqrt{ |v-v'|^2 + \frac{1}{4}\left( |v|^2 -  |v'|^2\right)^2}.
$$
This metric encodes the nonlocal anisotropic changes in the power of the weight. 

In this space, the linearized collision operator $\FL$ is non-negative in $L^2_{\vel}$ and it is coercive in
the sense that there is a constant $\la>0$ such that \cite[Theorem 8.1]{gsNonCut0}:  
\begin{equation}\label{coerc}
 \ang{g, \FL g}
\geq 
\la\nsm\{\FI-\FP\}g\nsm^2_{\spacen}.
\end{equation}
The norm $\spacen$ provides a sharp  characterization of the linearized collision operator \cite[(2.13)]{gsNonCut0}; in earlier work \cite{MR2322149} the sharp gain of velocity weight in $L^2_{\vel}$ was established for the non-derivative part of \eqref{normdef}.

\subsection{Discussion of the method}
There have been numerous investigations on the rate of convergence to Maxwellian equilibrium for the nonlinear Boltzmann equation or related kinetic equations in the whole space; see for example 
\cite{arXiv:0912.1742,MR575897,MR2209761,MR2366140,strainSOFT,SdecaySOFT,MR2357430,2010arXiv1006.3605D,GuoWang2011}.  Many of the early results are well documented in Glassey \cite{MR1379589}.  Further detailed discussions of more recent results can be found in \cite{2010arXiv1006.3605D} and \cite{SdecaySOFT}.  We  point out that this current work was motivated by several recent results 
\cite{gsNonCutA,gsNonCut0,gsNonCutEst,arXiv:0912.1742, MR0609540,2010arXiv1006.3605D,MR2209761,MR2366140,SdecaySOFT,GuoWang2011}.


To establish our main results in Theorem \ref{thm.main.decay}, we start from a differential inequality for high order derivatives \eqref{main.diff.ce}. In this inequality, we then apply a time weighted energy estimate combined with a new time-regularity comparison via dyadic decomposition. In this way, we can achieve the optimal decay rates in spite of the degenerate dissipation. 

 In order to obtain the high order differential inequality mentioned above, e.g. \eqref{main.diff.ce}, we need develop several product interpolation estimates.  The details of the proofs of these estimates are contained in Section \ref{secAPP:INTERP}.  There is a rather serious added complication in the context of the non cut-off Boltzmann equation because all of the product estimates are in spaces such as $L^p_{\spa}\spacen$ with the exotic non-isotropic $\spacen$ as in \eqref{normdef}.  We take the approach of using Sobolev type inequalities (in the variable $\spa$) in the mixed spaces $L^p_{\spa}\hilbertONE$ where $\hilbertONE$ is an arbitrary separable Hilbert space (in the variable $\vel$).  Then the desired product estimates for $L^p_{\spa}\spacen$ follow in the special case when $\hilbertONE = \spacen$.

Additionally, to prove the faster decay rates in Theorem \ref{thm.extra.decay} we generally use the time weighted estimates combined with the linear decay theory.  Notice that the slower decay results in Theorem \ref{thm.main.decay} and Corollary \ref{cor.decay.interp} are proven ``Non-Linearly'' without the linear decay theory since we can establish a uniform in time upper bound on $\|  f(t) \|_{\dot{B}^{-\ALTsig,\infty}_2 L^2_{\vel}}^2$ for $\ALTsig \in (0, \frac{\Ndim}{2}]$ as in \eqref{Theorem1A}.  In the regime when 
$\ALTsig \in (\frac{\Ndim}{2}, \frac{\Ndim+2}{2}]$ we are unable to prove these uniform bounds and they may be unavailable.  Instead we prove the faster linear decay rates in Theorem \ref{linear.decay.A} by using a detailed frequency analysis of the linearized operator \eqref{FT.SB} in a small frequency ball around the origin, which is motivated by the work of Ellis and Pinsky \cite{MR0609540}.  Using this refined linear decay theory, we gain an additional order of $t^{-\frac{1}{2}}$ decay on the non-linear term because it is purely microscopic.  We also need to iterate this non-linear decay analysis a finite number of times in order to overcome degeneracies in the time integral estimates and obtain the optimal decay rates.  

We will explain other difficulties when they are encountered at the appropriate places throughout the course of the paper.  In the next sub-section we will give a detailed description of the remaining notation, as well as stating the relevant existence result from Theorem \ref{thm.energy}.

\subsection{Notation, and the existence result}\label{sec.notation}

 For any non-negative integer $m$, we use $H^m$ to denote the
usual Sobolev spaces $H^m(\threed_{\spa}\times\threed_{\vel})$, $H^m(\threed_{\spa})$, or  $H^m(\threed_{\vel})$, respectively, where for example
$$
H^m_{\ell}(\threed_{\vel})
\eqdef
\left\{ f \in L^2_\ell (\threed_{\vel}) : \nsm f \nsm_{H^m_{\ell}(\threed)}^2 
\eqdef \sum_{|\beta| \le m}\int_{\threed} dv ~ w^\ell (\vel) \left| \partial_\beta f(v) \right|^2 <\infty
\right\}.
$$
Then let us denote $H^m_{0} = H^m$.  Here, for multi-indices, we denote
\begin{equation*}
 \pa^{\al}_\be=\pa_{x_1}^{\al_1}\cdots\pa_{x_\Ndim}^{\al_\Ndim}
    \pa_{\vel_1}^{\be_1}\cdots\pa_{\vel_\Ndim}^{\be_\Ndim},
    \quad
    \al=[\al_1,\ldots,\al_\Ndim],
    \quad
\be=[\be_1,\ldots,\be_\Ndim].
\end{equation*}
The length of $\al$ is $|\al|=\al_1+\cdots+\al_\Ndim$ and the length of
$\be$ is $|\be|=\be_1+\cdots+\be_\Ndim$.

Then comparisons of $\nsm \cdot \nsm_{\spacen}$ to the weighted isotropic Sobolev spaces are:
\begin{equation}
\nsm g \nsm_{L^2_{\gamma+2s}(\threed_{\vel})}^2
\lesssim 
\nsm g\nsm_{\spacen}^2 
\lesssim \nsm g \nsm_{H^1_{\gamma+2s}(\threed_{\vel})}^2.
\label{isocomp}
\end{equation}
See \cite[Eq. (2.15)]{gsNonCut0} for the sharp comparison which in particular implies \eqref{isocomp}.

For Banach spaces $X$ and $Y$, we write $X(\threed_{\vel}) = X_{\vel}$ and $Y(\threed_\spa) = Y_\spa$. 
Specifically, we use $L^p_\spa$, $L^p_{\vel}$ and $L^p_{\spa,\vel}$  to denote $L^p(\threed_\spa)$, $L^p(\threed_{\vel})$ and $L^p(\threed_\spa \times \threed_{\vel})$  with $1 \le p \le \infty$ respectively.  There should be no confusion between $L^2_{\spa}$, $L^2_{\vel}$ and $L^2_\ell$, etc, since $x$ and $\vel$ are never used to denote a weight. Given the spaces as above, we define the following ordered mixed spaces:  
\begin{eqnarray}
\label{definition.mixednorm}
\| h\|_{Y_{\spa} X_{\vel}}
& \eqdef &
\left\|~ \nsm h\nsm_{X(\threed_{\vel})} ~ \right\|_{Y(\threed_{\spa})},
\\
\| h\|_{ X_{\vel}Y_{\spa}}
& \eqdef &
 \nsm ~ \left\| h  \right\|_{Y(\threed_{\spa})}~ \nsm_{X(\threed_{\vel})},
\end{eqnarray}
Thus for example, as in Corollary \ref{cor.decay.interp}, the space $L^p_{\vel} L^r_{\spa}$ may be different from $ L^r_{\spa} L^p_{\vel}$.  Generally a norm with only one line $\nsm \cdot \nsm_{X_{\vel}}$ denotes that it is only in the ``$\vel$'' variable, however a norm with two sets of lines $\| \cdot \|$ is either in both variables ``$(\spa,\vel)$'' or only in the ``$\spa$'' variable (and there should be no confusion between these cases).  We remark that this is a slight departure from the notation used in the second author's previous papers, e.g. 
\cite{gsNonCut0,gsNonCutEst,gsNonCutA,SdecaySOFT}; in this paper it is necessary to distinguish between the ordering of the evaluation of the norms.

Recalling the notations surrounding \eqref{besovINFdef}, and in Section \ref{app.besov.h},  we will also use the following mixed Besov space semi-norm as
\begin{equation}
\label{Besovseminorm}
\|  g \|_{\dot{B}^{\ALTsig,q}_p X_{\vel}}
 \eqdef 
\|\left( 2^{\ALTsig j } \| \Delta_j g \|_{L^p_{\spa} X_{\vel}} \right)_j\|_{\ell^{q}_j},
\quad p,q \in [1,\infty], \quad \ALTsig \in \R,
\end{equation}
where for a sequence, $\left( a_j \right)_{j\in \Z}$, we use the standard $\ell^{q}_j$ norm as
$$
\|\left(a_j \right)_j\|_{\ell^{q}_j} \eqdef \left( \sum_{j\in \Z} |a_j|^q \right)^{1/q}
\quad \text{for $1\le q < \infty$,  and}
\quad
\|\left(a_j \right)_j\|_{\ell^{\infty}_j} \eqdef  \sup_{j\in \Z} |a_j|.
$$
Thus $\dot{B}^{s,q}_p$ is the homogeneous Besov space in the variable $x$.
We always use the Besov space and the frequency projection $ \Delta_j $ only in the spatial variable $\spa \in \threed_x$.
Notice that in the special case of Besov semi-norms, $\|  \cdot \|_{\dot{B}^{\ALTsig,q}_p X_{\vel}}$, we do not follow the Banach space ordering convention described above. As we will see, this will add an additional complication in the interpolation estimates we want to use.
 
We use $\langle\cdot,\cdot\rangle$ to denote the inner product over the Hilbert space $L^2_{\vel}$, i.e.
\begin{equation*}
    \langle g,h\rangle=\int_{\threed} g(\vel) \overline{h(\vel)} ~ d\vel,\ \ g=g(\vel), ~h=h(\vel)\in
    L^2_{\vel}.
\end{equation*}
Analogously $\left(\cdot,\cdot\right)$ denotes the inner product over $L^2(\threed_{\spa}\times\threed_{\vel})$.

Recalling \eqref{weigh},  we then consider the weighted anisotropic derivative space as in \eqref{normdef}: 
 \begin{gather*}
\nsm h \nsm^2_{{\spaceELLn}}
\eqdef 
\nsm w^{\ell} h\nsm_{L^2_{\gamma+2s}}^2 + \int_{\mathbb{R}^n} dv ~ 
w^{\gamma+2s+1}(v) 
w^{2\ell}(v)
\int_{\mathbb{R}^n} dv' 
~
 \frac{|h' - h|^2}{d(v,v')^{n+2s}} 
\ind_{d(v,v') \leq 1}.
\end{gather*}
Note that  
$
\nsm h \nsm_{{\spacen}}
=
\nsm h \nsm_{{\spacen_0}}.
$
Throughout the paper, we will use $\hilbertONE$ to denote a separable Hilbert space in the $v$ variable. In particular, one can take $\hilbertONE$ to be $L^2_v$, $L^2_{\gamma+2s}$, $H^s_{v}$ and $\mathbb{R}$ with the absolute value. Moreover, one can take $\hilbertONE=N^{s,\gamma}_{\ell}$. That $N^{s,\gamma}_{\ell}$ is separable follows from the equivalence of norms given in \cite[Eq. (7.7)]{gsNonCut0}.

Given a solution, $\solU(t,x,\vel)$, to the Boltzmann equation \eqref{Boltz}, we define an instant energy functional to be a continuous function, 
 $\CE_{\dK,\wN}(t)$, which satisfies
\begin{equation}
\CE_{\dK,\wN}(t)\approx
\sum_{|\al|+|\be|\leq \dK}\|w^{\wN-|\be|\rho}\pa^\al_\be \solU(t)\|_{L^2_{\spa} L^2_{\vel}}^2.
\label{def.eNm}
\end{equation}
Above and below: $\wN\ge 0$.
We also define the dissipation rate $\CD_{\dK,\wN}(t)$ as
\begin{equation}
\CD_{\dK,\wN}(t)
\eqdef
\sum_{1\leq |\al|\leq \dK}\|\pa^\al \solU(t)\|_{L^2_{\spa} \spacen_\wN}^2+\sum_{|\al|+|\be|\leq \dK}\|\pa^\al_\be \{\FI-\FP\}\solU(t)\|_{L^2_{\spa} \spacen_{\wN-|\be|\rho}}^2.
\label{def.dNm}
\end{equation}
Here we also use $\rho =1$ under \eqref{kernelP} and otherwise we use $\rho = - \gamma - 2s >0$ for the soft potentials \eqref{kernelPsing}.   (This $\rho$ is to correct for our change in the definition of the weight \eqref{weigh} from previous papers such as \cite{gsNonCut0,SdecaySOFT}.)   We do not explicitly use these functionals in our proofs herein.  Initially we define
\begin{equation}\label{def.id.jm}
    \SMeps\eqdef\CE_{\dK,\wN}(0).
\end{equation}
We can now state the following existence result, and Lyapunov inequalities:

\begin{theorem}\label{thm.energy}
\cite{gsNonCut0,SdecaySOFT}.
Fix $\wN \ge 0$ and $\solU_0(x,\vel)$.  
There are $\CE_{\dK,\wN}(t)$, $\CD_{\dK,\wN}(t)$ such that if $\SMeps$ is
sufficiently small, then the Cauchy problem to
the Boltzmann equation \eqref{Boltz} admits a unique global solution
$\solU(t,x,\vel)$ satisfying the Lyapunov inequality
\begin{eqnarray}
\frac{d}{dt}\CE_{\dK,\wN}(t)+\la \CD_{\dK,\wN}(t)\leq 0,
\quad \forall t\geq 0.
\label{thm.energy.1}
\end{eqnarray}
Here $\la >0$ may depend on $\wN$. 

In particular, we note that, for all $t \geq 0$:
$
\CE_{\dK,\wN}(t) \leq \SMeps.
$
\end{theorem}

This Theorem \ref{thm.energy} is the building block for our decay results stated earlier in Theorem \ref{thm.main.decay} and Corollary \ref{cor.decay.interp}.  In those statements we used a sufficient number of weights, which we now define precisely. 
We define the quantity $\ell_0^d$ by:
\begin{equation}
\label{ell0d}
\ell_0^d
\eqdef 
\left\{
\begin{array}{ccc}
\frac{n}{2} + \lceil \frac{n}{2} \rceil-3, & \text{for $n$ odd}, \\
n+ 2 \lceil \frac{n}{2} \rceil-7, & \text{for $n$ even}.
\end{array}
\right.
\end{equation}
Furthermore, let us define the quantity $M$ by:
\begin{equation}
\label{Mequation}
M \eqdef
\max \left\{ \left( \frac{2 \dK}{\Ndim -2} -1 \right), \ell_0^d \right\}
\end{equation}
 These choices come from \eqref{L1}, \eqref{L2}, \eqref{L3'}, \eqref{L4} and \eqref{L5} respectively.
 
Here and in several places in the rest of the paper, we will use the notation
\begin{equation}
\label{notation.wJt}
\wJt \eqdef \left\lceil \Ndim/2 \right\rceil-1,
\end{equation}
which denotes the largest integer that is strictly less than $\frac{\Ndim}{2}$.   We recall that in general 
$m=\left\lceil a \right\rceil$ is the smallest integer satisfying  $m \ge a$ and $m'=\left\lfloor a \right\rfloor$ is the largest integer satisfying  $m' \le a$.

Now we define the following weight:
\begin{equation} \label{wNz.def}
\wNz
\eqdef 
\left\{
\begin{array}{ccc}
\max\left\{
\frac{\gamma+2s}{2}+1, 2, 2(\gamma+2s) \right\}, & \text{hard potentials: \eqref{kernelP}}, \\
-\frac{(\gamma+2s)}{2}
\max\left\{ \left(M+1\right),(\frac{n}{2}+\dK)\right\},
&  \text{soft potentials: \eqref{kernelPsing}}.
\end{array}
\right.
\end{equation}
Above and below we are using $M$ from \eqref{Mequation}.
For future reference, we will also define the following related weight:
\begin{equation} \label{wNz.def'}
\PwNz
\eqdef 
\left\{
\begin{array}{ccc}
0, & \text{ for the hard potentials: \eqref{kernelP}}, \\
-\frac{( \gamma + 2s)}{2}M, &  \text{for the soft potentials: \eqref{kernelPsing}.}\\
\end{array}
\right.
\end{equation}
We note that in the rest of this article, we will implicitly assume sometimes without mention that $\wN \ge \wNz$.  We further define
\begin{equation} \label{wNz.def.first}
\wNz^1
\eqdef 
\frac{(\gamma+2s)^+}{2}.
\end{equation}
Here we recall the general notation $(a)^+ = \max\{a, 0\}$.

Throughout this paper  we let $C$  denote
some positive (generally large) inessential constant and $\la$ denotes some positive (generally small) inessential constant, where both $C$ and
$\la$ may change values from line to line. 
Furthermore $A \lesssim B$ means $A \le C B$, and 
$A \gtrsim B$ means $B \lesssim A$. If $C$ depends on a parameters $p_1,\ldots,p_j$, we write $A \leq C_{p_1,\ldots,p_j} B$.
In addition,
$A\approx B$ means $A \lesssim B$ and $B \lesssim A$.

\subsection{Organization of the paper}
In Section \ref{sec.decayl} we will prove several non-linear energy estimates for the solutions to the non cut-off Boltzmann equation \eqref{Boltz} from Theorem \ref{thm.energy} which will be used in the proof of Theorem \ref{thm.main.decay} and Corollary \ref{cor.decay.interp}.
After that in Section \ref{sec.decayNL} we give the proof of Theorem \ref{thm.main.decay}.  We also prove Corollary \ref{cor.decay.interp} and Theorem \ref{thm.extra.decay}.  
Then in Section \ref{secAPP:INTERP}, in part because of the exotic nature of some of our spaces \eqref{normdef}, we prove a collection of functional interpolation inequalities in separable Hilbert spaces.
In the first part of Section \ref{secAPP:linear} we prove large time decay rates of the linear Boltzmann equation \eqref{BoltzLIN} in Besov spaces using dyadic time-frequency splittings and a pointwise time-frequency differential inequality from \cite{SdecaySOFT}.
Finally, in the second part of Section \ref{secAPP:linear} we prove faster decay rates for the linearized problem in the hard potential case \eqref{kernelP} under the additional assumption that the initial data is purely microscopic as in \eqref{form.p}. 
Our analysis in this part is based on a precise understanding of the spectrum of the spatial Fourier transform of the linearized operator for frequencies near zero. 

\subsection{Acknowledgements}
V.S. was supported by a Simons Postdoctoral Fellowship.
R.M.S. was partially supported by the NSF grants DMS-0901463, DMS-1200747, and an Alfred P. Sloan Foundation Research Fellowship. The authors would like to thank the referee for a careful reading and for helpful comments.

\section{Non-linear energy estimates}
\label{sec.decayl}

In this section, we will prove some non-linear differential and integral inequalities for the solutions of the Boltzmann equation \eqref{Boltz} in Theorem \ref{thm.energy}. Our strategy will be to use product estimates from \cite{gsNonCut0,SdecaySOFT}, as well as the functional interpolation inequalities from Section \ref{secAPP:INTERP}. The vector-valued functions we study, among others, take values in the non-isotropic Sobolev spaces in $\vel$, which were previously used in \cite{gsNonCut0}. In proving these estimates, one encounters the difficulty that the macroscopic part doesn't appear in the coercivity estimate \eqref{coerc}, and hence these terms have to be taken care of separately. All of these issues are addressed in Sub-section \ref{Derivative estimates}.
Moreover, we can apply the Littlewood-Paley projection operators defined in Section \ref{app.besov.h} to obtain energy estimates for solutions of \eqref{Boltz} in functional Besov spaces. The latter question is studied in Sub-section \ref{Besov estimates}.

\subsection{Derivative estimates}
\label{Derivative estimates}
This sub-section is devoted to proving two energy estimates for solutions to the Boltzmann equation \eqref{Boltz}.
In Proposition \ref{prop:energy} we prove Lyapunov inequalities for the Sobolev norms of fixed order $k\in\{0, 1, \ldots, \dK\}$.  Then in Lemma \ref{lem.inter.f} we prove a differential inequality which includes the macroscopic components \eqref{coef.p.def} of the dissipation \eqref{def.dNm}.  

\begin{proposition}\label{prop:energy}
Suppose that $\SMeps$ from \eqref{def.id.jm} is sufficiently small with $\ell \geq \wNz$ for $\wNz$ as in \eqref{wNz.def}.
Let $\solU(t,\spa,\vel)$ be the solution to the Boltzmann equation \eqref{Boltz} from Theorem \ref{thm.energy}.  
Fix $k\in\{0, 1, \ldots, \dK\}$; then the following inequality holds
\begin{multline*}
\frac{1}{2} \frac{d}{dt} \sum_{|\alpha| = k} \|\pa^\al \solU \|^2_{L^2_{\spa}L^2_{\vel}}
+
\la
\sum_{|\alpha| = k} \|  \{\FI-\FP\}\pa^\al\solU(t) \|^2_{L^2_{x}\spacen}
\\
\lesssim
\SMeps \sum_{\wJ \le |\alpha| \le  \dK } \|\pa^\al \solU \|^2_{L^2_{x} \spacen},
\end{multline*}
where $\wJ \eqdef \min\{k+1, \dK\}$.  
\end{proposition}

\begin{proof}[Proof of Proposition \ref{prop:energy}]
We take $\partial^\alpha$ derivatives of \eqref{Boltz}, multiply by $\partial^\alpha \solU$, integrate over $\threed_\spa\times \threed_{\vel}$, and use \eqref{coerc} to obtain
\begin{multline*}
\frac{1}{2} \frac{d}{dt}  \|\pa^\al \solU \|^2_{L^2_{x,v}}
+
\la
 \|  \{\FI-\FP\}\pa^\al\solU(t) \|^2_{L^2_{x}\spacen}
 \\
\lesssim
\sum_{\al_1 \le \al}
\left| \left(
 \Ga(\pa^{\al_1} \solU,  \pa^{\al-\al_1} \solU), 
\{\FI-\FP\}\pa^\al\solU
\right) \right|.
\end{multline*}
The right hand side in the above is obtained by using that $\FP$ and $\pa^{\al}$ commute, and \eqref{ProjectionGamma}.  In the rest of this proof we will focus on estimating the non-linear term.

We write $f = \FP f + \{\FI-\FP\} f$ and we expand the non-linear term as:
\begin{equation}\notag
\Ga(\solU,\solU) =
\Ga (\FP \solU, \FP \solU)
+
\Ga (\{\FI-\FP\}\solU, \FP \solU)
+
\Ga (\solU, \{\FI-\FP\}\solU).
\end{equation}
We thus expand
$$
\left(
 \Ga(\pa^{\al_1} \solU,  \pa^{\al-\al_1} \solU), 
\{\FI-\FP\}\pa^\al\solU
\right) 
=
\Ac_1 + \Ac_2 + \Ac_3.
$$
where:
\begin{eqnarray}
\notag
\Ac_1 
&\eqdef &
 \left(  \Gamma (\FP \pa^{\al_1}  \solU, \FP \pa^{\al-\al_1}  \solU), \partial^\alpha \{\FI-\FP\}\solU \right),
\\
\label{ga.expand}
\Ac_2 
&\eqdef &
 \left(  \Gamma (\{\FI-\FP\} \pa^{\al_1} \solU, \FP \pa^{\al-\al_1}  \solU), \partial^\alpha \{\FI-\FP\}\solU \right),
\\
\notag
\Ac_3 
&\eqdef &
 \left(
 \Ga(\pa^{\al_1} \solU,  \{\FI-\FP\}\pa^{\al-\al_1} \solU), 
\{\FI-\FP\}\pa^\al\solU
\right).
\end{eqnarray}
We will estimate each of these terms individually.

The desired estimate for the term $\Ac_3$ then follows from 
\cite[Eq (6.6)]{gsNonCut0}.  In other words, using \cite[Eq (6.6)]{gsNonCut0} for any small $\delta>0$ we have
\begin{multline}\label{nl.term3}
\left| \Ac_3 \right|
=
\left| \left(
 \Ga(\pa^{\al_1} \solU,  \{\FI-\FP\}\pa^{\al-\al_1} \solU), 
\{\FI-\FP\}\pa^\al\solU
\right) \right|
 \\
\lesssim
\int_{\threed} dx ~ 
 | \pa^{\al_1} \solU  |_{L^2_{\vel}}
 \nsm \{\FI-\FP\}\pa^{\al-\al_1} \solU\nsm_{\spacen}
 \nsm\{\FI-\FP\}\pa^\al\solU \nsm_{\spacen}
  \\
\lesssim
C_\delta
\left\| 
 | \pa^{\al_1} \solU  |_{L^2_{\vel}}
 \nsm \{\FI-\FP\}\pa^{\al-\al_1} \solU \nsm_{\spacen} \right\|_{L^2_{\spa}}^2
 +
\delta
 \|\{\FI-\FP\}\pa^\al\solU \|_{L^2_{\spa}\spacen}^2.
\end{multline}
For the first term and the second term in \eqref{ga.expand}, we notice from \eqref{form.p} that
$$
\Ga (\solU, \FP \solU)= \sum_{i=1}^{\Ndim +2} \psi_i(t,x) \Ga (\solU, \chi_i),
$$
where the $\psi_i(t,x)$ are the elements from \eqref{coef.p.def} and the $\chi_i(v)$ are the smooth rapidly decaying velocity basis vectors in \eqref{nullLsp}.  Thus from \cite[Proposition 6.1]{gsNonCut0}:
\begin{multline}\label{nl.term2}
\left| \Ac_2 \right|
=
\left| \left(  \Gamma (\{\FI-\FP\} \pa^{\al_1} \solU, \FP \pa^{\al-\al_1}  \solU), \partial^\alpha \{\FI-\FP\}\solU \right) \right|
\\
\lesssim
\int_{\threed}dx ~ 
  \nsm  \partial^{{\al_1}} \{\FI-\FP\} \solU\nsm_{L^2_{\gamma+2s-(\Ndim-1)}}
  |\pa^{\al-\al_1} [a,b,c]| ~
  \nsm   \{\FI-\FP\}\pa^{\al} \solU\nsm_{L^2_{\gamma+2s-(\Ndim-1)}}
  \\
\lesssim
C_\delta
\left\| 
  \nsm  \partial^{{\al_1}} \{\FI-\FP\} \solU\nsm_{L^2_{\gamma+2s}}
  |\pa^{\al-\al_1} [a,b,c]| 
  \right\|_{L^2_{\spa}}^2
  +
\delta
 \|  \{\FI-\FP\}\pa^{\al} \solU\|_{L^2_{\spa}\spacen}^2.
\end{multline}
Above the first inequality is a statement of \cite[Proposition 6.1]{gsNonCut0}.  Then in the second inequality we discard some of the velocity decay weights because they will not be used.  Also 
 here $|[a,b,c]|$ is the Euclidean absolute value norm of the coefficients from \eqref{coef.p.def}.    The last case to consider is $\Ac_1$.  From \cite[Proposition 6.1]{gsNonCut0} again
 \begin{multline}\label{nl.term1}
\left| \Ac_1 \right|
=
\left| \left(  \Gamma (\FP \pa^{\al_1}  \solU, \FP \pa^{\al-\al_1}  \solU), \partial^\alpha \{\FI-\FP\}\solU \right) \right|
\\
\lesssim
\int_{\threed}dx ~ 
|\pa^{\al_1} [a,b,c]| ~|\pa^{\al-\al_1} [a,b,c]| ~
\nsm  \partial^\alpha \{\FI-\FP\}\solU\nsm_{L^2_{\gamma+2s-(\Ndim-1)}}
  \\
\lesssim
C_\delta
\left\| 
~  |\pa^{\al_1} [a,b,c]| ~
  |\pa^{\al-\al_1} [a,b,c]| ~
  \right\|_{L^2_{\spa}}^2
  +
\delta
 \|  \{\FI-\FP\}\pa^{\al} \solU\|_{L^2_{\spa}\spacen}^2.
\end{multline}
Collecting these estimates in \eqref{nl.term3}, \eqref{nl.term2}, \eqref{nl.term1}, and summing over $|\alpha| = k$ we obtain
$$
\frac{1}{2} \frac{d}{dt} \sum_{|\alpha| = k} \|\pa^\al \solU \|^2_{L^2_{x,v}}
+
\la
\sum_{|\alpha| = k} \|  \{\FI-\FP\}\pa^\al\solU(t) \|^2_{L^2_{x}\spacen}
\lesssim
 B .
$$
Here $B = B_1 + B_2 + B_3$ with
\begin{eqnarray}
\notag
B_1 &\eqdef  &
\sum_{|\alpha| = k}
\sum_{\al_1 \le \al}
\left\| 
~  |\pa^{\al_1} [a,b,c]| ~
  |\pa^{\al-\al_1} [a,b,c]| ~
  \right\|_{L^2_{\spa}}^2,
  \\
  \label{Bdef}
  B_2 &\eqdef &
  \sum_{|\alpha| = k}
\sum_{\al_1 \le \al}
  \left\|   \nsm  \partial^{{\al_1}} \{\FI-\FP\} \solU\nsm_{L^2_{\gamma+2s}}  |\pa^{\al-\al_1} [a,b,c]|   \right\|_{L^2_{\spa}}^2,
    \\
    \notag
  B_3 &\eqdef &
  \sum_{|\alpha| = k}
\sum_{\al_1 \le \al}
\left\|  \nsm \pa^{\al_1} \solU  \nsm_{L^2_{\vel}} \nsm \{\FI-\FP\}\pa^{\al-\al_1} \solU\nsm_{\spacen} \right\|_{L^2_{\spa}}^2.
\end{eqnarray}
Thus we have reduced the proof of Proposition \ref{prop:energy} to proving that 
\begin{equation}\label{prop.e.rest}
 B 
\lesssim
\SMeps \sum_{\wJ \le |\alpha| \le  \dK } \|\pa^\al \solU \|^2_{L^2_{x} \spacen}.
\end{equation}
We will prove \eqref{prop.e.rest} for each of the terms in \eqref{Bdef} individually. 

In order to bound $B_1$ and $B_2$, we will use interpolation. 
Suppose that $\hilbertONE$ is a separable Hilbert space  of functions in $v$ as in Section \ref{sec.notation}.
Furthermore, suppose that $|\al | = k \in \{0, 1, \ldots, \dK\}$, $\al_1 \le \al$. 
Then, there exist $p, q\ge 2$ satisfying $\frac{1}{p}+\frac{1}{q} = \frac{1}{2}$ such that we have:
\begin{equation}  \label{lem.app.interp.sec}
\| \partial^{\al-\al_1}  f \|_{ L^q_{\spa}  \hilbertONE}
\| \partial^{\al_1}   f \|_{ L^p_{\spa}  \hilbertONE}
\lesssim
\|  f \|_{H^{\ksob}_{\spa} \hilbertONE}
\sum_{|\alpha'| = \min\{k+1, \dK\} } \|\pa^{\al'} \solU \|_{L^2_{\spa} \hilbertONE}.
\end{equation}
The inequality \eqref{lem.app.interp.sec} follows from the results of Lemma \ref{partialderivativesLp} and Lemma \ref{lem.app.interp}, all of which are proved in Section \ref{secAPP:INTERP}.  

We begin by looking at the term 
$B_1$ from \eqref{Bdef}.  We note that by \eqref{coef.p.def}, by the exponential decay in $v$ of $\sqrt{\FM}$, and by the Cauchy-Schwarz inequality:
\begin{equation}
\label{macro.point.p.est}
|\pa^{\al_1} [a,b,c]| \lesssim 
| w^{-\wE } \pa^{\al_1} f |_{L^2_{\vel}},
\end{equation}
which holds for all $\wE$.
We can furthermore deduce that
\begin{equation}
\label{macro.point.estimate}
|\FP \pa^{\al_1} f|_{L^2_{\gamma+2s}} \lesssim |\pa^{\al_1} f|_{L^2_{\gamma+2s}}.
\end{equation}
Let us take $\hilbertONE=L^2_{\gamma+2s}$ in 
\eqref{lem.app.interp.sec}. Then, for the $p$, $q$ which one obtains, 
we can use H\"{o}lder's inequality in $x$ to estimate: 
 \begin{equation}\label{nl.term1.e1}
\left\| 
~  |\pa^{\al_1} [a,b,c]| ~
  |\pa^{\al-\al_1} [a,b,c]| ~
  \right\|_{L^2_{\spa}}^2 
\lesssim
  \|   \pa^{\al_1} f \|_{L^q_{\spa} \spaceL}^2
  \|   \pa^{\al-\al_1} f \|_{L^p_{\spa} \spaceL}^2.
\end{equation}
The fact that $B_1$ is bounded by the right-hand side in \eqref{prop.e.rest} follows directly from \eqref{nl.term1.e1} and \eqref{lem.app.interp.sec} with $\hilbertONE=L^2_{\gamma+2s}$ if we use Theorem \ref{thm.energy} (to deduce the fact that it suffices to estimate everything at $t=0$) and \eqref{def.id.jm}. More precisely, we observe that, by construction:
$$
\|f(t)\|^2_{H^{\ksob}_x L^2_{\gamma+2s}} \approx \sum_{|\alpha| \leq \ksob} \|w^{\frac{\gamma+2s}{2}}\pa^{\al} f(t)\|_{L^2_x L^2_v}^2 \lesssim \CE_{\dK,\ell}(t)
\lesssim
\epsilon_{\dK,\ell},
$$ 
which holds since $\ell \geq \frac{\gamma+2s}{2}$ in the hard potential case, and since $\ell \geq 0$ in the soft potential case by our choice of $\ell$, and because of Theorem \ref{thm.energy} and \eqref{def.id.jm}.

For the term $B_2$ in \eqref{Bdef}, we argue similarly.
We let $p$, $q$ be as in the bound for $B_1$.  
By H\"{o}lder's inequality in $x$, we have: 
\begin{multline}\label{nl.term2.e1}
  \left\|   \nsm  \partial^{{\al_1}} \{\FI-\FP\} \solU\nsm_{L^2_{\gamma+2s}}  |\pa^{\al-\al_1} [a,b,c]|   \right\|_{L^2_{\spa}}^2
  \\
\lesssim
  \|  \pa^{\al_1} f \|_{L^q_{\spa} \spaceL}^2
  \|  \pa^{\al-\al_1} f \|_{L^p_{\spa}\spaceL}^2.
\end{multline}
Here we used  that $\{\FI-\FP\}$ is a bounded linear operator on $L^2_{\gamma+2s}$ which commutes with $\pa^{\al_1}$. The boundedness property follows from the fact that $\FP$ is bounded on $L^2_{\gamma+2s}$ by \eqref{macro.point.estimate}. 
Then, exactly as before, \eqref{prop.e.rest} follows from \eqref{nl.term2.e1} and \eqref{lem.app.interp.sec} using also Theorem \ref{thm.energy} and \eqref{def.id.jm}.  
 
For the last term $B_3$ in \eqref{Bdef}, also when $\frac{1}{p}+\frac{1}{q} = \frac{1}{2}$, we have
  \begin{equation}\label{nl.term3.e1}
\left\|  \nsm \pa^{\al_1} \solU  \nsm_{L^2_{\vel}} \nsm \{\FI-\FP\}\pa^{\al-\al_1} \solU\nsm_{\spacen} \right\|_{L^2_{\spa}}^2
  \\
\lesssim
  \|  \pa^{\al_1} f \|_{L^q_{\spa} L^2_{\vel}}^2
  \|  \pa^{\al-\al_1} f \|_{L^p_{\spa} \spacen}^2.
\end{equation}
In this case, for the upper bound in \eqref{nl.term3.e1} we note that there exist $p,q \geq 2$ such that $\frac{1}{p}+\frac{1}{q}=\frac{1}{2}$ for which the following estimate holds: 
\begin{multline}\label{goal.prove.last}
    \|  \pa^{\al_1} f \|_{L^q_{\spa} L^2_{\vel}}
  \|  \pa^{\al-\al_1} f \|_{L^p_{\spa} \spacen}
  \\
\lesssim
\left( \|  f \|_{H^{\ksob}_{\spa} \spacen_{\PwNz}}+ \| w^{\PwNz}  f \|_{H^{\dK}_{\spa} L^2_{\vel}}\right)
\sum_{\wJ \le |\al'| \le \dK} \| \pa^{\al'}    \solU \|_{L^2_{\spa} \spacen}.
\end{multline}
Again $\wJ \eqdef\min\{k+1, \dK\}$ and $\PwNz$ is from \eqref{wNz.def'}.
Inequality \eqref{goal.prove.last} follows from Lemma \ref{partialderivativesLp} and Lemma \ref{last.lem.app}, which are proved in Section \ref{secAPP:INTERP}.

Using the fact that $\dK \geq 2 \ksob \geq \ksob + 2$, \eqref{isocomp}, and that $s \in (0,1)$, we obtain 
  \begin{equation}\label{boundepsilon}
\|  f \|_{H^{\ksob}_{\spa} \spacen_{\PwNz}}^2
\lesssim
\|  f \|_{H^{\dK-2}_{\spa} \spacen_{\PwNz}}^2 
\lesssim
\| w^{\PwNz} f \|_{H^{\dK-2}_\spa H^1_{\gamma+2s}}^2.
\end{equation}
Notice that we have dispensed with the geometric content of $\spacen_{\PwNz}$ because  it is 
our goal to show that the above expression is $\lesssim \SMeps$, which is an isotropic norm.   
In order to do this, we will use our assumption 
that $\ell \geq \wNz$ with $\wNz$ in \eqref{wNz.def} and consider the hard potential case \eqref{kernelP} and soft potential case \eqref{kernelPsing} separately.

In the hard potential case \eqref{kernelP},  recall that $\rho=1$ in \eqref{def.eNm} and $\PwNz = 0$ in \eqref{wNz.def'}. Using  $\CE_{\dK,\ell}(t)$ and $\epsilon_{\dK,\ell}$ from \eqref{def.eNm} and \eqref{def.id.jm} and Theorem \ref{thm.energy} we have
$$
\| f\|^2_{H^{\dK-2}_x H^1_{\gamma+2s}}
\lesssim
 \sum_{|\alpha|\le\dK-2, |\beta|\le 1} \|w^{\frac{1}{2}(\gamma+2s)}\partial^{\alpha}_{\beta}f\|_{L^2_x L^2_v}^2
 \lesssim
\CE_{\dK,\ell}(t)
 \lesssim
\epsilon_{\dK,\ell},
$$
provided that $\ell -1\geq  \frac{1}{2}(\gamma+2s)$,
which holds since $\ell \ge \wNz$ using \eqref{wNz.def}.  The desired bound on $B_3$ from \eqref{prop.e.rest} then follows from \eqref{nl.term3.e1}, \eqref{goal.prove.last}, and this analysis.

In the soft potential case \eqref{kernelPsing}, we recall $\rho=-\gamma-2s$ and we note that:
\begin{multline}\notag
\|w^{\PwNz}f\|^2_{H^{\dK-2}_x H^1_{\gamma+2s}}
\approx \sum_{|\alpha|\le\dK-2, |\beta|\le 1} \|w^{\PwNz+\frac{1}{2}(\gamma+2s)}\partial^{\alpha}_{\beta}f\|_{L^2_x L^2_v}^2
\\
 \lesssim
\sum_{|\beta|\le 1, |\alpha|\le \dK-2}\|w^{\ell-|\beta|\rho}\partial^{\alpha}_{\beta}f\|_{L^2_x L^2_v}^2 
 \lesssim
\CE_{\dK,\ell}(t)
 \lesssim
\epsilon_{\dK,\ell},
\end{multline}
whenever $\ell - \rho \geq \PwNz - \frac{\rho}{2}$, or equivalently that  $\ell \geq  \PwNz-\frac{1}{2}(\gamma+2s)$, which holds since $\wNz \geq  \PwNz-\frac{1}{2}(\gamma+2s)$. 
We thus deduce again that $B_3$ satisfies the bound in \eqref{prop.e.rest}.
\end{proof}

We recall that the proof of Proposition \ref{prop:energy} relied on the use of \eqref{coerc} and hence we only obtained the microscopic terms $\|\{\FI-\FP\}\pa^{\al}f(t)\|_{L^2_xN^{s,\gamma}}^2$ on the left-hand side of the inequality. In order to control the macroscopic terms we use an interaction functional approach to prove the following bound:

\begin{lemma}\label{lem.inter.f}
Under the conditions from Theorem \ref{thm.energy}, there exists continuous functionals $\CI^k(t)$, 
for any $k\in\{0, 1, \ldots, \dK -1\}$, such that 
\begin{equation}\label{inter.diff}
-\frac{d \CI^k}{dt} + \la \sum_{|\alpha| = k+1}  \left\| \partial^\alpha [a,b,c] \right\|^2_{L^2_{\spa}} 
\lesssim
\sum_{k \le |\alpha| \le k+1} \| \{\FI - \FP \} \partial^\alpha \solU \|_{L^2_{\spa} L^2_{\gamma+2s}}^2.
\end{equation}
The functional $\CI^k(t)$ furthermore satisfies for any $m\ge 0$ the uniform estimates
\begin{equation}\label{functional.est}
\left|  \CI^k(t) \right| 
\lesssim
\sum_{k \le |\alpha| \le k+1} \| w^{-m} \partial^\alpha \solU (t) \|_{L^2_{\spa}L^2_{\vel}}^2.
\end{equation}
Additionally, we will give the precise definition of $\CI^k(t)$ below.
\end{lemma}

This lemma was essentially already proven in \cite{gsNonCut0}, except that some of the estimates therein were too crude as written in the statements of the theorems and lemmas.  This type of estimate is well known, and we refer to for example \cite{MR2420519,gsNonCut0,MR2000470,MR2095473,GuoWang2011,MR2259206,SdecaySOFT,ZhuStrain,ZhuStrain2} for previous developments.  
We will explain carefully the main differences between this estimate, and what is done in \cite{gsNonCut0}.

\begin{proof}[Proof of Lemma \ref{lem.inter.f}]
We define $\CI^k(t)$ as follows
$$
\CI^k(t) = \sum_{|\alpha| =k} \left\{
\mathcal{I}_a^\alpha(t)
+
\mathcal{I}_b^\alpha(t)
+
\mathcal{I}_c^\alpha(t)\right\},
$$
where each of the functionals above are defined as
\begin{eqnarray*}
\mathcal{I}_a^\alpha(t) &\eqdef &
\int_{\domain}  ~ dx ~ \left( \nabla \cdot \partial ^\alpha b \right)\partial^\alpha a(t,x)
+
\sum_{i=1}^\Ndim
\int_{\domain}  ~ dx ~ \partial_{x_i}\partial ^\alpha r_{bi} ~
\partial^\alpha a(t,x),
\\
\mathcal{I}_b^\alpha(t) &\eqdef &
- 
 \sum_{j\neq i} \int_{\domain} ~ dx ~ \partial_{x_j} \partial^\alpha r_{ij}\partial^\alpha b_i,
\\
\mathcal{I}_c^\alpha(t) &\eqdef &
- \int_{\domain} ~ dx ~ \partial ^\alpha r_c(t,x) ~ \cdot \nabla_{\spa} \partial ^\alpha c(t,x).
\end{eqnarray*}
Here the $r_{bi}, r_{ij}, r_c$ are all of the form 
$
 \langle \{{\bf I-P}\}f, e_k \rangle
$
where each $e_k$ is some fixed linear combination of the following basis:
\begin{gather}
\left( v_i|v|^2\sqrt{\mu } \right)_{1\le i\le n}, ~ 
\left(v_i^2\sqrt{\mu } \right)_{1\le i\le n},  ~
\left(v_iv_j\sqrt{\mu } \right)_{1\le i<j\le n}, ~
\left(v_i\sqrt{\mu } \right)_{1\le i\le n}, ~
\sqrt{\mu }.  
\label{base}
\end{gather}
Notice that \eqref{functional.est} follows directly from the definition of $\CI^k(t)$.

Our goal is then to establish \eqref{inter.diff}.  Now following the proof of \cite[Theorem 8.4]{gsNonCut0} 
we directly obtain
\begin{multline}\label{old.gs.est}
-\frac{d \CI^k}{dt} + \la \sum_{|\alpha| = k+1}  \left\| \partial^\alpha [a,b,c] \right\|^2_{L^2_{\spa}} 
\lesssim
\sum_{k \le |\alpha| \le k+1} \| \{\FI - \FP \} \partial^\alpha \solU \|_{L^2_{\spa} L^2_{\gamma+2s}}^2
\\
+
\sum_{|\alpha| = k} 
\sum_{j}
\left\| \langle \partial^\alpha \Gamma(f,f), e_j\rangle\right\|_{L^2_{\spa}}^2.
\end{multline}
Above (and below)   each ``$e_j$'' is a fixed linear combination of the basis \eqref{base}, and $\sum_{j}$ is a finite sum.
We give a detailed explanation of how to establish \eqref{old.gs.est} from the proof of \cite[Theorem 8.4]{gsNonCut0}.\footnote{The statement of \cite[Theorem 8.4]{gsNonCut0} is too crude to establish \eqref{old.gs.est} because it does not keep track of the order of derivatives in each term (since that was not needed in \cite{gsNonCut0}).}  Particularly the proof of inequality 
\cite[(8.21)]{gsNonCut0} without using \cite[Lemmas 8.6 \& 8.7]{gsNonCut0} actually yields
\begin{gather}
\notag
-\frac{d \mathcal{I}_a^\alpha  }{dt}
+
\|\nabla \partial ^\alpha a\|^2_{L^2_x}
-
\eta \| \nabla \cdot \partial ^\alpha b \|^2_{L^2_x} 
\lesssim
\left\| \langle  l(\{{\bf I-P}\}\partial^{\alpha} f), e\rangle\right\|_{L^2_{\spa}}^2.
\\
+
C_\eta
\|\{{\bf I-P}\} \nabla_x  \partial^\alpha f\|^2_{L^2_{\gamma+2s}}
+
\left\| \langle \partial^{\alpha} \Gamma(f,f), e\rangle\right\|_{L^2_{\spa}}^2,
\label{mainAest}
\end{gather}
which holds for any small $\eta>0$.  The ``$e$'' above (and below) is some fixed linear combination of the basis \eqref{base}. Here
\begin{equation}
l(\{{\bf I-P}\}f) \eqdef -\{v\cdot \nabla _x+L\}\{{\bf I-P}\}f.  \label{l}
\end{equation}
Similarly the proof of inequalities
\cite[(8.22)]{gsNonCut0} and \cite[(8.24)]{gsNonCut0} without using \cite[Lemma's 8.6 \& 8.7]{gsNonCut0} more precisely yield for any small $\eta>0$ that
\begin{gather}
\notag
-
\frac{d \mathcal{I}_c^\alpha  }{dt}
+
\|\nabla \partial ^\alpha c\|^2_{L^2_x}
-
\eta \| \nabla \cdot \partial ^\alpha b \|^2_{L^2_x} 
\lesssim
\left\| \langle  l(\{{\bf I-P}\}\partial^{\alpha} f), e\rangle\right\|_{L^2_{\spa}}^2.
\\
+
C_\eta
\|\{{\bf I-P}\} \nabla_x  \partial^\alpha f\|^2_{L^2_{\gamma+2s}}
+
\left\| \langle \partial^{\alpha} \Gamma(f,f), e\rangle\right\|_{L^2_{\spa}}^2,
\label{mainCest}
\end{gather}
and further
\begin{gather}
\notag
-\frac{d \mathcal{I}_b^\alpha  }{dt}
+
\|\nabla \partial ^\alpha b\|^2_{L^2_x}
-
 \eta \left\{ \| \nabla  \partial ^\alpha a \|^2_{L^2_x} + \| \nabla  \partial ^\alpha c \|^2_{L^2_x}  \right\}
\lesssim
\sum_{j} 
\left\| \langle  l(\{{\bf I-P}\}\partial^{\alpha} f), e_j\rangle\right\|_{L^2_{\spa}}^2.
\\
+
C_\eta
\|\{{\bf I-P}\} \nabla_x  \partial^\alpha f\|^2_{L^2_{\gamma+2s}}
+
\sum_{j} 
\left\| \langle \partial^{\alpha} \Gamma(f,f), e_j\rangle\right\|_{L^2_{\spa}}^2.
\label{mainBest}
\end{gather}
Then following the proof of \cite[Lemma 8.6]{gsNonCut0} by combining in particular \eqref{l} with \cite[(6.12)]{gsNonCut0} we conclude that
$$
\sum_{|\alpha| =k} 
\left\| \langle  l(\{{\bf I-P}\}\partial^{\alpha} f), e\rangle\right\|_{L^2_{\spa}}^2
\lesssim
\sum_{k \le |\alpha| \le k+1} \| \{\FI - \FP \} \partial^\alpha \solU \|_{L^2_{\spa} L^2_{\gamma+2s}}^2
$$
Adding together \eqref{mainAest}, \eqref{mainCest},  and \eqref{mainBest} and summing over ${|\alpha| =k}$, and using 
the above inequality, yields  
\eqref{old.gs.est}.

We then claim that we have the following estimate 
\begin{multline}\label{est.nl.interact}
\sum_{|\al| = k} 
\left\| \langle \partial^\alpha \Gamma(f,f), e\rangle\right\|_{L^2_{\spa}}^2 
\\
\lesssim 
\SMeps
\sum_{|\alpha'| = k+1}\left(
  \left\| \partial^{\alpha'} [a,b,c] \right\|_{L^2_{\spa}}^2  
+
 \| \{\FI - \FP \} \partial^{\alpha'} \solU \|_{L^2_{\spa} L^2_{\gamma+2s}}^2
\right). 
\end{multline}
Now \eqref{old.gs.est} combined with \eqref{est.nl.interact} establishes \eqref{inter.diff} since $\SMeps$ is sufficiently small. Thus our goal is reduced to establishing \eqref{est.nl.interact}.

We now use equation (6.12) of \cite[Proposition 6.1]{gsNonCut0} to obtain that 
$$
\sum_{|\al| = k} 
\left\| \langle \partial^\alpha \Gamma(f,f), e\rangle\right\|_{L^2_{\spa}}
\lesssim
\sum_{|\al| = k} \sum_{\al_1 \le \al}
\left\| ~ \nsm \partial^{\al_1} f \nsm_{L^2_{-m}} ~ \nsm \partial^{\al - \al_1} f \nsm_{L^2_{-m}} ~ \right\|_{L^2_{\spa}},
$$
which holds for any large $m \ge 0$. Here, we are also using the fact that $e$ satisfies the property (4.1) in \cite{gsNonCut0} which we need in order to apply 
\cite[Proposition 6.1]{gsNonCut0}. More precisely, we are using the fact that $|e| \lesssim \exp(-\lambda |v|^2)$, which follows since $e$ is a linear combination of functions satisfying this bound.

Then for all $p,q \geq 2$ with $\frac{1}{p}+\frac{1}{q} = \frac{1}{2}$ we have
$$
\left\| ~ \nsm \partial^{\al_1} f \nsm_{L^2_{-m}} ~ \nsm \partial^{\al - \al_1} f \nsm_{L^2_{-m}} ~ \right\|_{L^2_{\spa}}
\lesssim
\left\|  \partial^{\al_1} f  \right\|_{L^q_{\spa} \spaceL}
\left\|   \partial^{\al - \al_1} f \right\|_{L^p_{\spa} \spaceL}.
$$
In the hard potential case, we take $m \geq 0$ and in the soft potential case, we take $m \geq \frac{-\gamma-2s}{2}$.
We do a micro-macro decomposition, as in \eqref{form.p}, to see that 
\begin{multline}\label{terms.n.est}
\left\|  \partial^{\al_1} f  \right\|_{L^q_{\spa} \spaceL}
\left\|  \partial^{\al - \al_1} f \right\|_{L^p_{\spa} \spaceL}
\\
\lesssim
\left\|  \partial^{\al_1} [a,b,c]  \right\|_{L^q_{\spa}}
\left\|  \partial^{\al - \al_1} [a,b,c]  \right\|_{L^p_{\spa} }
\\
+
\left\|  \partial^{\al_1} [a,b,c]  \right\|_{L^q_{\spa}}
\left\|   \partial^{\al - \al_1} \{\FI - \FP \}f \right\|_{L^p_{\spa} \spaceL}
\\
+
\left\|  \partial^{\al_1} \{\FI - \FP \}f  \right\|_{L^q_{\spa} \spaceL}
\left\|  \partial^{\al - \al_1} [a,b,c]  \right\|_{L^p_{\spa} }
\\
+
\left\|  \partial^{\al_1} \{\FI - \FP \}f  \right\|_{L^q_{\spa} \spaceL}
\left\|  \partial^{\al - \al_1} \{\FI - \FP \}f \right\|_{L^p_{\spa} \spaceL}.
\end{multline}
To estimate these terms we will use \eqref{lem.app.interp.B} just below.

For any $k\in \{0, 1, \ldots, \dK-1\}$  such that ${|\al| = k}$ and $\al_1 \le \al$ there exists $p, q\ge 2$ satisfying $\frac{1}{p}+\frac{1}{q} = \frac{1}{2}$ such that we have
\begin{multline}\label{lem.app.interp.B}
\|  \partial^{\al - \al_1} \solU \|_{ L^q_{\spa}  \hilbertONE}
\|  \partial^{\al_1}   g \|_{ L^p_{\spa}  \hilbertTWO}
\lesssim
\|  f \|_{H^{\dK-2}_{\spa} \hilbertONE}
\sum_{|\alpha'| = k+1 } \|\pa^{\al'} g \|_{L^2_{\spa} \hilbertTWO}
\\
+
\|  g \|_{H^{\dK-2}_{\spa} \hilbertTWO}
\sum_{|\alpha'| = k+1 } \|\pa^{\al'} \solU \|_{L^2_{\spa} \hilbertONE}.
\end{multline}
This holds for any separable Hilbert spaces $\hilbertONE$ and $\hilbertTWO$ such as those from Section \ref{sec.notation}. The bound \eqref{lem.app.interp.B} follows from Lemma \ref{partialderivativesLp} and Lemma \ref{lem.app.interp}, which are proved in Section \ref{secAPP:INTERP}. We are also using the fact that $\dK -2 \geq \ksob$. Notice further that, in \eqref{lem.app.interp.B}, $\hilbertONE$ and $\hilbertTWO$ could be $\mathbb{R}$ with norm given by absolute values, as in for instance $\left\|  \partial^{\al_1} [a,b,c]  \right\|_{L^q_{\spa}}$. 
To finish this off, we notice that all of the terms in the upper bounds of \eqref{terms.n.est} can be bounded above by the norms in the lower bound of the inequality of \eqref{lem.app.interp.B} with $\|\cdot\|_{\hilbertONE}$ and 
$\|\cdot\|_{\hilbertTWO}$ either given by absolute values, $\nsm \cdot \nsm$, or by $|  ~ \cdot  ~ |_{ \spaceL}$.  Thus using \eqref{lem.app.interp.B} with the appropriate $p$, $q$ as in \eqref{lem.app.interp.B}, and using Theorem \ref{thm.energy}, in \eqref{terms.n.est} we notice that \eqref{est.nl.interact} holds true. More precisely, by Theorem \ref{thm.energy}, we obtain uniform bounds on the full instant energy functional $\CE_{\dK,\wN}(t)$ for all non-negative times. Here, we also use the fact that $\ell \geq \frac{\gamma+2s}{2}$ by \eqref{wNz.def}.
\end{proof}


\subsection{Estimates in the homogeneous Besov space}
\label{Besov estimates}
In this sub-section, we assume that the initial data $f_0$ is sufficiently regular and we prove the following integral inequality for the functional Besov norms of the solution $f$ to \eqref{Boltz}.

\begin{proposition}\label{prop.besov.ev}
Consider $f(t,x,v)$, the solution to the Boltzmann equation obtained in Theorem \ref{thm.energy}, with initial data $f_0(x,v)$ satisfying $\| f_0 \|_{\dot{B}_2^{-\ALTsig,\infty}L^2_{\vel}} < \infty$ with $\ALTsig \in (0, \Ndim/2]$.  Let $\SMeps$ from \eqref{def.id.jm} be small with $\wN \ge \wNz^1$ with $\wNz^1$ from \eqref{wNz.def.first}.  Then
\begin{multline}\label{besov.prop.small}
\| f(t) \|_{\dot{B}_2^{-\ALTsig,\infty}L^2_{\vel}}^2
\lesssim 
\| f_0 \|_{\dot{B}_2^{-\ALTsig,\infty}L^2_{\vel}}^2
  \\
+
\SMeps \int_0^t ds 
\sum_{|\beta| \le 2} \| \partial_\beta \{ \FI - \FP\} f (s) \|_{L^2_{\spa}L^2_{\gamma+2s}}^2
\\
+
 \int_0^t ds 
\left\|    \FP \solU (s)\right\|_{L^{2}_{\spa}L^2_{\vel}}^{2}
\sum_{|\alpha| =\lfloor \frac{\Ndim}{2} -\ALTsig \rfloor }\left\| \partial^\alpha   \FP \solU(s)\right\|_{L^{2}_{\spa}L^2_{\vel}}^{2(1-\theta)}
\sum_{|\alpha'| =\lfloor \frac{\Ndim}{2} -\ALTsig+1 \rfloor}\left\| \partial^{\alpha'}   \FP \solU(s)\right\|_{L^{2}_{\spa}L^2_{\vel}}^{2\theta},
\end{multline}
where $\theta = \frac{\Ndim}{2} -\ALTsig -\lfloor \frac{\Ndim}{2} -\ALTsig \rfloor\in [0,1)$.
\end{proposition}
We note that the integrand in the first integral is related to the dissipation rate \eqref{def.dNm}, whose time integral we know is finite by Theorem \ref{thm.energy}. This will be a crucial observation in the following section, when we prove uniform a priori bounds on the macroscopic part in a functional Besov space. The precise bound is given in \eqref{est.unif.besov}. 

\begin{proof}[Proof of Proposition \ref{prop.besov.ev}.]
The operators  $\Delta_j$ are defined in Section \ref{app.besov.h}.
We take $\Delta_j$ of \eqref{Boltz}, multiply \eqref{Boltz} by $\Delta_j f$, integrate over $\threed_{\spa} \times \threed_{\vel}$ and use \eqref{coerc} to obtain
$$
\frac{1}{2} \frac{d}{dt} \| \Delta_j f\|_{L^2_{x,v}}^2
+
\la \| \{ \FI - \FP \} \Delta_j f \|_{L^2_{\spa} \spacen}^2 
\lesssim
\left| \left( 
\Delta_j \Ga (f, f), 
\{ \FI - \FP \} \Delta_j f 
\right) \right|. 
$$
Here, we used \eqref{ProjectionGamma} and the fact that $\FP$ and $\Delta_j$ commute.  We estimate the upper bound directly as: 
$$
\left| \left( 
\Delta_j \Ga (f, f), 
\{ \FI - \FP \} \Delta_j f 
\right) \right|
\le
C_\la \| w^{\frac{-\gamma - 2s}{2}} \Delta_j \Ga (f, f)\|_{L^2_{\spa}L^2_{\vel}}^2
+
\frac{\la}{2}
\|\{ \FI - \FP \} \Delta_j f \|_{L^2_{\spa} \spaceL}^2.
$$
As in \eqref{app.bernstein.XV}, using the Bernstein inequality, one obtains
$$
 \| w^{\frac{-\gamma - 2s}{2}} \Delta_j \Ga (f, f)\|_{L^2_{x}L^2_{v}}^2
\lesssim
2^{2\ALTsig j} \| w^{\frac{-\gamma - 2s}{2}} \Ga (f, f)\|_{L^p_{x}L^2_{v}}^2,
$$
where $\ALTsig = \frac{\Ndim}{p} - \frac{\Ndim}{2}$ for $p \in [1, 2]$.  Equivalently 
$p = \frac{2 \Ndim}{\Ndim + 2\ALTsig}$ for $\ALTsig \in [0,  \frac{\Ndim}{2}]$.
We need to estimate $\| w^{\frac{-\gamma - 2s}{2}} \Ga (f, f)\|_{L^p_{x}L^2_{v}}^2$.
Let us use the estimate from \cite[Proposition 3.1, Eq (3.20)]{SdecaySOFT} to obtain
\begin{equation}\label{upper.nl.est}
 \| w^{\frac{-\gamma - 2s}{2}} \Ga (f, f)\|_{L^p_{x}L^2_{v}}^2
\lesssim
\left\|  ~ \left\nsm w^{\frac{(\gamma+2s)^+}{2}} f \right\nsm_{L^2_v} ~ \left\nsm w^{\frac{(\gamma + 2s)^+}{2}} f \right\nsm_{H^\id_v} ~ \right\|_{L^p_{x}}^2
\end{equation}
Here $\id =1$ if $s\in (0, 1/2)$ and $\id=2$ if $s\in [1/2,1)$ as in \eqref{kernelQ} and $(\gamma+2s)^+$ is the positive part of $\gamma+2s$.  
More precisely, we recall that \cite[Proposition 3.1, Eq (3.20)]{SdecaySOFT} states that for real numbers $b^{+}, b^{-}, b' \geq 0$ with $b^{-} \geq b'$, and $b=b^{+}-b^{-}$, one has the following uniform estimate:
\begin{equation}
\label{Prop3.1}
|w^b\Gamma(f,g)|_{L^2_{\vel}} \lesssim |w^{b^{+}-b'} f|_{L^2_{\gamma+2s}}
 |w^{b+b'} g|_{H^{i}_{(\gamma+2s)+(\gamma+2s)^{+}}}
\end{equation}
where $i=1$ for $s \in (0,\frac{1}{2})$ and $i=2$ for $s \in [\frac{1}{2},1)$.  Then to obtain \eqref{upper.nl.est} we use 
\eqref{Prop3.1}.  In the hard potential case \eqref{kernelP}, we take $b^{+}=0, b^{-}=\frac{\gamma+2s}{2}$ and $b'=0$.  In the soft potential case \eqref{kernelPsing}, we take $b^{+}=-\frac{\gamma+2s}{2}, b^{-}=0$, and $b'=0$. 

In the following we will prove upper bounds for the estimate in \eqref{upper.nl.est}.
We do a micro-macro decomposition, e.g. \eqref{form.p}, of $\solU$ to further bound \eqref{upper.nl.est} as
\begin{multline} \label{terms.to.est}
\left\|  ~ \left\nsm w^{\frac{(\gamma+2s)^+}{2}} f \right\nsm_{L^2_v} ~ \left\nsm w^{\frac{(\gamma+2s)^+}{2}} f \right\nsm_{H^\id} ~ \right\|_{L^p_{x}}^2
\lesssim
\left\|  ~ \left| [a,b,c]\right|^2 ~ \right\|_{L^p_{x}}^2
\\
+
\left\|  ~ \left\nsm w^{\frac{(\gamma+2s)^+}{2}}\{ \FI - \FP\} \solU \right\nsm_{L^2_v} ~ \left| [a,b,c]\right| ~ \right\|_{L^p_{x}}^2
\\
+
\left\|  ~ \left| [a,b,c]\right| ~ \left\nsm w^{\frac{(\gamma+2s)^+}{2}} \{ \FI - \FP\} \solU \right\nsm_{H^\id} ~ \right\|_{L^p_{x}}^2
\\
+
\left\|  ~  \left\nsm w^{\frac{(\gamma+2s)^+}{2}}\{ \FI - \FP\} \solU \right\nsm_{L^2_v} ~ 
\left\nsm w^{\frac{(\gamma+2s)^+}{2}} \{ \FI - \FP\} \solU \right\nsm_{H^\id}  ~ \right\|_{L^p_{x}}^2.
\end{multline}
We will now estimate each of the terms in the upper bound of \eqref{terms.to.est} separately.  The main difficulty in estimating these terms arises from the macroscopic parts of the solution.  Notice that for the first upper bound in \eqref{terms.to.est} we have the equality
\begin{equation}\notag
\left\|  ~ \left| [a,b,c]\right|^2 ~ \right\|_{L^p_{x}}^2
=
\left\|   [a,b,c] \right\|_{L^{2p}_{x}}^4.
\end{equation}
If $p=1$, which is the case when $\ALTsig = \frac{\Ndim}{2}$, then a major difficulty is that this term can not be further estimated from above in terms of the dissipation from \eqref{def.dNm}.  
We first note that, by arguing as in the proof of \eqref{macro.point.p.est}, we can deduce that:
\begin{equation}
\label{macro.point.p.est'}
|\pa^{\al_1} [a,b,c]| \lesssim 
|\pa^{\al_1} \FP f |_{L^2_{\vel}}
\end{equation}
Since $\frac{1}{2}+\frac{\ALTsig}{\Ndim} = \frac{1}{p}$, use H{\"o}lder and \eqref{macro.point.p.est'} with $\al_1=0$ to obtain 
\begin{equation}
\label{abc}
\left\|  ~ \left| [a,b,c]\right|^2 ~ \right\|_{L^p_{x}}^2
\lesssim
\left\|   \FP \solU \right\|_{L^{2}_{x}L^2_{\vel}}^2
\left\|   \FP \solU \right\|_{L^{\frac{\Ndim}{\ALTsig} }_{x}L^2_{\vel}}^2.
\end{equation}
Since $\ALTsig \in (0,  \frac{\Ndim}{2}]$ then $\frac{\Ndim}{\ALTsig} \in [2,\infty)$.

Let us first consider the subcase when $\frac{n}{2}-\ALTsig \notin\mathbb{Z}$.  For the  term in $L^{\frac{\Ndim}{\ALTsig}}_{\spa}L^2_{\vel}$ we use \eqref{func.sob.old} and Lemma \ref{partialderivativesLp}  to obtain
\begin{equation}
\label{abc2}
\left\|  \FP \solU\right\|_{L^{{\frac{\Ndim}{\ALTsig} }}_{x}L^2_{\vel}}^2
\lesssim
\left(
\sum_{|\alpha| = m}\left\|  \pa^\alpha \FP \solU\right\|_{L^{2}_{x}L^2_{\vel}} \right)^{2(1-\theta)}
\left(\sum_{|\alpha| = \wE}
\left\| \pa^\alpha \FP \solU \right\|_{L^{2}_{x}L^2_{\vel}}  \right)^{2 \theta}
\end{equation}
where now $\theta = \frac{\frac{\Ndim}{2} - \ALTsig-m}{\wE - m}$ and $m$ and $\wE$, with $m \neq \wE$ are to be chosen.  We generally choose  $\wE\eqdef \lfloor \frac{\Ndim}{2} -\ALTsig+1 \rfloor$  and choose $m \eqdef \lfloor \frac{\Ndim}{2} -\ALTsig \rfloor$.
Then $j-m =1$ and $\theta = \frac{\Ndim}{2} - \ALTsig-\lfloor \frac{\Ndim}{2} -\ALTsig \rfloor \in (0,1)$, since by assumption $\frac{\Ndim}{2} - \ALTsig \notin \mathbb{Z}$.

If $\frac{\Ndim}{2} - \ALTsig \in\mathbb{Z}$, we can use \eqref{func.sob} to deduce that \eqref{abc2} holds with $\theta=0$ and $m=\frac{\Ndim}{2} -\ALTsig$.
This completes our estimate for the first term in \eqref{terms.to.est}.

The remaining terms can be estimated quickly.  Indeed since 
$\frac{1}{2}+\frac{\ALTsig}{\Ndim} = \frac{1}{p}$ the last three terms in \eqref{terms.to.est} are all bounded above by a constant multiple of 
\begin{multline}
\label{abc.last.three}
\| w^{\frac{(\gamma+2s)^+}{2}} f\|_{L^{\frac{\Ndim}{\ALTsig} }_{x}L^2_{v}}^2
  \sum_{|\beta| \le 2} \| w^{\frac{(\gamma+2s)^+}{2}}\partial_\beta \{ \FI - \FP\} f \|_{L^2_{x,v}}^2
  \\
\lesssim \SMeps \sum_{|\beta| \le 2} \| w^{\wN}\partial_\beta \{ \FI - \FP\} f \|_{L^2_{x}L^2_{\gamma+2s}}^2.  
\end{multline}
We used the functional Sobolev embedding \eqref{func.sobolev} to obtain the last inequality above. More precisely, we note that:
$
\|w^{\frac{(\gamma+2s)^+}{2}}f\|_{L^{\frac{n}{\ALTsig}}_{\spa} L^2_{\vel}}
\lesssim \|w^{\frac{(\gamma+2s)^+}{2}}f\|_{H^m_{\spa} L^2_{\vel}}
$
for $m=\lfloor \frac{\Ndim}{2}-\ALTsig + 1\rfloor \leq \dK$.
Hence, since $\ell \geq \frac{(\gamma+2s)^+}{2}$, it follows from Theorem \ref{thm.energy} that
$\|w^{\frac{(\gamma+2s)^+}{2}}f\|^2_{L^{\frac{n}{\ALTsig}}_{\spa} L^2_{\vel}} \lesssim \epsilon_{\dK,\ell}.$
The bound \eqref{abc.last.three} now follows.
\end{proof}


\section{Large-time non-linear decay in Besov spaces}\label{sec.decayNL}

Our goal in this section is to prove the main result. We will do this by combining the differential inequalities from the previous section to prove a stronger differential inequality for a quantity which is equivalent to the appropriate Sobolev norm of the solution $f$ to \eqref{Boltz}. More precisely, for a fixed $k \in\{ 0, 1, \ldots, \dK -1\}$ and for a small $\ka >0$, we define $\CE^k(t)$ to be the following continuous energy functional
\begin{equation}\notag 
\CE^k(t) \eqdef  \frac{1}{2} \sum_{k \le |\alpha| \le \dK} \| \partial^\alpha f(t) \|_{L^2_{\spa} L^2_{\vel}}^2
-
\ka \sum_{k \leq j \leq K-1} \CI^j(t).
\end{equation}
Then by \eqref{functional.est} for $\ka >0$ sufficiently small we have 
\begin{equation}
\label{CEequivalence}
\CE^k(t) \approx \sum_{k \le |\alpha| \le \dK} \| \partial^\alpha f(t) \|_{L^2_{\spa} L^2_{\vel}}^2.  
\end{equation}
Our goal will be to prove a strong differential inequality for $\CE^k(t)$. 
We achieve this by using a new dyadic time-regularity summation argument, and inequalities from Section \ref{secAPP:INTERP}. The main differential inequality is given in \eqref{differential.inequality}. It gives the high Sobolev norm estimate $\CE^k(t) \lesssim (1+t)^{-(k+\ALTsig)}$. In order to derive the differential inequality, we need to assume the a priori uniform bound \eqref{est.unif.besov} on the macroscopic terms in the functional Sobolev space. The latter bound is proved by using the integral inequality in Besov space given in Proposition \ref{prop.besov.ev}.

In Sub-section \ref{subsection1}, we prove some preliminary differential inequalities which follow from the results of the previous section. In Sub-section \ref{subsection2}, we give a  proof of Theorem \ref{thm.main.decay} under the assumption of the a priori bound \eqref{est.unif.besov}. Furthermore, in Sub-section \ref{subsection4}, we verify the a priori bound \eqref{est.unif.besov}. In Sub-section \ref{subsection5}, we collect all the estimates and use an interpolation argument to deduce the bounds which are claimed in Theorem \ref{thm.main.decay} and in Corollary \ref{cor.decay.interp}. We note that the bounds we prove are, in fact, slightly stronger than the ones in the statement; see \eqref{interpolation2}. Finally, in Sub-section \ref{secAPP:faster}, we prove Theorem \ref{thm.extra.decay} by using the improved linear decay estimates from Section \ref{secAPP:linear}. 

\subsection{Some preliminary differential inequalities}
\label{subsection1}
In this sub-section, we collect the main estimates from Section \ref{sec.decayl} in order to deduce one differential inequality for $\CE^k(t)$.
Namely, using Lemma \ref{lem.inter.f} with Proposition \ref{prop:energy}, we deduce the following instantaneous differential inequality for some $\eta>0$:
\begin{equation}\label{main.diff.ce}
 \frac{d \CE^k}{dt} (t) 
 + 
 \eta \left(  \CD^h_{k+1}(t)
 +
  \CD^m_{k}(t) \right)
 \le 0.
\end{equation}
where the hydrodynamic part of the dissipation, $\CD^h_{k+1}(t)$, and the microscopic part of the dissipation,  $\CD^m_{k}(t)$,  are each defined as 
$$
\CD^h_{k+1}(t) \eqdef \sum_{k+1 \le |\alpha| \le \dK} \|  \partial^\alpha [a,b,c]\|_{L^2_{x}}^2,
\quad 
\CD^m_{k}(t) \eqdef \sum_{k \le |\alpha| \le \dK} \| \{\FI - \FP \}  \partial^\alpha f \|_{L^2_\spa L^2_{\gamma+2s}}^2.
$$
More precisely, we first note that, by Lemma \ref{lem.inter.f}, Proposition \ref{prop:energy} and by the definition of $\CE^k(t)$, it follows that for some $C_1,C_2>0$ we have
\begin{multline}\label{dCEdt}
\frac{d \CE^k}{dt} (t) \leq C_1 \epsilon_{\dK,\ell} \sum_{k+1 \leq |\alpha| \leq \dK} \|\partial^{\alpha} f\|_{L^2_x N^{s,\gamma}}^2 - \lambda \sum_{k \leq |\alpha| \leq \dK} \|\{\FI-\FP\} \partial^{\alpha} f\|_{L^2_x N^{s,\gamma}}^2
\\
+C_2 \,\kappa \sum_{k \leq |\alpha| \leq \dK} \|\{\FI-\FP\}\partial^{\alpha}f\|_{L^2_x L^2_{\gamma+2s}}^2 - \kappa \lambda \sum_{k+1 \leq |\alpha| \leq \dK} \|\partial^{\alpha} [a,b,c]\|_{L^2_x}^2
\\
\leq (C_1 \epsilon_{\dK,\ell}-\lambda+C_2 \kappa) \sum_{k \leq |\alpha| \leq \dK} \|\{\FI-\FP\} \partial^{\alpha}f\|_{L^2_x N^{s,\gamma}}^2 
\\
+C_1 \epsilon_{\dK,\ell} \sum_{k+1 \leq 
|\alpha| \leq \dK} \|\FP \partial^{\alpha}f\|_{L^2_x N^{s,\gamma}}^2- \kappa \lambda \sum_{k+1 \leq |\alpha| \leq \dK} \|\partial^{\alpha} [a,b,c]\|_{L^2_x}^2.
\end{multline}
Now, we note that, for some $C_3>0$:
\begin{equation}
\label{Pabc}
\|\FP \partial^{\alpha}f\|_{L^2_x N^{s,\gamma}} \leq C_3 \|\partial^{\alpha}[a,b,c]\|_{L^2_x}
\end{equation}
To prove this we recall \eqref{form.p} and the fact that $\FP$ and $\partial^{\alpha}$ commute.
Then \eqref{Pabc} follows taking $|\cdot|_{N^{s,\gamma}}$, using the triangle inequality, and taking $\|\cdot\|_{L^2_x}$.

One first fixes $\lambda>0$ which satisfies Lemma \ref{lem.inter.f} and Proposition \ref{prop:energy}. Afterwards, one takes $\kappa>0$ small in order to satisfy \eqref{CEequivalence} and for which $-\lambda+C_2 \kappa <0$.
Finally, one chooses $\epsilon_{\dK,\ell}$ small enough so that $C_1\epsilon_{\dK,\ell}-\lambda+C_2 \kappa<0$ and $C_1 C_3^2 \epsilon_{\dK,\ell} - \kappa \lambda<0$.
Substituting these estimates into \eqref{dCEdt}, we deduce that \eqref{main.diff.ce} indeed holds. 

We now prove Theorem \ref{thm.main.decay}.  We will suppose in the proof below that for some $\ALTsig \in (0, \frac{\Ndim}{2}]$ we have the following uniform estimate 
\begin{equation}\label{est.unif.besov}
\| [a,b,c](t)\|_{\dot{B}^{-\ALTsig,\infty}_2} \le C_0<\infty, \quad \forall t \ge 0.
\end{equation}
Let us see how Theorem \ref{thm.main.decay} follows if we know this additional bound.

\subsection{Proof of \eqref{Theorem1B} in Theorem \ref{thm.main.decay} assuming the a priori bound \eqref{est.unif.besov}}
\label{subsection2}
The proof is based on a dyadic time-regularity decomposition. 

\begin{proof}[Proof of \eqref{Theorem1B} in Theorem \ref{thm.main.decay}]
We use time-weighted estimates.  
Fix $s\ge 0$ to be chosen later, and $\varepsilon>0$ small, we multiply \eqref{main.diff.ce} by the time weight $(1+\varepsilon t)^s$ to obtain
\begin{equation}
\label{Theorem1A'}
 \frac{d }{dt} \left( (1+\varepsilon t)^s \CE^k(t) \right)
 + 
 \eta (1+\varepsilon t)^s  \left(  \CD^h_{k+1}(t)
 +
  \CD^m_{k}(t) \right)
 \le s\varepsilon(1+\varepsilon t)^{s-1}  \CE^k(t).
\end{equation}
We use \eqref{CEequivalence}, the decomposition \eqref{form.p} with \eqref{coef.p.def}, and estimates which are analogous to the one used in the proof of \eqref{Pabc} to obtain
\begin{equation}\label{est.upper.differen}
\CE^k(t) \lesssim 
\sum_{ |\alpha| =k} \|  \partial^\alpha [a,b,c]\|_{L^2_{x}}^2(t)
+
\CD^h_{k+1}(t) 
+
\sum_{k \le |\alpha| \le \dK} \| \{\FI - \FP \}  \partial^\alpha f \|_{L^2_\spa L^2_{\vel}}^2(t).
\end{equation}
We will handle each of the terms in the upper bound of \eqref{est.upper.differen} separately.  

To handle the first term in the upper bound of \eqref{est.upper.differen} we notice that
\begin{equation}\label{LP.equiv}
\sum_{ |\alpha| =m} \|  \partial^\alpha [a,b,c]\|_{L^2_{x}}^2(t)
\approx
\sum_{ j \in \Z}  2^{2m j}\|  \Delta_j [a,b,c]\|_{L^2_{x}}^2(t), \quad \forall m =0,1,2,\ldots
\end{equation}
Recall $\Delta_j$ are the Littlewood-Paley projections onto frequencies $2^j$ which are defined in Section \ref{app.besov.h}. Notice that \eqref{LP.equiv} is a consequence of Theorem \ref{LittlewoodPaleyHilbertSpace} combined with Lemma \ref{partialderivativesLp} and \eqref{app.bernstein.XVI}.  
Now \eqref{LP.equiv} implies that we have
\begin{multline}\notag
 (1+\varepsilon t)^{-1} 
\sum_{ 2^j  \sqrt{1+\varepsilon t} \ge 1 }  2^{2 k j}\|  \Delta_j [a,b,c]\|_{L^2_{x}}^2(t)
\lesssim
\sum_{ 2^j  \sqrt{1+\varepsilon t} \ge 1}  2^{2(k+1) j}\|  \Delta_j [a,b,c]\|_{L^2_{x}}^2(t)
\\
\lesssim \sum_{ |\alpha| =k+1} \|  \partial^\alpha [a,b,c]\|_{L^2_{x}}^2(t).
\end{multline}
Crucially the implied constants are uniform in $\sqrt{1+\varepsilon t}$.  
Using \eqref{LP.equiv} we have shown 
\begin{multline}\label{LP.lbound}
 (1+\varepsilon t)^{-1} \sum_{ |\alpha| =k} \|  \partial^\alpha [a,b,c]\|_{L^2_{x}}^2(t)
\lesssim\sum_{ |\alpha| =k+1} \|  \partial^\alpha [a,b,c]\|_{L^2_{x}}^2(t)
\\
+ (1+\varepsilon t)^{-1} 
\sum_{ 2^j  \sqrt{1+\varepsilon t} \le 1 }  2^{2 k j}\|  \Delta_j [a,b,c]\|_{L^2_{x}}^2(t).
\end{multline}
Furthermore, we have the direct estimate 
\begin{multline}\notag
\sum_{ 2^j  \sqrt{1+\varepsilon t} \le 1 }  2^{2 k j}\|  \Delta_j [a,b,c](t)\|_{L^2_{x}}^2
\\
\lesssim
\| [a,b,c](t)\|_{\dot{B}^{-\ALTsig,\infty}_2}^2
\sum_{ 2^j  \sqrt{1+\varepsilon t} \le 1 }  2^{2 (k+\ALTsig) j}
\\
\lesssim
\| [a,b,c](t)\|_{\dot{B}^{-\ALTsig,\infty}_2}^2
 (1+\varepsilon t)^{-(k+\ALTsig)}.
\end{multline}
In the last inequality we have used the following calculation
\begin{multline}\notag
\sum_{ 2^j  \sqrt{1+\varepsilon t} \le 1 }  2^{2 (k+\ALTsig) j}
=
 (1+\varepsilon t)^{-(k+\ALTsig)} 
\sum_{ 2^j  \sqrt{1+\varepsilon t} \le 1 }  \left(  2^j  \sqrt{1+\varepsilon t} \right)^{2 (k+\ALTsig) }
\\
\lesssim
 (1+\varepsilon t)^{-(k+\ALTsig)},
\end{multline}
where the implied uniform constant in the upper bound is independent of the size of $\sqrt{1+\varepsilon t}$.  Collecting these estimates, including \eqref{LP.lbound}, we obtain 
\begin{multline}\notag 
 (1+\varepsilon t)^{-1} \sum_{ |\alpha| =k} \|  \partial^\alpha [a,b,c]\|_{L^2_{x}}^2(t)
\lesssim  \sum_{ |\alpha| =k+1} \|  \partial^\alpha [a,b,c]\|_{L^2_{x}}^2(t)
\\
+  \| [a,b,c](t)\|_{\dot{B}^{-\ALTsig,\infty}_2}^2 (1+\varepsilon t)^{-1-(k+\ALTsig)}.
\end{multline}
which by \eqref{est.unif.besov} is:
\begin{equation}
\label{Theorem1B.needeed}
\lesssim \sum_{ |\alpha| =k+1} \|  \partial^\alpha [a,b,c]\|_{L^2_{x}}^2(t) + C_0^2 (1+\varepsilon t)^{-1-(k+\ALTsig)}.
\end{equation}
This will be our main estimate for the first term in the upper bound of \eqref{est.upper.differen}.

For the third term in \eqref{est.upper.differen} we will use a time-velocity splitting.  Recalling \eqref{weigh},  we define the time-velocity splitting sets by
\begin{equation}
E(t) = \{(1+\varepsilon t)^{-1} \le  w^{\gamma+2s}(v) \}, \quad  E^c(t)= \{ 1 <  w^{-\gamma-2s}(v) (1+\varepsilon t)^{-1}\}.
\label{split.E}
\end{equation}
With this splitting, we have the following estimate
\begin{multline}\label{Theorem1C}
(1+\varepsilon t)^{-1}
\sum_{k \le |\alpha| \le \dK} \| \{\FI - \FP \}  \partial^\alpha f \|_{L^2_\spa L^2_{\vel}}^2
\lesssim
\sum_{k \le |\alpha| \le \dK} \| \ind_E \{\FI - \FP \}  \partial^\alpha f \|_{L^2_\spa \spaceL}^2
\\
+
\Am (1+\varepsilon t)^{-1-(k+\ALTsig)}\sum_{k \le |\alpha| \le \dK} 
\|w^{\wNz} \ind_{E^c} \{\FI - \FP \}  \partial^\alpha f \|_{L^2_\spa L^2_{\vel}}^2
\\
\lesssim
\sum_{k \le |\alpha| \le \dK} \| \ind_E \{\FI - \FP \}  \partial^\alpha f \|_{L^2_\spa \spaceL}^2
+
\Am (1+\varepsilon t)^{-1-(k+\ALTsig)} \SMeps.
\end{multline}
where $\ind_E$ is the usual indicator function of the set $E$ from \eqref{split.E}.  Using $E^c$ the last estimate above holds for $\Am = 1$ since $\wNz \ge -(\gamma+2s)(k+\ALTsig)/2\ge 0$ in the case of the soft potentials \eqref{kernelPsing} using \eqref{wNz.def}.  For the hard potentials \eqref{kernelP}, we notice that $\ind_E \equiv 1$ (always) and we then have the above estimate with $\Am = 0$.

We use \eqref{Theorem1A'}, \eqref{est.upper.differen}, \eqref{Theorem1B.needeed}, and \eqref{Theorem1C} to deduce that, for $\varepsilon>0$ sufficiently small\begin{multline}
\label{differential.inequality} 
 \frac{d }{dt} \left( (1+\varepsilon t)^s \CE^k(t) \right)
 + 
\lambda (1+\varepsilon t)^s  \left(  \CD^h_{k+1}(t)
 +
  \CD^m_{k}(t) \right)
  \\
\lesssim
(1+\varepsilon t)^{s-1-(k+\ALTsig)} \left(C_0^2+ \SMeps\right).
\end{multline}
Choosing $s = k+\ALTsig + \delta$ for a small $\delta\in (0,1)$ and integrating this in time:
\begin{multline}\label{main.decay.proof1}
  \CE^k(t)
\lesssim
\left(C_0^2+ \SMeps\right)
(1+\varepsilon t)^{-k-\ALTsig - \delta}
\int_0^t du ~ (1+\varepsilon u)^{\delta-1}
\\
+(1+\varepsilon t)^{-k-\ALTsig - \delta} \CE^k(0)
\\
\lesssim
(1+ t)^{-k-\ALTsig - \delta} (1+ t)^{ \delta}
\approx (1+ t)^{-(k+\ALTsig)}.
\end{multline}
The constant is uniform in $t \ge 0$.  Then \eqref{main.decay.proof1} holds as long as we have \eqref{est.unif.besov}.  We note that \eqref{Theorem1B} immediately follows from \eqref{CEequivalence} and \eqref{main.decay.proof1}.  \end{proof}

\subsection{Proof of the a priori bound \eqref{est.unif.besov}}
\label{subsection4}

In this sub-section, we verify the a priori bound \eqref{est.unif.besov} which was crucial in the proof of Theorem \ref{thm.main.decay}. In order to prove this bound, we will use Proposition \ref{prop.besov.ev}. As was noted earlier, the first integral on the right-hand side of the inequality obtained in Proposition \ref{prop.besov.ev} is bounded by the integral of the dissipation \eqref{def.dNm}, whereas for the second integral, we have to work harder.
As we will see in the proof, the case when  $\ALTsig \in (0,\frac{n-2}{2})$ can still 
be estimated by using the integral of the dissipation, whereas the case $\ALTsig \in [\frac{n-2}{2},\frac{n}{2}]$ is more difficult, and it requires an additional interpolation step. We are interested in obtaining the endpoint case $\rho=\frac{n}{2}$ since $ \dot{B}^{-\frac{n}{2},\infty}_2 \supset L^1(\mathbb{R}^{n}_x)$.

\begin{proof}[Proof of \eqref{est.unif.besov}]
We will estimate the last two terms in the upper bound of \eqref{besov.prop.small}.  In particular, for any solution $\solU(t)$ as in Theorem \ref{thm.energy}, we will prove the special case of \eqref{Theorem1A} when $m=-\ALTsig$
\begin{equation}\label{est.full.unif.besov}
\| \solU(t)\|_{\dot{B}^{-\ALTsig,\infty}_2 L^2_{\vel} }\le C_0<\infty, \quad \forall t \ge 0,
\quad \text{if} \quad \| \solU_0\|_{\dot{B}^{-\ALTsig,\infty}_2 L^2_{\vel} }<\infty.
\end{equation}
We note that this bound is stronger than \eqref{est.unif.besov} since $\FP$ is a projection on $L^2_{\vel}$ and since the Littlewood-Paley projections $\Delta_j$ commute with $\FP$. In other words, we are using that for all $j$:
$$2^{-\ALTsig j} \|\Delta_j f\|_{L^2_{\spa} L^2_{\vel}} \geq 
2^{-\ALTsig j} \|\Delta_j \FP f\|_{L^2_{\spa} L^2_{\vel}} \thickapprox
2^{-\ALTsig j} |\Delta_j [a,b,c]|_{L^2_{\vel}}$$
where we are also using \eqref{coef.p.def}. By taking suprema in $j$, it follows that the condition \eqref{est.full.unif.besov} is indeed stronger than \eqref{est.unif.besov}.

For the second term in \eqref{besov.prop.small}, for some finite constant $C>0$, from \eqref{thm.energy.1} we have that
\begin{multline}
\label{DnM.bound1}
\int_0^t ds \sum_{|\beta| \leq 2} \|\partial_{\beta} \{\FI-\FP\}f(s)\|_{L^2_x L^2_{\gamma+2s}}^2
\\
\leq 
 \int_0^t ds 
\sum_{|\beta|\le 2} \|\partial_{\beta} \{\FI-\FP\}f(s)\|^2_{L^2_{\spa} N^{s,\gamma}_{\ell-2\rho}}
\le C,
\end{multline}
provided that $\ell-2\rho \geq 0$.
This condition is satisfied because $\ell \ge \wNz$ for $\wNz$ from \eqref{wNz.def}. Namely, for the hard potentials \eqref{kernelP}, we have $\wNz \geq 2=2\rho$. 
In the soft potential case \eqref{kernelPsing}, we also know that $\wNz \geq 2(-\gamma-2s)=2\rho$ from \eqref{wNz.def}.

To estimate the last term in the upper bound of \eqref{besov.prop.small} we first suppose that $\ALTsig \in (0, \frac{\Ndim - 2}{2}]$; then \eqref{est.unif.besov} will follow directly from \eqref{besov.prop.small} when combined with the time integrated Lyapunov inequality \eqref{thm.energy.1}.  
In particular $\ALTsig \in (0, \frac{\Ndim - 2}{2}]$ implies that $\lfloor \frac{\Ndim}{2} -\ALTsig \rfloor \ge 1$ and also
$2 \le \lfloor \frac{\Ndim}{2} -\ALTsig +1 \rfloor \le \ksob$.  Then we have
\begin{multline}\notag 
 \int_0^t ds 
\left\|    \FP \solU (s)\right\|_{L^{2}_{\spa}L^2_{\vel}}^{2}
\sum_{|\alpha| =\lfloor \frac{\Ndim}{2} -\ALTsig \rfloor }\left\| \partial^\alpha   \FP \solU(s)\right\|_{L^{2}_{\spa}L^2_{\vel}}^{2(1-\theta)}
\sum_{|\alpha'| =\lfloor \frac{\Ndim}{2} -\ALTsig+1 \rfloor}\left\| \partial^{\alpha'}   \FP \solU(s)\right\|_{L^{2}_{\spa}L^2_{\vel}}^{2\theta}
\\
\lesssim
\SMeps
 \int_0^t ds 
\sum_{1 \le |\alpha| \le \ksob }\left\| \partial^\alpha   \FP \solU(s)\right\|_{L^{2}_{\spa}L^2_{\vel}}^{2}
\lesssim 1.
\end{multline}
Recall from \eqref{besov.prop.small} that $\theta \in [0,1)$.  In the above we have used the time integrated Lyapunov inequality \eqref{thm.energy.1}, \eqref{def.dNm} and the fact that $\FP$ is a bounded operator on $L^2_x L^2_{\vel}$  which commutes with differentiation in $x$. Then when $\ALTsig \in (0, \frac{\Ndim - 2}{2}]$
and $\| f_0\|_{\dot{B}^{-\ALTsig,\infty}_2 L^2_{\vel}} < \infty$, 
 \eqref{est.unif.besov} follows from \eqref{besov.prop.small} and these last few calculations.

The next case we consider is when $\ALTsig \in (\frac{\Ndim - 2}{2}, \frac{\Ndim}{2})$ and $\| f_0\|_{\dot{B}^{-\ALTsig,\infty}_2 L^2_{\vel}} + \| f_0\|_{L^2_{\spa} L^2_{\vel}} < \infty$ (as we will see below, the case $\ALTsig=\frac{\Ndim}{2}$ is a little bit different). In this situation we obtain that $\| f_0\|_{\dot{B}^{-\ALTsig',\infty}_2 L^2_{\vel}}< \infty$ for any $\ALTsig' \in (0, \ALTsig)$ by interpolation:
\begin{equation}\label{interpolationresult}
\| f_0\|_{\dot{B}^{-\ALTsig',\infty}_2 L^2_{\vel}} 
\le 
\| f_0\|_{\dot{B}^{-\ALTsig,\infty}_2 L^2_{\vel}}^{\ALTsig'/\ALTsig}  
\| f_0\|_{L^2_{\spa} L^2_{\vel}}^{(\ALTsig-\ALTsig')/\ALTsig}.
\end{equation}  
Then \eqref{interpolationresult} follows from Lemma \ref{besov.LM} with $\hilbertONE=L^2_{\vel}$ 
and $L^2_{\spa} L^2_{\vel} \approx \dot{B}^{0,2}_2 L^2_{\vel} \subset \dot{B}^{0,\infty}_2 L^2_{\vel}$.
We conclude that \eqref{main.decay.proof1} holds for any $\ALTsig'  \in (0, \frac{\Ndim - 2}{2})$ and any $k\in \{0,1,\ldots, \dK -1\}$.

Then, by using \eqref{besov.prop.small}, \eqref{DnM.bound1}, \eqref{CEequivalence} and \eqref{main.decay.proof1} for $\ALTsig'$, it follows that
\begin{multline}\label{besov.decay.l}
\| f(t) \|_{\dot{B}_2^{-\ALTsig,\infty}L^2_{\vel}}^2
\lesssim 
\| f_0 \|_{\dot{B}_2^{-\ALTsig,\infty}L^2_{\vel}}^2 + 1
\\
+
 \int_0^t ds ~ 
 \left( 1+s \right)^{-\ALTsig'}
\left( 1+s \right)^{-\left(\lfloor \frac{\Ndim}{2} -\ALTsig \rfloor + \ALTsig'\right)(1-\theta)}
\left( 1+s \right))^{-\left(\lfloor \frac{\Ndim}{2} -\ALTsig +1\rfloor + \ALTsig'\right)\theta}
\\
\le C_0 < \infty.
\end{multline}
Here from \eqref{besov.prop.small} we use $\theta = \frac{\Ndim}{2} -\ALTsig -\lfloor \frac{\Ndim}{2} -\ALTsig \rfloor$.
Given any $\ALTsig \in (\frac{\Ndim - 2}{2}, \frac{\Ndim}{2})$ 
we choose  $\ALTsig' \eqdef \frac{\Ndim - 2}{2} - \epsilon'\in (0, \frac{\Ndim - 2}{2})$ for any sufficiently small
$\epsilon'\in (0, \frac{\theta}{2})$.  This then guarantees that the upper bound for \eqref{besov.decay.l} is finite since $\DgE$. More precisely, we want to guarantee that $-2\ALTsig'-\lfloor \frac{\Ndim}{2}-\ALTsig \rfloor-\theta= -2\ALTsig'-\theta<-1$, which can be shown to follow from from $2\epsilon'-\theta < \Ndim-3$. Then the above choice of $\epsilon'$ is sufficient.
In the above calculation, we used the fact that $\lfloor \frac{\Ndim}{2}-\ALTsig \rfloor=0$ whenever $\ALTsig \in (\frac{\Ndim - 2}{2}, \frac{\Ndim}{2})$.

Finally, if $\ALTsig = \frac{\Ndim}{2}$ and $\| f_0\|_{\dot{B}^{-\ALTsig,\infty}_2 L^2_{\vel}}< \infty$ we again use the estimates \eqref{interpolationresult} and \eqref{besov.decay.l} with $\theta =0$ and
  $\ALTsig' \eqdef \frac{\Ndim }{2} - \epsilon \in (\frac{\Ndim - 2}{2}, \frac{\Ndim}{2})$ for some sufficiently small
$\epsilon\in (0, 1/2)$.  Then $-2\ALTsig'<-1$ in  \eqref{besov.decay.l} and this establishes \eqref{est.unif.besov} for any $\ALTsig \in [0, \frac{\Ndim}{2}]$.  \end{proof}

\subsection{Conclusion of the proof; the interpolation step}
\label{subsection5}
In this sub-section, we prove the exact statement of the result in Theorem \ref{thm.main.decay}. 
For completeness, we explain now how to deduce \eqref{Theorem1A} and \eqref{Theorem1B}.  In particular from \eqref{CEequivalence}, \eqref{main.decay.proof1} and the fact that
$\dot{H}^{k}_{\spa,\vel}  \approx \dot{B}^{k,2}_2 L^2_{\vel}$, we obtain that
\begin{equation}\notag
\|f(t)\|_{\dot{B}^{k,2}_2 L^2_{\vel}}^2 \lesssim (1+t)^{-(k+\ALTsig)},
\end{equation}
for $k \in \{0,1,\ldots,\dK-1\},$
which is \eqref{Theorem1B}.
We furthermore use the above, \eqref{est.full.unif.besov}, and \eqref{func.interp'} with $\hilbertONE=L^2_{\vel}$ to deduce that
\begin{equation}
\label{interpolation2}
\|f(t)\|_{\dot{B}^{a,s}_2 L^2_{\vel}}^2 \lesssim (1+t)^{-(a+\ALTsig)},
\quad \forall s \ge 2\left(\frac{k+\ALTsig}{a+\ALTsig} \right),
\quad \forall a \in [-\ALTsig,k],
\end{equation}
which is stronger than the bound in \eqref{Theorem1A}.  These estimates hold uniformly in $t \ge 0$.  

Specifically to obtain \eqref{interpolation2} we used the interpolation result \eqref{func.interp'} with $\ell = a$, $k= - \ALTsig$, $m=k$, $q' = 2\left(\frac{k+\ALTsig}{a+\ALTsig} \right)$, $r' = \infty$, $p' =2$ and $q=r=p=2$.
We also noted the following embeddings 
$\dot{B}^{a,q'}_2 L^2_{\vel} \subset \dot{B}^{a,s}_2 L^2_{\vel} \subset  \dot{B}^{a,\infty}_2 L^2_{\vel}$.

Finally, we prove Corollary \ref{cor.decay.interp}:

\begin{proof}[Proof of Corollary \ref{cor.decay.interp}]
Fix some $2 \le r \le \infty$.
We will use the interpolation estimate
\begin{equation}\label{func.sob.new.old}
\sum_{|\alpha| = k } \| \partial^\alpha \solU(t)\|_{L^r_\spa L^2_{\vel}}^2
\lesssim
\|  \solU(t) \|_{L^2_{\spa} L^2_{\vel} }^{2(1-\theta)}
\sum_{|\alpha| = \dK-1 }
\| \partial^\alpha \solU(t) \|_{L^2_{\spa} L^2_{\vel} }^{2\theta},
~~
\theta = \frac{k+\frac{\Ndim}{2} - \frac{\Ndim}{r} }{\dK-1}.
\end{equation}
This follows from \eqref{func.sob.old} (with $m=0$, $\ALTsig = \dK-1$, and $\hilbertONE = L^2_{\vel}$) combined with Lemma \ref{partialderivativesLp}.  The interpolation \eqref{func.sob.new.old} holds (using \eqref{func.sob.old}) when $\dK-1 > k + \frac{\Ndim}{2} - \frac{\Ndim}{r}$ (which explains the corresponding restriction in Corollary \ref{cor.decay.interp}).  Then Corollary \ref{cor.decay.interp} follows by applying \eqref{Theorem1B} to the upper bounds of \eqref{func.sob.new.old}. 
\end{proof}


\subsection{The proof of the faster time-decay rates in Theorem \ref{thm.extra.decay}}\label{secAPP:faster}
Let us now prove Theorem \ref{thm.extra.decay}. In order to prove this theorem, we will need to use linear decay estimates given by Proposition \ref{linear.decay.B} and Proposition \ref{linear.decay.m} below. We will use estimates with weights of the form $(1+\varepsilon t)^s$, for $\varepsilon$ a small parameter.  Consider the linearized Boltzmann equation with a microscopic source $\sourceG=\sourceG(t,x,\vel)$:
\begin{equation}\label{ls}
    \left\{\begin{array}{l}
    \pa_t \solU+\vel\cdot\na_x \solU +\FL \solU =\sourceG,\\
 \solU|_{t=0}=\solU_0.
    \end{array}\right.
\end{equation}
For the nonlinear system \eqref{Boltz}, the non-homogeneous source term is given by
\begin{equation}\label{def.g.non}
    \sourceG=\Ga(\solU,\solU) = \{\FI-\FP\} \Ga(\solU,\solU).
\end{equation}
This fact follows from \eqref{ProjectionGamma}.
Solutions of \eqref{ls} formally take the following form
\begin{equation}
    \solU(t)=\semiG(t)\solU_0 + \int_0^t d\tau ~ \semiG(t-\tau)~\sourceG(\tau), \quad \semiG(t)\eqdef e^{-t\SB}, \quad \SB \eqdef \FL+\vel\cdot\na_x.
    \label{ls.semi}
\end{equation}
Here $\semiG(t)$ is the linear solution operator corresponding to  \eqref{ls} with $\sourceG=0$.

Our goal in this section will be to prove Theorem \ref{thm.extra.decay}. This theorem is more subtle that the previous decay theorems because of the more negative regularity exponent in for example the space $\dot{B}^{-(\Ndim+2)/2,\infty}_2$.  In this situation we do not have uniform in time bounds such as either \eqref{est.full.unif.besov} or even \eqref{est.unif.besov}.  Therefore the previous methods are difficult to apply, and instead we will use linear decay estimates.

\begin{proposition}\label{linear.decay.B}  Suppose $m , \ALTsig \in \R$ with $m+\ALTsig > 0$ and $\wN\in \R$.
Then
$$
\|w^{\wN}\semiG(t) f_0\|_{\dot{H}^m_{\spa}L^2_{\vel}} \lesssim (1+t)^{-\frac{\ALTsig+m}{2}} 
\|w^{\wN+\sigma} f_0\|_{\dot{H}^m_{\spa}L^2_{\vel} \cap \dot{B}^{-\ALTsig,\infty}_2L^2_{\vel} }.
$$
This holds when $ \|  w^{\wN+\sigma} \solU_0 \|_{ \dot{H}^m_{\spa}L^2_{\vel} \cap \dot{B}^{-\ALTsig,\infty}_2L^2_{\vel}  }< \infty$.  Notice that for the additional weight on the initial data we assume
$\sigma > -(m+\ALTsig)(\gamma+2s)>0$ for the soft potentials \eqref{kernelPsing}.  And for the hard potentials \eqref{kernelP} we take $\sigma =0$. 
\end{proposition}

We point out that Proposition \ref{linear.decay.B} is proven in Theorem \ref{thm.decay.lin} of Section \ref{secAPP:linear}.  In the following, we observe faster decay in the hard potential case \eqref{kernelP} when the initial data is microscopic, as in \eqref{form.p}.

\begin{proposition}\label{linear.decay.m}
Suppose the initial condition $f_0$ in \eqref{ls} with $\sourceG=0$ satisfies:
\begin{equation}
\label{microscopicf0}
\FP f_0=0.
\end{equation}
Fix $m , \ALTsig \in \R$ with $m+\ALTsig > 0$ and $\wN \ge 0$.
Then we have
$$
\|w^{\wN} \semiG(t) f_0\|_{\dot{H}^m_{\spa}L^2_{\vel}} \lesssim (1+t)^{-\frac{\ALTsig+m+1}{2}} 
\|w^{\wN} f_0\|_{\dot{H}^m_{\spa}L^2_{\vel} \cap \dot{B}^{-\ALTsig,\infty}_2L^2_{\vel}}.
$$
This faster decay is proven in the hard potential case  \eqref{kernelP}.
\end{proposition}

Again Proposition \ref{linear.decay.m} is proven in Theorem \ref{linear.decay.A} of Section \ref{secAPP:linear}.  Now we use these linear decay results, and previous developments to prove Theorem \ref{thm.extra.decay}.

\begin{proof}[Proof of Theorem \ref{thm.extra.decay}]
Starting from \eqref{thm.energy.1}, we obtain as before
\begin{equation}\label{main.diff.decay}
\frac{d}{dt}\CE_{\dK,\wN}(t)
 + 
 \la \left(  \CD^h_{\dK}(t)
 +
  \CD^m_{\dK,\wN}(t) \right)
 \le 0.
\end{equation}
where the hydrodynamic part of the dissipation, $\CD^h_{\dK}(t)$, and the microscopic part of the dissipation,  
$\CD^m_{\dK,\wN}(t)$,  are each defined as in the following
$$
\CD^h_{\dK}(t) \eqdef \sum_{1 \le |\alpha| \le \dK} \|  \partial^\alpha [a,b,c]\|_{L^2_{x}}^2,
\quad 
\CD^m_{\dK,\wN}(t) \eqdef \sum_{|\al|+|\be|\leq \dK}\|\pa^\al_\be \{\FI-\FP\}\solU(t)\|_{L^2_{\spa} \spacen_{\wN-|\be|\rho}}^2.
$$
Here we also recall \eqref{def.dNm} for the definition of $\rho$.  First for the case of Theorem \ref{thm.extra.decay} when $\ALTsig \in (0, \Ndim/2]$ and $\| f_0\|_{\dot{B}^{-\ALTsig,\infty}_2 L^2_{\vel}} < \infty$, then we still have \eqref{est.full.unif.besov} and we can use the proof of Theorem \ref{thm.main.decay}.  This establishes \eqref{TheoremExtra} when $\ALTsig \in (0, \Ndim/2]$.

In the remainder of our proof,  we suppose \eqref{kernelP} and we consider the case when $\| \FP f_0\|_{\dot{B}^{-\ALTsig,\infty}_2 L^2_{\vel}} < \infty$ and also
$\| \{\FI - \FP\} f_0\|_{\dot{B}^{-\ALTsig+1,\infty}_2 L^2_{\vel}} < \infty$ for $\ALTsig \in (\Ndim/2, (\Ndim+2)/2]$. Notice that in this situation we may not know either \eqref{est.full.unif.besov} or even \eqref{est.unif.besov}.  Instead we will use the time-weighted estimates and the linear decay theory. 

  Fix $s\ge 0$ to be chosen later, and $\varepsilon>0$ small (also determined below), we now multiply \eqref{main.diff.decay} by the time weight $(1+\varepsilon t)^s$ to obtain
\begin{equation}\label{Theorem1B.weigh}
 \frac{d}{dt} \left( (1+\varepsilon t)^s  \CE_{\dK,\wN}(t) \right)
 + 
 \la (1+\varepsilon t)^s  
   \left(  \CD^h_{\dK}(t)
 +
  \CD^m_{\dK,\wN}(t) \right)
 \le s\varepsilon(1+\varepsilon t)^{s-1}  \CE_{\dK,\wN}(t).
\end{equation}
We use \eqref{def.eNm}, the decomposition \eqref{form.p} with \eqref{coef.p.def}, and estimates which are analogous to the one used in the proof of \eqref{Pabc} to obtain
\begin{equation}\label{est.upper.differe.weigh}
\CE_{\dK,\wN}(t)
\lesssim 
 \|  [a,b,c]\|_{L^2_{x}}^2(t)
+
\CD^h_{\dK}(t) 
+
\sum_{|\al|+|\be|\leq \dK}\| w^{\ell - |\beta| \rho} \pa^\al_\be \{\FI-\FP\}\solU(t)\|_{L^2_{\spa}L^2_{\vel}}^2.
\end{equation}
We will handle each of the terms in the upper bound of \eqref{est.upper.differe.weigh} separately.

Initially, our focus will be on the first term in the upper bound of \eqref{est.upper.differe.weigh}.    As in  \eqref{ls.semi}, with $\sourceG$ given by \eqref{def.g.non}, we expand the solution to \eqref{Boltz} as
\begin{equation}  \label{Boltz.rep}
   \solU(t)=\semiG(t)\FP\solU_0+\semiG(t)\{\FI-\FP\}\solU_0+I_1(t),
\end{equation}
where we additionally use \eqref{def.g.non} to observe that
$$
  I_1(t) =\int_0^t\semiG(t-\tau)\{\FI-\FP\}\Gamma(\solU,\solU)(\tau)d\tau.
$$
We now apply Propositions \ref{linear.decay.B} and \ref{linear.decay.m}
 to $\semiG(t)\FP\solU_0$ and $\semiG(t)\{\FI-\FP\}\solU_0$ respectively, to obtain
 \begin{eqnarray}  \notag
    \| \semiG(t)\FP\solU_0\|_{L^2_{\vel} L^2_{\spa} }
  &  \lesssim &
    (1+t)^{-\frac{\rho}{2}} \| \FP\solU_0\|_{L^2_{\vel} L^2_{\spa} \cap \dot{B}_{2}^{-\ALTsig,\infty}L^2_{\vel}  },
\\
    \| \semiG(t)\{ \FI - \FP\}\solU_0\|_{L^2_{\vel} L^2_{\spa} }
  &  \lesssim &
    (1+t)^{-\frac{\rho}{2}} \| \{ \FI - \FP\}\solU_0\|_{L^2_{\vel} L^2_{\spa} \cap \dot{B}_{2}^{-\ALTsig+1,\infty}L^2_{\vel}  },
    \notag
\end{eqnarray}
 where we recall that here $\ALTsig \in (\Ndim/2, (\Ndim+2)/2]$.   For $I_1(t)$, we use Proposition \ref{linear.decay.m} to deduce the estimate
\begin{multline*}
    \|  I_1(t)\|_{L^2_{\vel} L^2_{\spa}}
    \leq 
    \int_0^t\| \semiG(t-\tau)\{ \FI - \FP\}\Gamma(\solU,\solU)(\tau)\|_{L^2_{\vel} L^2_{\spa}}d\tau
    \\
\lesssim
    \int_0^t(1+t-\tau)^{-\frac{\ALTsig}{2}}\| \Gamma(\solU,\solU)(\tau)\|_{L^2_{\vel} L^2_{\spa}\cap \dot{B}_{2}^{-\ALTsig+1,\infty}L^2_{\vel}}d\tau
        \\
\lesssim
    \int_0^t(1+t-\tau)^{-\frac{\ALTsig}{2}}\| \Gamma(\solU,\solU)(\tau)\|_{L^2_{\vel} L^2_{\spa}\cap L^2_{\vel}L^p_{\spa}}d\tau.
\end{multline*}
For the final inequality above we used the embedding 
$
\dot{B}_{2}^{-\ALTsig+1,\infty}L^2_{\vel} \supset L^2_{\vel}\dot{B}_{2}^{-\ALTsig+1,\infty}
$
where $L^2_{\vel}\dot{B}_{2}^{-\ALTsig+1,\infty}$ denotes the space with norm 
$
\left\| \| \cdot \|_{\dot{B}_{2}^{-\ALTsig+1,\infty}} \right\|_{L^2_{\vel}}.
$
This is then followed by the embedding 
$
\dot{B}_{2}^{-\ALTsig+1,\infty} \supset  L^p_{\spa}
$
for $p = \frac{\Ndim}{\ALTsig -1 + \frac{\Ndim}{2}}\in [1, 2)$ when $\ALTsig \in (\Ndim/2, (\Ndim + 2)/2]$  which itself is a consequence of Lemma \ref{lem.besov.emb}.   Combining these we obtain 
$
\dot{B}_{2}^{-\ALTsig+1,\infty}L^2_{\vel} \supset L^2_{\vel}\dot{B}_{2}^{-\ALTsig+1,\infty} \supset 
L^2_{\vel}L^p_{\spa}
$
as desired.

Notice further that by interpolation, \cite[Equation (3.22)]{SdecaySOFT}, and $p \in [1,2)$, we have, for all $\tau \in [0,t]$ that
\begin{equation}
\label{bilinearGamma}
\| \Gamma(\solU,\solU)(\tau)\|_{L^2_{\vel} L^2_{\spa}\cap  L^2_{\vel} L^p_{\spa}}
\lesssim 
\| \Gamma(\solU,\solU)(\tau)\|_{L^2_{\vel} L^2_{\spa}\cap L^2_{\vel}L^1_{\spa}}
\lesssim \CE_{\dK,\wN' }(\tau).  
\end{equation}
The last inequality holds for $\wN' = 2(\gamma+2s)$.   Now we will assume that
\begin{equation}
\CE_{\dK,\wN}(t)  \lesssim  (1+t)^{-\alpha}
\label{inidial.decay.use}
\end{equation}
holds for some $\alpha>0$ and use this iteratively to upgrade the decay rate.   Then, in the following estimates, we will always be working with $\ell \geq \wNz \ge \wN'$ for $\wNz$ in \eqref{wNz.def}.

Since $\| \FP f_0\|_{\dot{B}^{-\ALTsig,\infty}_2 L^2_{\vel}} +\| \{\FI - \FP\} f_0\|_{\dot{B}^{-\ALTsig+1,\infty}_2 L^2_{\vel}} < \infty$  we have 
$\| f_0\|_{\dot{B}^{-\ALTsig+1,\infty}_2 L^2_{\vel}} < \infty$ as in \eqref{interpolationresult}, since $f_0 \in L^2_{\spa} L^2_{\vel}$. Then for $\ALTsig \in (\Ndim/2, (\Ndim+2)/2]$, from the first part of Theorem \ref{thm.extra.decay} which is already proven, \eqref{inidial.decay.use} holds with $\alpha = \ALTsig-1$.


We collect all of the previous estimates and use \eqref{Boltz.rep} to conclude that 
\begin{multline}\label{thm.ns.p2}
\|  [a,b,c]\|_{L^2_{\spa}}(t)
\lesssim
\|  \solU(t)\|_{L^2_{\spa}L^2_{\vel}}
\\
\lesssim
    \|\semiG(t)\FP\solU_0\|_{L^2_{\vel} L^2_{\spa} }
    +
    \|\semiG(t)\{ \FI - \FP\}\solU_0\|_{L^2_{\vel} L^2_{\spa} }
    +
     \| I_1(t)\|_{L^2_{\vel} L^2_{\spa}}
     \\
\lesssim
(1+t)^{-\frac{\rho}{2}}
+
 \int_0^t d\tau ~ (1+t-\tau)^{-\frac{\ALTsig}{2}} (1+\tau)^{-\alpha}.
\end{multline}
The first inequality above used estimate \eqref{macro.point.p.est}.   Then we evaluate the time integral as in \cite[Proposition 4.5]{strainSOFT}, to conclude that 
\begin{equation}
 \int_0^t d\tau ~ (1+t-\tau)^{-\frac{\ALTsig}{2}} (1+\tau)^{-\alpha}
\lesssim
A(t)(1+t)^{-\min\{\frac{\ALTsig}{2}+\alpha-1,\frac{\ALTsig}{2}, \alpha\}}.
\label{desired.a.est}
\end{equation}
where $A(t) = \log(2+t)$ if $\max\{\frac{\ALTsig}{2}, \alpha\} = 1$ and $A(t) = 1$ otherwise.
Then
\begin{equation}
\|  [a,b,c]\|_{L^2_{\spa}}(t)
\lesssim
(1+t)^{-\beta}.
\label{desired.a.est.spec}
\end{equation}
where if $\ALTsig > 2$, i.e.  $\ALTsig \in (2, (\Ndim + 2)/2]$ then since $\alpha = \ALTsig-1$ we have from \eqref{thm.ns.p2} and \eqref{desired.a.est} that
$
\beta = \frac{\rho}{2}.
$
This is the desired estimate when $\ALTsig > 2$.

For the third term in the upper bound of \eqref{est.upper.differe.weigh}, we have the uniform estimate
\begin{equation}\label{CTheorem1C}
\| w^{\ell - |\beta| \rho} \pa^\al_\be \{\FI-\FP\}\solU(t)\|_{L^2_{\spa}L^2_{\vel}}^2
\lesssim
(1+\varepsilon t) \|   w^{\ell - |\beta| \rho} \pa^\al_\be \{\FI - \FP \}  f \|_{L^2_\spa \spaceL}^2.
\end{equation}
Here we explicitly used the hard potentials \eqref{kernelP} assumption.

Collect \eqref{CTheorem1C} and \eqref{desired.a.est.spec}, with $\beta>0$, into \eqref{est.upper.differe.weigh} and choose $\varepsilon>0$ sufficiently small; then plug these into \eqref{Theorem1B.weigh} to obtain uniformly in $t \ge 0$ that
\begin{equation}\label{Theorem1B.weigh.latest}
 \frac{d}{dt} \left( (1+\varepsilon t)^s  \CE_{\dK,\wN}(t) \right)
 + 
 \la (1+\varepsilon t)^s  
   \left(  \CD^h_{\dK}(t)
 +
  \CD^m_{\dK,\wN}(t) \right)
\lesssim (1+\varepsilon t)^{s-1-2\beta} .
\end{equation}
Now choose $s = 2\beta + \delta$ for any small $\delta\in (0,1)$ and integrate this in time to obtain
\begin{multline}\label{main.decay.proof1.latest}
\CE_{\dK,\wN}(t)
\lesssim
(1+\varepsilon t)^{-2\beta - \delta}
\CE_{\dK,\wN}(0)
+
(1+\varepsilon t)^{-2\beta - \delta}
\int_0^t du ~ (1+\varepsilon u)^{\delta-1}
\\
\lesssim
(1+ t)^{-2\beta - \delta} (1+ t)^{ \delta}
\approx (1+ t)^{-2\beta}.
\end{multline}
We thus have Theorem \ref{thm.extra.decay} when $\ALTsig \in (2, (\Ndim + 2)/2]$ and $\beta = \frac{\rho}{2}$.

Next consider the case when $\ALTsig =2$.  Then as in  \eqref{thm.ns.p2} and \eqref{desired.a.est}, we obtain
$$
\|  [a,b,c]\|_{L^2_{\spa}}(t)
\lesssim
\log(2+t)(1+t)^{-1}.
$$
In this situation we encounter a temporal log.  Then, using this last estimate, analogous to \eqref{CTheorem1C}, \eqref{Theorem1B.weigh.latest} and \eqref{main.decay.proof1.latest} we get the estimate
$$
\CE_{\dK,\wN}(t)  \lesssim  \log^2(2+t)(1+t)^{-2}.
$$
This estimate loses the optimal decay rate by a log.  However we can plug this estimate back into \eqref{thm.ns.p2} and \eqref{desired.a.est} to obtain  \eqref{desired.a.est.spec} with $\beta = 1$.  Then again following \eqref{CTheorem1C}, \eqref{Theorem1B.weigh.latest} and \eqref{main.decay.proof1.latest} we obtain 
\eqref{inidial.decay.use} with $\alpha = \ALTsig$ when  $\ALTsig =2$.

Lastly suppose that we have $\ALTsig \in (\Ndim/2, 2)$. This case only occurs for $\Ndim=3$.  Then we have  \eqref{inidial.decay.use} with $\alpha = \ALTsig-1$, and we have the time integral estimates 
 \eqref{thm.ns.p2} and \eqref{desired.a.est}, but with $\ALTsig \in (\frac{3}{2}, 2)$ this only yields 
 \eqref{desired.a.est.spec} with $\beta = \frac{3}{2} \ALTsig - 2$.

Then we apply the same procedure, e.g. \eqref{desired.a.est.spec}, \eqref{CTheorem1C}, \eqref{Theorem1B.weigh.latest} and \eqref{main.decay.proof1.latest}, to observe that, when $\ALTsig \in (3/2, 2)$, then  \eqref{inidial.decay.use}  holds with $\alpha = \left( 3\ALTsig - 4 \right)>\ALTsig-1$.  However since $3\ALTsig - 4 < \ALTsig$ when $\ALTsig \in (3/2, 2)$, 
we will need to iterate more.

Now suppose \eqref{inidial.decay.use} holds for $\alpha = \alpha_j$ for $j=0,\ldots, k$ where $\alpha_0= \left( 3\ALTsig - 4 \right)$.  Then we obtain from  \eqref{thm.ns.p2} and \eqref{desired.a.est} that \eqref{desired.a.est.spec} holds with $\beta = \beta_{k}$
where 
$$
\beta_{k} \eqdef 
\min\left\{\frac{\ALTsig}{2}+\alpha_{k}-1,\frac{\ALTsig}{2}, \alpha_{k}\right\}.
$$
Then combining this estimate with \eqref{desired.a.est.spec},  \eqref{CTheorem1C}, \eqref{Theorem1B.weigh.latest} and \eqref{main.decay.proof1.latest} we obtain that \eqref{inidial.decay.use} holds with $\alpha= \alpha_{k+1} = 2\beta_{k}$ where
$$
\alpha_{k+1} =2\beta_{k}
=
2\min\left\{\frac{\ALTsig}{2}+\alpha_{k}-1,\frac{\ALTsig}{2}\right\}.
$$
We continue this procedure until $\alpha_{k+1} = \ALTsig$ in which case we have 
Theorem \ref{thm.extra.decay} for this value of $\ALTsig$ and we terminate the algorithm.     


Now if $\alpha_{k}\ge 1$ then $\alpha_{k+1} =\ALTsig$ and when $\alpha_{k}<  1$ then $\alpha_{k+1} =2\alpha_{k}+\ALTsig-2$.  Therefore we define the following recursion for $k\ge 0$ by
$$
\tilde{\alpha}_{k+1} =2\tilde{\alpha}_{k}+\ALTsig-2
=
2^{k+1}\tilde{\alpha}_{0}+\left(2^{k+1} - 1 \right)(\ALTsig-2).
$$
If $\tilde{\alpha}_{0} = \alpha_{0} = \left( 3\ALTsig - 4 \right)$ with $\ALTsig \in (3/2, 2)$,
then $\tilde{\alpha}_{0}+(\ALTsig-2)>0$. Hence, $(\tilde{\alpha}_{k+1})$ is increasing and $\tilde{\alpha}_{k+1}  \to \infty$ as $k \to \infty$.    We conclude that for any $\ALTsig \in (3/2, 2)$ there exists a finite integer $k\ge 0$ such that  $\alpha_{k}\ge 1$ then $\alpha_{k+1} =\ALTsig$ and the algorithm terminates.   It follows that Theorem \ref{thm.extra.decay} holds when $\ALTsig \in (3/2, 2)$, when $\Ndim=3$.  
\end{proof}

The next section is dedicated to the proof of the functional interpolation inequalities we needed to use in order to prove the differential and integral inequalities in Sections \ref{sec.decayl} and \ref{sec.decayNL}, as well as the linear decay estimates Proposition \ref{linear.decay.B} and Proposition \ref{linear.decay.m}.  In Section \ref{secAPP:linear}, we prove the linear decay estimates.

\section{Functional interpolation inequalities and auxiliary results}
\label{secAPP:INTERP}

In this section, we develop several functional type Sobolev inequalities which we use to rigorously justify the proofs of the nonlinear energy estimates in Sections \ref{sec.decayl} and \ref{sec.decayNL}. We will use analogues of the Calder\'{o}n-Zygmund theory in the functional framework. The key point  is that the generalizations of the Littlewood-Paley Inequality and the 
H\"{o}rmander-Mikhlin Multiplier Theorem hold in the functional setting. Most of the claims that we will use can be deduced from the existing literature, such as \cite{MR1800316} or \cite{Tao2}.

An additional subtlety of working in the functional setting is that the Besov seminorm, as given in \eqref{Besovseminorm}, doesn't correspond to the convention for mixed norms given in \eqref{definition.mixednorm}. Namely, the use of the definition on \eqref{definition.mixednorm} would require taking the Besov norm of the function $\|f(\cdot)\|_{\hilbertONE}$, whereas in \eqref{Besovseminorm}, we localize in the frequency variable dual to $\spa$ inside the $\hilbertONE$ norm. Hence, we can't automatically use any of the Sobolev embeddings in mixed norm spaces, but we have to rederive them by looking at the dyadic components separately and by using the functional Calder\'{o}n-Zygmund theory.

To do this we will use the vector-valued Calder\'{o}n-Zygmund theory, for which we refer the reader to \cite[Chapter 5.5]{MR1800316} or to \cite{Tao2}. In Sub-section \ref{app.besov.h}, we recall the basic properties of homogeneous Besov spaces and their embedding properties in the scalar-valued setting. Sub-section \ref{app.sobolev.ineq} is devoted to the main properties of functional Besov spaces. In particular, we study the Sobolev-type inequalities which one can prove in these spaces. Furthermore, we recall the functional Littlewood-Paley theory and H\"{o}rmander-Mikhlin Multiplier theory in Sub-section \ref{subsectionA4}. Finally, in Sub-section \ref{sec:main.prod}, we prove the product estimates which we used in Section \ref{sec.decayl}. More precisely, we will prove the estimates we needed in order to deduce \eqref{lem.app.interp.sec}, \eqref{goal.prove.last} and \eqref{lem.app.interp.B}.

\subsection{Homogeneous Besov spaces}\label{app.besov.h}
For an integrable function $g: \threed_{\spa} \to\R$, its Fourier transform is defined by
\begin{equation*}
  \widehat{g}(\fva)= \CF g(\fva)\eqdef \int_{\threed} e^{-2\pi \rmi  x\cdot \fva } g(x)dx, \quad
  x\cdot
   \fva\eqdef\sum_{j=1}^\Ndim \spa_j \fva_j,
   \quad
   \fva\in\threed.
\end{equation*}
We define $\Lambda^k$, the Riesz potential of order $k \in \R$, by:
$$
\CF(\Lambda^k f) (\xi)\eqdef|\xi|^k \widehat{f}(\xi).$$
We now describe a standard Littlewood-Paley decomposition on $\mathbb{R}^{\Ndim}_\spa$ as follows.  Let 
$\phi \in C^\infty_c(\threed_\fva)$ be such that $\phi(\fva) = 1$ when $|\fva| \le 1$ and $\phi(\fva) = 0$ when $|\fva| \ge 2$.  Let 
$\varphi(\fva) = \phi(\fva) - \phi(2\fva)$ and $\varphi_j(\fva) = \varphi(\frac{\fva}{2^j})$
for ${j \in \Z}$.  Then
$
\sum_{j \in \Z} \varphi_j(\fva) = 1, \quad \fva \ne 0.
$
Further let $\CF(\psi)(\fva) = \varphi(\fva)$ and then $\psi_j(\spa) =  2^{\Ndim j} \psi(2^{ j} \spa)$ satisfies $\CF(\psi_j)= \varphi_j$.  We define 
$$
\Delta_j f \eqdef (\psi_j * f )(x).
$$
And if, say, $f\in L^p(\threed_{\spa})$, for $1<p<\infty$, then $f =\sum_{j \in \Z}\Delta_j (f)$, with convergence in $L^p$.
Now we define the homogeneous Besov seminorm for $1 \le q < \infty$ by
$$
\| f \|_{\dot{B}^{\ALTsig, q}_p \hilbertONE}
\eqdef
\left( \sum_{j \in \Z} \left( 2^{\ALTsig j} \| \Delta_j  f \|_{L^p_{\spa} \hilbertONE} \right)^q \right)^{1/q}, 
\quad 
\| f \|_{\dot{B}^{\ALTsig, \infty}_p \hilbertONE}
\eqdef
\sup_{j \in \Z} \left( 2^{\ALTsig j} \| \Delta_j f \|_{L^p_{\spa} \hilbertONE} \right).
$$
The following embeddings are known:
$$
\dot{B}^{\ALTsig, q'}_p\hilbertONE \subset \dot{B}^{\ALTsig, q}_p\hilbertONE,
\quad
\text{for $q \ge q'$}.
$$
As was noted at the beginning of this section, the Besov interpolation estimates in the scalar-valued setting are difficult to apply directly in the functional setting due to our definition of the functional Besov seminorm 
\eqref{Besovseminorm}. 

\subsection{Functional Sobolev-type inequalities in Besov spaces}\label{app.sobolev.ineq}

The main tool which is going to allow us to develop the functional Besov theory is the following Minkowski-type inequality for $f=f(\spa,\vel)$ and $g=g(\spa)$:
\begin{equation}\label{app.minkowski}
\| f * g \|_{\hilbertONE}\leq \| f  \|_{\hilbertONE} * |g|,
\end{equation}
which can be shown by an application of 
the Cauchy-Schwarz inequality in $\hilbertONE$.

We can use \eqref{app.minkowski} and Young's inequality to prove the following Bernstein-type inequalities: 
\begin{eqnarray}\label{app.bernstein.XV}
 \| \Delta_j f  \|_{L^p_{\spa} \hilbertONE}  \lesssim  2^{\left( \frac{\Ndim}{q} - \frac{\Ndim}{p} \right) j}  \| \Delta_j f  \|_{L^q_{\spa} \hilbertONE}, && \quad  \forall 1 \leq q \leq p \le \infty,
\\
\label{app.bernstein.XVI}
 \|\Delta_j \Lambda^s f\|_{L^p \hilbertONE}  \approx  2^{js} \|\Delta_j f\|_{L^p \hilbertONE}, && \quad \forall 1 \leq  p \le \infty,~ s \in \R.
\end{eqnarray}
These inequalities are useful since they give us estimates on the pieces we are considering in the Besov seminorm \eqref{Besovseminorm}.

We can also deduce the following $L^p_{\spa} \hilbertONE$ embedding by using \eqref{app.bernstein.XV}:

\begin{lemma}\label{lem.besov.emb}
Suppose that $\ALTsig > 0$ and $1\le p \le 2$.  We have the embedding 
$
L^p \hilbertONE \subset \dot{B}^{-\ALTsig, \infty}_q \hilbertONE
$ where $\frac{1}{p}-\frac{1}{q} = \frac{\ALTsig}{\Ndim}$.  In particular we have the estimate
\begin{equation}
\label{lem.besov.emb.bound}
\| f \|_{\dot{B}^{-\ALTsig, \infty}_q \hilbertONE} \lesssim \| f\|_{L^p_{\spa} \hilbertONE}.
\end{equation}
This holds for example with $\ALTsig = \frac{\Ndim}{2}$, $q=2$ and $p=1$.
\end{lemma}

As a consequence of \eqref{app.bernstein.XVI}, we obtain, for all $s_1,s_2 \in \mathbb{R}$ and for all $1 \leq p \leq \infty$:
\begin{equation}
\label{A3'}
\|\Lambda^{s_1}f\|_{\dot{B}^{s_2,p}_q \hilbertONE} \approx \|f\|_{\dot{B}^{s_1+s_2,p}_q \hilbertONE}.
\end{equation}

Let us note the following interpolation result:

\begin{lemma}\label{lem.app.opt.sob}
Suppose that $m \ne \ALTsig$. We have the following interpolation estimate:
\begin{equation}\notag 
\|  f \|_{ \dot{B}^{k,1}_p \hilbertONE}
\lesssim
\|  f \|_{\dot{B}^{m,\infty}_r \hilbertONE}^{1-\theta}
\|  f \|_{\dot{B}^{\ALTsig,\infty}_r \hilbertONE}^{\theta}
\end{equation}
where $0 < \theta < 1$ and $1 \le r \le p \le \infty$. We also require:
\begin{equation}\label{theta.gns}
 k+\frac{\Ndim}{r} -\frac{\Ndim}{p}
=
 m(1-\theta) + \ALTsig \theta.
\end{equation}
\end{lemma}

\begin{proof}[Proof of Lemma \ref{lem.app.opt.sob}]  
Without loss of generality suppose that $m < \ALTsig$.  
For $R \in \R$ to be chosen later, we expand out
$$
\| f \|_{\dot{B}^{k,1}_p \hilbertONE} 
=
 \sum_{j \in \Z}  2^{k j} \| \Delta_j  f \|_{L^p_{\spa} \hilbertONE} =  \sum_{j \ge R} + \sum_{j < R}.
$$
Now using \eqref{app.bernstein.XV} we obtain
$$
  \sum_{j \ge R} 2^{k j} \| \Delta_j  f \|_{L^p_{\spa} \hilbertONE}
\lesssim
  \sum_{j \ge R}  2^{k j+\left( \frac{\Ndim}{r} - \frac{\Ndim}{p} \right) j}
  \| \Delta_j  f \|_{L^r_{\spa} \hilbertONE}
\lesssim
  2^{\left(k +\left( \frac{\Ndim}{r} - \frac{\Ndim}{p} \right) -\ALTsig\right) R} 
  \| f \|_{\dot{B}^{\ALTsig,\infty}_r \hilbertONE} .
$$
For the other term
$$
  \sum_{j < R} 2^{k j} \| \Delta_j  f \|_{L^p_{\spa} \hilbertONE}
\lesssim
 \sum_{j < R}
2^{k j+\left( \frac{\Ndim}{r} - \frac{\Ndim}{p} \right) j}
  \| \Delta_j  f \|_{L^r_{\spa} \hilbertONE}
\lesssim
  2^{\left(k +\left( \frac{\Ndim}{r} - \frac{\Ndim}{p} \right) -m\right) R} 
  \| f \|_{\dot{B}^{m,\infty}_r \hilbertONE} .
$$
Choosing 
$
R = \log_2 \left(\frac{ \| f \|_{\dot{B}^{\ALTsig,\infty}_r \hilbertONE} }{\| f \|_{\dot{B}^{m,\infty}_r \hilbertONE} } \right)^{1/(\ALTsig-m)}
$ yields the result.
\end{proof} 

Notice that Lemma \ref{lem.app.opt.sob} directly implies an optimized functional Sobolev inequality of Gagliardo-Nirenberg-Sobolev-type.  We obtain directly that for $m \neq \ALTsig$:
\begin{equation}\label{func.sob.old}
\| \Lambda^k g \|_{ L^p_{\spa} \hilbertONE}
\lesssim
\| \Lambda^m g \|_{L^2_{\spa} \hilbertONE}^{1-\theta}
\| \Lambda^\ALTsig g \|_{L^2_{\spa} \hilbertONE}^{\theta}
\end{equation}
where $0 < \theta < 1$, $2 \le p \le \infty$,
and again $\theta$ satisfies \eqref{theta.gns}.

We will frequently use the following functional Sobolev type inequalities 
\begin{equation}\label{func.sob}
\|  g \|_{ L^q_{\spa} \hilbertONE}
\lesssim
\| \Lambda^\ALTsig g \|_{L^p_{\spa} \hilbertONE},
\quad \frac{1}{p} - \frac{1}{q} = \frac{\ALTsig}{\Ndim}, \quad 1 < p < q < \infty,
\end{equation}
which implies that $p=\frac{\Ndim q }{\Ndim + q\ALTsig}$.  
The functional inequality \eqref{func.sob} follows directly from \eqref{app.minkowski} combined with the standard fractional integration proof of  \eqref{func.sob} when there is not an additional function space $\hilbertONE$. In other words, the inequality \eqref{app.minkowski} allows us to reduce the proof of the vector-valued case to the scalar-valued case. For the details of the scalar-valued case, we refer the reader to \cite[Proposition A.3]{MR2233925}.

We observe the following Besov space variant of \eqref{func.sob}; namely:
\begin{equation}\label{func.bes}
\|  g \|_{ \dot{B}^{0,2}_q \hilbertONE}
\lesssim
\| \Lambda^\ALTsig g \|_{\dot{B}^{0,2}_p \hilbertONE},
\quad \frac{1}{p} - \frac{1}{q} = \frac{\ALTsig}{\Ndim}, \quad 1 < p < q < \infty.
\end{equation}
Notice that \eqref{func.bes} immediately follows from applying \eqref{func.sob} to the individual functions $\Delta_j g$ and then taking the $\ell^2_j$ norms.

We will also use the functional Sobolev embedding:
\begin{equation}
\label{func.sobolev}
\| g \|_{L^q_{\spa} \hilbertONE} \leq C_{q,n} ||g||_{H^k_x \hilbertONE}, \quad \mbox{ whenever } \quad k + \frac{\Ndim}{q} \geq \frac{\Ndim}{2}, \quad 2 \leq q < \infty,
\end{equation}
and
$ 
\| g \|_{L^{\infty}_{\spa} \hilbertONE} \leq C_n ||g||_{H^k_x \hilbertONE},
$
$\forall k > \frac{\Ndim}{2}$. 
In the endpoint case $k + \frac{\Ndim}{q} = \frac{\Ndim}{2}$, \eqref{func.sobolev} follows directly from \eqref{func.sob}.     In all the other non-endpoint cases, \eqref{func.sobolev} follows directly from \eqref{func.sob.old}.  
We can deduce the Besov version of \eqref{func.sobolev}:
\begin{equation}
\label{func.besov}
\| g \|_{\dot{B}^{0,2}_q \hilbertONE} \leq C_{q,n} ||g||_{H^k_x \hilbertONE}, \quad \mbox{ whenever } \quad k + \frac{\Ndim}{q} \geq \frac{\Ndim}{2}, \quad 2 \leq q < \infty,
\end{equation}
and
$ 
\| g \|_{\dot{B}^{0,2}_{\infty} \hilbertONE} \leq C_n ||g||_{H^k_x \hilbertONE},
$
$\forall k > \frac{\Ndim}{2}$.  We will use these inequalities to prove our main product estimates in Section \ref{sec:main.prod}.

First let us give one more Besov space interpolation estimate:

\begin{lemma}\label{besov.LM}
Fix $m>\ell \ge k$, and $1 \leq p \leq q \leq r \leq \infty$.  We have
\begin{equation}
\label{func.interp'}
\|g\|_{\dot{B}^{\ell,q'}_q\hilbertONE} \le  \|g\|_{\dot{B}^{k,r'}_r\hilbertONE}^{\theta}
\|g\|_{\dot{B}^{m,p'}_p\hilbertONE}^{1-\theta}.
\end{equation}
These parameters satisfy the following restrictions
\begin{equation}\notag
\ell = k \theta + m(1-\theta),
\quad \frac{1}{q}=\frac{\theta}{r}+\frac{1-\theta}{p},
\quad \frac{1}{q'}=\frac{\theta}{r'}+\frac{1-\theta}{p'}.
\end{equation}
Also $1 \leq p' \leq q' \leq r' \leq \infty$ and solving we have $\theta = \frac{m- \ell }{m-k} \in (0,1]$.
\end{lemma}
The most common case of this inequality that we will use is
\begin{equation}
\label{func.interp}
\|g\|_{\dot{B}^{\ell,2}_q\hilbertONE} \lesssim  \|g\|_{\dot{B}^{0,2}_r\hilbertONE}^{\theta}
\|g\|_{\dot{B}^{\ell+m,2}_p\hilbertONE}^{1-\theta},
\end{equation}
with $\ell, m > 0$, for $\theta=\frac{\ell}{\ell+m} \in (0,1)$.

\begin{proof}[Proof of Lemma \ref{besov.LM}]
Recall
$\|\Delta_j g\|_{L^q_x \hilbertONE}=\|\|\Delta_j g(x,\cdot)\|_{\hilbertONE}\|_{L^q_x}.$
We use H{\"{o}}lder's inequality in $x$ to obtain
$$
\|\Delta_j g\|_{L^q \hilbertONE}\leq \|\|\Delta_j g(x,\cdot)\|_{\hilbertONE}\|_{L^r_x}^{\theta} \|\|\Delta_j g(x,\cdot)\|_{\hilbertONE}\|_{L^p_x}^{1-\theta}= \|\Delta_j g\|_{L^r \hilbertONE}^{\theta} \|\Delta_j g\|_{L^p \hilbertONE}^{1-\theta}.
$$
The claim then follows by applying H\"{o}lder's inequality in $j$.
\end{proof}

\begin{remark}
Note that some other physical-space proofs of analogous interpolations don't easily generalize to the functional setting due to the definition \eqref{Besovseminorm}.
\end{remark}

\subsection{The Littlewood-Paley Inequality for Hilbert Space-valued functions}
\label{subsectionA4}
In this sub-section, we use the tools from vector-valued Calder\'{o}n-Zygmund theory to obtain additional functional Sobolev inequalities in Besov spaces.
In the following, we will always further suppose that $\hilbertONE$ is some separable Hilbert space acting only on the variables $\vel \in \threed_{\vel}$ as in Section \ref{sec.notation}.  
We sometimes also use $\hilbertTWO$ as a second Hilbert space which satisfies the same assumptions.

By using the vector-valued Calder\'{o}n-Zygmund theory, it is possible to show the following vector-valued Littlewood-Paley inequality (c.f.\cite{MR1800316}):

\begin{theorem}
\label{LittlewoodPaleyHilbertSpace}
Suppose that $\|f\|_{L^p_x \hilbertONE}< \infty$ for
some $1<p<\infty$.  Then
$$
\left\| \left( \sum_{j\in \Z} \|\Delta_j f(x)\|_{\hilbertONE}^2 \right)^{1/2}  \right\|_{L^p_x} \approx \|f\|_{L^p_x \hilbertONE}.
$$ 
\end{theorem}

From the Littlewood-Paley Theorem, we can deduce a Sobolev embedding bound:

\begin{lemma}
Fix $2 \leq p <\infty$, $s \in \mathbb{R}$.
Then we have:
\begin{equation}
\label{LpBesov2}
\|\Lambda^s f\|_{L^p_x \hilbertONE} \lesssim \|f\|_{\dot{B}^{s,2}_p \hilbertONE} \approx
\| \Lambda^s f\|_{\dot{B}^{0,2}_p \hilbertONE}.
\end{equation}
\end{lemma}

We will also need to use a Hilbert space-valued version of the H\"{o}rmander-Mikhlin Multiplier Theorem. 
For a function $m:\mathbb{R}^n_{\xi} \rightarrow \mathbb{C}$, we define the Fourier multiplier operator $m(D)$ on $L^2_x \hilbertONE$ by:
$$
\CF(m(D)f)(\xi)\eqdef m(\xi) \widehat{f}(\xi).
$$
Then we have (c.f. \cite{MR1800316} and \cite{Tao2})

\begin{proposition}
\label{Multiplier}
Suppose that $m:\mathbb{R}^n_{\xi} \rightarrow \mathbb{C}$ is a bounded function such that 
$$
|\nabla^k_{\xi} m(\xi)| \lesssim \frac{1}{|\xi|^k},
\quad
\forall \xi \in \mathbb{R}^n \setminus \{0\}, \quad 0 \leq k \leq n+2.
$$ 
Then $m(D): L^p_x \hilbertONE \rightarrow L^p_x \hilbertONE$ is a bounded operator for all $1<p<\infty$.
\end{proposition}
By using Proposition \ref{Multiplier}, with $m(\xi)\eqdef\frac{\xi^{\alpha}}{|\xi|^k}$, it is possible to deduce:

\begin{lemma} \label{partialderivativesLp}
Let $\hilbertONE$ be a Hilbert space and let $1<p<\infty$.
For all multiindices $\alpha=(\alpha_1,\ldots,\alpha_n)$ with $|\alpha|=k$, the following bound holds:
\begin{equation} \notag 
\|\partial^{\alpha} f\|_{L^p_x \hilbertONE} \lesssim \|\Lambda^k f\|_{L^p_x \hilbertONE}.
\end{equation}
In particular, when $p=2$, we can use Plancherel's Theorem to deduce:
\begin{equation}\notag 
\|\Lambda^k f\|_{L^2_x \hilbertONE} \lesssim \sum_{|\alpha|=k} \|\partial^{\alpha}f\|_{L^2_x \hilbertONE}.
\end{equation}
\end{lemma}


We note that Lemma \ref{partialderivativesLp} is important because the product estimates \eqref{lem.app.interp.sec}, \eqref{goal.prove.last} and \eqref{lem.app.interp.B} in Section \ref{sec.decayl} are given in terms of $\|\pa^{\al}f\|_{L^p_{\spa} \hilbertONE}$, whereas the estimates in this section are given in terms of $\|\Lambda^k f\|_{L^p_{\spa} \hilbertONE}$. 

\subsection{The main product estimates}
\label{sec:main.prod}

In this sub-section, we prove product estimates which allow us to deduce \eqref{lem.app.interp.sec}, \eqref{goal.prove.last} and \eqref{lem.app.interp.B}. The first bound we prove is Lemma \ref{lem.app.interp} which holds in the framework of general Hilbert spaces $\hilbertONE$ and $\hilbertTWO$ as in Section \ref{sec.notation}. The second result is Lemma \ref{last.lem.app} in which the Hilbert spaces are $L^2_{\vel}$ and $N^{s,\gamma}$. The proofs of both results are based on a case-by-case analysis in which one uses the functional Sobolev type inequalities from Sub-section \ref{app.sobolev.ineq}. We have not attempted to optimize the choice of $p$ and $q$, nor of the weight $\ell'$ used in Lemma \ref{last.lem.app}.  The first product estimate we prove is:

\begin{lemma}
\label{lem.app.interp}
For any $k\in\{0, 1, \ldots, \dK\}$ and $\wM \in\{0, 1, \ldots, k\}$ 
there exists $p, q\ge 2$ satisfying $\frac{1}{p}+\frac{1}{q} = \frac{1}{2}$ such that we have
\begin{equation}\notag 
\| \Lambda^{k-\wM}  f \|_{ L^q_{\spa}  \hilbertONE}
\| \Lambda^{\wM}   g \|_{ L^p_{\spa}  \hilbertTWO}
\lesssim
\|  g \|_{H^{\ksob}_{\spa} \hilbertTWO}
\|  \Lambda^{\wJ } f \|_{L^2_{\spa} \hilbertONE}
+
\|  f \|_{H^{\ksob}_{\spa} \hilbertONE}
\|   \Lambda^{\wJ } g \|_{L^2_{\spa} \hilbertTWO},
\end{equation}
where $\wJ \eqdef \min\{k+1, \dK\}$. Also $\hilbertONE$ and $\hilbertTWO$ are Hilbert spaces as in Section \ref{sec.notation} 
\end{lemma}

Note clearly that by using Lemma \ref{lem.app.interp}, together with Lemma \ref{partialderivativesLp}, we can directly deduce \eqref{lem.app.interp.sec} and  \eqref{lem.app.interp.B}.

\begin{proof}[Proof of Lemma \ref{lem.app.interp}]  
We prove this lemma in a series of several special cases.  
Suppose first that $\wM = 0$.  Then we choose $q=2^* = \frac{2\Ndim}{\Ndim - 2}$, as in \eqref{func.sob}, so that 
\begin{equation}\label{sob.est.one}
  \| \Lambda^{k} f \|_{L^{2^*}_{\spa} \hilbertONE}
\lesssim
  \|  \Lambda^{k+1} f \|_{L^2_{x}\hilbertONE}.
\end{equation}
In this case $p = \Ndim$ and we use \eqref{func.sobolev} to obtain  
\begin{equation}\label{sobe.est.two}
  \|  g \|_{L^\Ndim_{\spa} \hilbertTWO}
      \lesssim  
  \|  g \|_{H^{\ksob-1}_{\spa} \hilbertTWO},
\end{equation}
since $K^*_n -1 \geq \frac{n}{2}-1$.
This then establishes Lemma \ref{lem.app.interp} if $\wM =0$ and $\wJ=k+1$.  If $\wM =0$ and $\wJ=k=\dK$ we choose $q=2$ so that $p=\infty$  and we use the embedding 
$L^\infty_{\spa} \hilbertTWO \supset H^{\ksob}_{\spa} \hilbertTWO$
as in  \eqref{func.sobolev}.  Then we obtain
Lemma \ref{lem.app.interp} in this case as well.  The cases $\wM =k$, $k=0$ and $k=1$ can both be handled similarly.  

Next we consider the case $k \geq 2$, $\wJ = k = \dK$ and $i \in\{ 1, \ldots, k-1\}$.  In this situation we choose $q=2\frac{k}{k-i}$ and $p=2\frac{k}{i}$. Since $q,p \geq 2$, we can use \eqref{LpBesov2} and \eqref{func.interp} with $\theta = \frac{i}{k}$ twice to obtain:
\begin{multline}\notag 
\| \Lambda^{k-\wM}  f \|_{ L^{2k/(k-i)}_{\spa}  \hilbertONE}
\| \Lambda^{\wM}   g \|_{ L^{2k/i}_{\spa}  \hilbertTWO}
\lesssim \| \Lambda^{k-\wM}  f \|_{\dot{B}^{0,2}_{\frac{2k}{k-i}} \hilbertONE}
\| \Lambda^{\wM}   g \|_{\dot{B}^{0,2}_{\frac{2k}{i}} \hilbertTWO}
\\
\lesssim
\|  f \|_{\dot{B}^{0,2}_{\infty} \hilbertONE}^{\theta}
\|  \Lambda^{\wJ } f \|_{\dot{B}^{0,2}_2 \hilbertONE}^{1-\theta}
\|  g \|_{\dot{B}^{0,2}_{\infty} \hilbertTWO}^{1-\theta}
\|   \Lambda^{\wJ } g \|_{\dot{B}^{0,2}_2 \hilbertTWO}^{\theta} 
\\
\lesssim
\|  g \|_{H^{\ksob}_{\spa} \hilbertTWO}
\|  \Lambda^{\wJ } f \|_{L^2_{\spa} \hilbertONE}
+
\|  f \|_{H^{\ksob}_{\spa} \hilbertONE}
\|   \Lambda^{\wJ } g \|_{L^2_{\spa} \hilbertTWO}.
\end{multline}
Here we used Young's inequality, \eqref{func.besov} and the fact that 
$\|f\|_{\dot{B}^{0,2}_2 \hilbertONE} \approx \|f\|_{L^2_x \hilbertONE}$.  
Then, in this case, we also obtain the result.

We suppose then in the rest of our argument that we have $k\in\{2, \ldots, \dK-1\}$ and $i \in\{ 1, \ldots, k-1\}$ so that always $\wJ = k +1$.  This is the last case to consider in our proof.    We first use \eqref{func.sob}  with $ 2^* = \frac{2\Ndim}{\Ndim - 2} \leq q < \infty$ to obtain
\begin{equation} \label{app.func.sob.useful}
\| \Lambda^{k-\wM}  f \|_{ L^q_{\spa}  \hilbertONE}
\lesssim
\| \Lambda^{k+1-\wM}  f \|_{ L^{q'}_{\spa}  \hilbertONE},
\quad 
q' = \frac{\Ndim q}{\Ndim +q} \geq 2.
\end{equation}
Then we apply \eqref{LpBesov2}  and \eqref{func.interp} with $\theta = \frac{\wM}{k+1}$ to obtain:
\begin{equation}\label{useful.next.est}
\| \Lambda^{k+1-\wM}  f \|_{ L^{q'}_{\spa}  \hilbertONE}
\lesssim 
\| \Lambda^{k+1-\wM}  f \|_{ \dot{B}^{0,2}_{q'}  \hilbertONE}
\lesssim
\|   f \|_{ \dot{B}^{0,2}_{r'} \hilbertONE}^{\theta}
\| \Lambda^{k+1}  f \|_{ \dot{B}^{0,2}_2  \hilbertONE}^{1-\theta},
\end{equation}
where $r'$ is obtained from the restriction $\frac{1}{q'}=\frac{\theta}{r'}+\frac{1-\theta}{2}.$

We use \eqref{LpBesov2} and \eqref{func.interp} again (switching the $\theta$ and $1-\theta$) to obtain:
\begin{equation}\notag 
\| \Lambda^{i}  g \|_{ L^{p}_{\spa}  \hilbertTWO}
\lesssim
\| \Lambda^{i}  g \|_{\dot{B}^{0,2}_p  \hilbertTWO}
\lesssim
\|   g \|_{ \dot{B}^{0,2}_r  \hilbertTWO}^{1-\theta}
\| \Lambda^{k+1}  g \|_{ \dot{B}^{0,2}_2  \hilbertTWO}^{\theta},
\quad 
\theta = \frac{\wM}{k+1},
\end{equation}
where $r$ is obtained from $\frac{1}{p}=\frac{1-\theta}{r}+\frac{\theta}{2}.$

The value of $\theta$ is the same in both inequalities by design.  For these inequalities to hold, from \eqref{func.interp} and $\frac{1}{p}+\frac{1}{q} = \frac{1}{2}$, we require that $q$ satisfies:
\begin{equation}\label{require1}
\frac{2\Ndim}{\Ndim - 2} \leq q < \infty,
\quad
 \frac{2q}{q-2}=p \le \frac{2(k+1)}{\wM},
\quad
 \frac{\Ndim q}{\Ndim +q}=q'  \le \frac{2(k+1)}{k+1 - \wM}.
\end{equation}
We note that the first condition is equivalent to $q' \geq 2$, which we use in order to apply \eqref{func.interp}. The second condition is equivalent to $\frac{1}{p} \geq \frac{\theta}{2}$, and the third condition is 
$\frac{1}{q'} \geq \frac{1-\theta}{2}$.  These allow us to solve for $r$ and $r'$, which will then be at least $2$.  

We claim that we can find $q \geq 2^*$ satisfying \eqref{require1}.  Supposing for the moment that this is the case, we obtain:
\begin{multline}\notag 
\| \Lambda^{k-\wM}  f \|_{ L^{q}_{\spa}  \hilbertONE}
\| \Lambda^{\wM}   g \|_{ L^{p}_{\spa}  \hilbertTWO}
\lesssim
\| \Lambda^{k+1-\wM}  f \|_{\dot{B}^{0,2}_{q'}  \hilbertONE}
\| \Lambda^{\wM}   g \|_{ \dot{B}^{0,2}_{p} \hilbertTWO}
\\
\lesssim
\|  f \|_{\dot{B}^{0,2}_{r'} \hilbertONE}^{\theta}
\|  \Lambda^{\wJ } f \|_{\dot{B}^{0,2}_2 \hilbertONE}^{1-\theta}
\|  g \|_{\dot{B}^{0,2}_{r} \hilbertTWO}^{1-\theta}
\|   \Lambda^{\wJ } g \|_{\dot{B}^{0,2}_2 \hilbertTWO}^{\theta}
\\
\lesssim
\|  g \|_{H^{\ksob}_{\spa} \hilbertTWO}
\|  \Lambda^{\wJ } f \|_{L^2_{\spa} \hilbertONE}
+
\|  f \|_{H^{\ksob}_{\spa} \hilbertONE}
\|   \Lambda^{\wJ } g \|_{L^2_{\spa} \hilbertTWO}.
\end{multline}
This follows by collecting the estimates in this paragraph, using the embedding \eqref{func.besov} and applying Young's inequality.

Thus we will have proven Lemma \ref{lem.app.interp} as soon as we establish \eqref{require1}.  
The second condition in \eqref{require1} is equivalent to
$$
q \ge \frac{2(k+1)}{k+1 - \wM}.
$$
The third condition in \eqref{require1} is equivalent to:
\begin{equation}
\label{require1'}
\frac{1}{q} \geq \frac{(n-2)(k+1)-ni}{2n(k+1)}.
\end{equation}
If we assume that $i < \frac{\Ndim - 2}{\Ndim}( k + 1)$, 
then we can check that
$\left(\max\left\{\frac{2(k+1)}{k+1 - \wM},\frac{2\Ndim}{\Ndim - 2} \right\}, \frac{2\Ndim (k+1)}{(\Ndim-2)(k+1-\wM) - \Ndim \wM} \right)$ 
is a non-empty open interval, hence we observe that we can always satisfy \eqref{require1} when 
$i < \frac{\Ndim - 2}{\Ndim}( k + 1)$.
Alternatively if $i \geq \frac{\Ndim - 2}{\Ndim}( k + 1)$ then \eqref{require1'} is immediately satisfied since 
the right-hand side is non-positive.
\end{proof}
 
An additional product estimate which we will use is as follows.
 
 \begin{lemma}\label{last.lem.app}
 For any $k\in\{0, 1, \ldots, \dK\}$ and $\wM \in\{0, 1, \ldots, k\}$ 
there exists $p, q\ge 2$ satisfying $\frac{1}{p}+\frac{1}{q} = \frac{1}{2}$ such that we have
  \begin{equation}\notag
    \|  \Lambda^{k-\wM}   f \|_{L^q_{\spa} L^2_{\vel}}
  \| \Lambda^{\wM}  f  \|_{L^p_{\spa} \spacen}
  \\
\lesssim
\left( \|  f \|_{H^{\ksob}_{\spa} \spacen_{\wN'}}+ \| w^{\wN'}  f \|_{H^{\dK}_{\spa} L^2_{\vel}}\right)
\sum_{\wJ \le m \le \dK} \| \Lambda^{m}    \solU \|_{L^2_{\spa} \spacen},
\end{equation}
where $\wJ = \min\{k+1, \dK\}$ and we suppose that $\ell' \geq \PwNz$, where $\PwNz$ is defined in \eqref{wNz.def'}.
\end{lemma}

By using Lemma \ref{last.lem.app}, together with the result of Lemma \ref{partialderivativesLp}, we can directly deduce \eqref{goal.prove.last}.

\begin{proof}[Proof of Lemma \ref{last.lem.app}]
We first prove this inequality for the hard potentials \eqref{kernelP}, when $\gamma+2s \ge 0$.  In this case we  have:
$|f|_{L^2_v} \leq |f|_{L^2_{\gamma+2s}} \leq |f|_{N^{s,\gamma}}.$
Hence:
$$\|\Lambda^{k-i}f\|_{L^q_x L^2_{\gamma+2s}} \|\Lambda^i f\|_{L^p_x N^{s,\gamma}} 
\leq \|\Lambda^{k-i}f\|_{L^q_x N^{s,\gamma}} \|\Lambda^i f\|_{L^p_x N^{s,\gamma}}$$
We then use Lemma \ref{lem.app.interp} to deduce this expression is:
$$\lesssim \|f\|_{H^{K^*_n}_x N^{s,\gamma}} \|\Lambda^{j_{*}}f\|_{L^2_x N^{s,\gamma}}$$
Hence, the claim follows in the case of hard potentials if we take $\wN' \ge 0$.

Thus the rest of this proof is focused on the soft potential case of \eqref{kernelPsing}, when $\gamma+2s < 0$. 
We again prove it in a series of special cases.  
Consider when $\wM = k$ and $\wJ=k=\dK$.  In this situation choose $p=2$ and $q=\infty$, and use the functional Sobolev embedding $L^\infty_{\spa} L^2_{\vel} \supset H_x^{\ksob} L^2_{\vel}$ as in \eqref{func.sobolev}.  In this case the lemma holds for any $\wN' \ge 0$.  If alternatively $\wM = k$ and $\wJ=k+1$ we choose $p=2^*=\frac{2\Ndim}{\Ndim - 2}$ and $q=\Ndim$ and use the same estimates as in \eqref{sob.est.one} and \eqref{sobe.est.two} (we again only need $\wN' \ge0$).  Of course the case $k=0$ is also covered by this aforementioned analysis.

Next we turn to the case when $\wM=0$ and $\wJ = k=\dK$.  In this situation we must choose $q=2$ and $p=\infty$.  
By interpolation
\begin{equation}\label{interp.w.last}
\| \Lambda^{k} f \|_{L^2_{\spa} L^2_{\vel}}
\lesssim
\| \Lambda^{k} f \|_{L^2_{\spa} \spaceL}^{\tilde{\theta}}
\| w^{\left(\frac{-\gamma-2s}{2}\right) \frac{\tilde{\theta}}{1-{\tilde{\theta}}}} \Lambda^{k} f \|_{L^2_{\spa} L^2_{\vel}}^{1-\tilde{\theta}},
\end{equation}
which holds for any $\tilde{\theta} \in (0,1)$. 
In order to prove \eqref{interp.w.last}, we note that:
$$
| \Lambda^{k} f |_{L^2_{\vel}}=||w^{\frac{\alpha}{\tilde{\theta}}}\Lambda^{k} f|^{\tilde{\theta}} |w^{\frac{-\alpha}{1-\tilde{\theta}}}\Lambda^{k} f|^{1-\tilde{\theta}}|_{L^2_v}
\leq |w^{\frac{\alpha}{\tilde{\theta}}} \Lambda^{k} f|_{L^2_v}^{\tilde{\theta}}
|w^{\frac{-\alpha}{1-\tilde{\theta}}}\Lambda^{k} f|_{L^2_v}^{1-\tilde{\theta}},
$$
by H\"{o}lder's inequality.
We then apply H\"{o}lder's inequality in $\spa$ and we further choose $\alpha\eqdef\frac{(\gamma+2s)\tilde{\theta}}{2}$ to obtain \eqref{interp.w.last}.

On the other hand, we use \eqref{func.sob.old} (with $\theta$ and $1-\theta$ reversed) to obtain:
\begin{equation}\notag 
\|f \|_{L^{\infty}_{\spa} \spacen}
\lesssim
\| \Lambda^{m'} f \|_{L^2_{\spa} \spacen}^{\theta'}
\| \Lambda^{\wJ} f \|_{L^2_{\spa} \spacen}^{1-\theta'}.
\end{equation}
for appropriate $m'$ and $\theta'$.
Now from \eqref{func.sob.old}  since $i=0$, $\wJ=k=\dK$ and $p=\infty$ we have
$
\theta' = \frac{\dK-\frac{\Ndim}{2} }{\dK - m'}.
$
Furthermore always $\dK>\frac{\Ndim}{2}$  so that we choose 
$m' =0$. Hence, the condition $m' \neq \wJ$, which is an assumption in order to use \eqref{func.sob.old}  is satisfied. Moreover, we note that $\theta'=\frac{2\dK-\Ndim}{2\dK} \in (0,1)$.
Thus, we just choose $\tilde{\theta}\eqdef\theta' \in (0,1).$

By using the bounds in the previous two paragraphs, and Young's inequality, it follow that the contribution from this case is:
\begin{multline} \notag
\| \Lambda^{k} f \|_{L^2_{\spa} L^2_{\vel}} \|f \|_{L^{\infty}_{\spa} \spacen}
\lesssim 
\| w^{\left(\frac{-\gamma-2s}{2}\right) \frac{\tilde{\theta}}{1-{\tilde{\theta}}}} \Lambda^{k} f \|_{L^2_{\spa}L^2_v} \|\Lambda^{j_{*}}f\|_{L^2_x N^{s,\gamma}}
\\+
\| \Lambda^{j_{*}} f \|_{L^2_{\spa} \spaceL} \|f\|_{L^2_x N^{s,\gamma}} 
\\
\lesssim \left(\| w^{\left(\frac{-\gamma-2s}{2}\right) \frac{\tilde{\theta}}{1-{\tilde{\theta}}}} f \|_{H^K_{\spa}L^2_v} + \|f\|_{H^{K^*_n}_x N^{s,\gamma}_{\ell'}}\right)
\| \Lambda^{j_{*}} f \|_{L^2_{\spa} N^{s,\gamma}},
\end{multline} 
provided that $\ell' \geq 0$.
Finally, we note that we have to take:
\begin{equation}
\label{L1}
\ell' \geq
{\frac{\left(-\gamma-2s\right)}{2} \frac{\tilde{\theta}}{1-{\tilde{\theta}}}}
=
\frac{\left|\gamma+2s\right|}{2}
\left(\frac{2\dK - \Ndim}{\Ndim}\right).
\end{equation} 
This condition contributes to the size of the weight in \eqref{wNz.def'}.

Suppose next that $\wM =0$ and $\wJ =k+1$.    Then we choose $q=2^*=\frac{2n}{n-2}$ , use \eqref{sob.est.one} and interpolation (as in \eqref{interp.w.last} with $k$ replaced by $k+1$) to achieve that:
\begin{equation}\notag 
\| \Lambda^{k} f \|_{ L^{2^*}_{\spa} L^2_{\vel}}
\lesssim
\| \Lambda^{k+1} f \|_{L^2_{\spa} L^2_{\vel}}
\lesssim
\| w^{\left(\frac{-\gamma-2s}{2}\right) \frac{\tilde{\theta}}{1-{\tilde{\theta}}}}\Lambda^{k+1}  f \|_{L^2_{\spa} L^2_{\vel}}^{(1-\tilde{\theta})}
\| \Lambda^{k+1} f \|_{L^2_{\spa} \spaceL}^{\tilde{\theta}}.
\end{equation}
For the other term we use \eqref{func.sob.old} (since $K\ge 2 \ksob$) to deduce:
\begin{equation}\notag 
\|   f \|_{L^{\Ndim}_{\spa} \spacen}
\lesssim
\| f \|_{L^2_{\spa} \spacen}^{\theta'}
\| \Lambda^{\dK} f \|_{L^2_{\spa} \spacen}^{1-\theta'}.
\end{equation}
Now from \eqref{func.sob.old}  since $i=0$ and $p=\Ndim$ we have
$
\theta' = \frac{\dK-\frac{\Ndim}{2} +1}{\dK } \in (0,1).
$
We now choose $\tilde{\theta}=\theta'$. 
By using the obtained bounds, Young's inequality, the fact that $\wJ = k+1 \leq K$, and arguing as earlier, we obtain that, since $\ell' \geq 0$
\begin{multline} \notag
\| \Lambda^{k} f \|_{ L^{2^*}_{\spa} L^2_{\vel}}
\|   f \|_{L^{\Ndim}_{\spa} \spacen}
\lesssim 
\|f\|_{L^2_x N^{s,\gamma}}
\|\Lambda^{\wJ} f\|_{L^2_x N^{s,\gamma}}
\\
+
\| w^{\left(\frac{-\gamma-2s}{2}\right) \frac{\tilde{\theta}}{1-{\tilde{\theta}}}}\Lambda^{k+1}f\|_{L^2_x L^2_v}
\|\Lambda^{\dK} f\|_{L^2_x N^{s,\gamma}}
\\
\lesssim \left(\|f\|_{H^{K^*_n}_x N^{s,\gamma}_{\ell'}}+\| w^{\frac{-\gamma-2s}{2} \frac{\tilde{\theta}}{1-\tilde{\theta}}}f\|_{H^K_x L^2_v}\right) 
\sum_{j_{*} \leq m \leq K} \|\Lambda^m f\|_{L^2_x N^{s,\gamma}}.
\end{multline} 
The term coming from this contribution satisfies the required bound provided that:
\begin{equation}
\label{L2}
\ell' \geq \frac{-\gamma-2s}{2} \frac{\tilde{\theta}}{1-\tilde{\theta}} = \frac{-\gamma-2s}{2} \frac{2K-n+2}{n-2}
=
\frac{-\gamma-2s}{2}
\left( \frac{2 \dK}{\Ndim -2} -1 \right).
\end{equation}
We note that the case $k=1$ is covered by the previous arguments.

Now we consider the case when $k \geq 2$, $\wM \in \{1, \ldots, k-1\}$ with $\wJ=k=\dK$.  Take $q = \frac{2 \dK}{\dK-\wM }$ and $\theta=\frac{\dK-\wM}{\dK}$. We use \eqref{LpBesov2}, \eqref{func.interp} and then \eqref{interp.w.last} to obtain
\begin{multline}\notag 
\| \Lambda^{\dK-\wM} f \|_{ L^{2 \dK/(\dK-\wM)}_{\spa} L^2_{\vel}}
\lesssim
\| \Lambda^{\dK-\wM} f \|_{ \dot{B}^{0,2}_{2 \dK/(\dK-\wM)} L^2_{\vel}}
\lesssim
\|  f \|_{\dot{B}^{0,2}_{\infty} L^2_{\vel}}^{1-\theta}
\| \Lambda^{\dK} f \|_{\dot{B}^{0,2}_2 L^2_{\vel}}^{\theta}
\\
= \|  f \|_{\dot{B}^{0,2}_{\infty} L^2_{\vel}}^{1-\theta}
\| \Lambda^{\dK} f \|_{L^2_x L^2_{\vel}}^{\theta}
\lesssim
\|  f \|_{\dot{B}^{0,2}_{\infty} L^2_{\vel}}^{1-\theta}
\| w^{\left(\frac{-\gamma-2s}{2}\right) \frac{\tilde{\theta}}{1-{\tilde{\theta}}}}\Lambda^{\dK}  f \|_{L^2_{\spa} L^2_{\vel}}^{\theta(1-\tilde{\theta})}
\| \Lambda^{\dK} f \|_{L^2_{\spa} \spaceL}^{\theta\tilde{\theta}}.
\end{multline}
Here again $\tilde{\theta} \in (0,1)$ is arbitrary.

Furthermore, $p =  \frac{2 \dK}{\wM }$
and, for $\theta' = \frac{\dK - \frac{\Ndim}{2}\left(  \frac{\dK-\wM}{\dK}  \right) -\wM}{\dK }$, we use \eqref{func.sob.old} to obtain
\begin{equation}\notag 
\|  \Lambda^{\wM} f \|_{L^{2 \dK/\wM }_{\spa} \spacen}
\lesssim
\|  f \|_{L^2_{\spa} \spacen}^{\theta'}
\| \Lambda^{\dK} f \|_{L^2_{\spa} \spacen}^{1-\theta'}.
\end{equation}
Combining the previous estimates, it follows that:
$$
\| \Lambda^{\dK-\wM} f \|_{ L^{2 \dK/(\dK-\wM)}_{\spa} L^2_{\vel}} \|  \Lambda^{\wM} f \|_{L^{2 \dK/\wM }_{\spa} \spacen}
$$
$$
\lesssim \|  f \|_{\dot{B}^{0,2}_{\infty} L^2_{\vel}}^{1-\theta}
\| w^{\left(\frac{-\gamma-2s}{2}\right) \frac{\tilde{\theta}}{1-{\tilde{\theta}}}}\Lambda^{\dK}  f \|_{L^2_{\spa} L^2_{\vel}}^{\theta(1-\tilde{\theta})}
\| \Lambda^{\dK} f \|_{L^2_{\spa} \spaceL}^{\theta\tilde{\theta}} \|  f \|_{L^2_{\spa} \spacen}^{\theta'}
\| \Lambda^{\dK} f \|_{L^2_{\spa} \spacen}^{1-\theta'},
$$
Since $\theta > \theta'$, we can choose 
$\tilde{\theta} = \frac{\theta'}{\theta}=1-\frac{\Ndim}{2}\frac{1}{K} \in (0,1)$.
By the condition that $(1-\theta')+\theta \tilde{\theta}=1$, it follows that the above product is:
$$
\lesssim \| \Lambda^{\dK} f \|_{L^2_{\spa} \spacen} (\|f\|_{H^{K^*_n}_x L^2_v}+\| w^{\left(\frac{-\gamma-2s}{2}\right) \frac{\tilde{\theta}}{1-{\tilde{\theta}}}}\Lambda^{\dK}  f \|_{L^2_{\spa} L^2_{\vel}}+
\|f\|_{L^2_x N^{s,\gamma}}).
$$
Here, we used Young's inequality and \eqref{func.besov}.
Note that
$
\dK  >\frac{\Ndim}{2}\left(  \frac{\dK-\wM}{\dK}  \right) +\wM
$
since $\frac{\Ndim}{2}\frac{1}{\dK} < 1$. 
As before, we deduce that $\ell'$ has to satisfy the bound:
\begin{equation}
\label{L3'}
\ell' \geq \frac{-\gamma-2s}{2} \frac{\tilde{\theta}}{1-\tilde{\theta}}= \frac{-\gamma-2s}{2} 
\left( \frac{2}{\Ndim} \dK - 1 \right).
\end{equation}
This completes the our estimate in this case.

Notice the only case remaining is when $\wM \in \{1, \ldots, k-1\}$ with $k\in \{2, \ldots, \dK-1\}$.  
Let us first consider the subcase when:
\begin{equation}
\label{subcase1}
k+1 \geq \frac{\Ndim}{2}.
\end{equation}
Notice that \eqref{subcase1} covers all the remaining cases when $\Ndim \le 6$.  

The first step now is to use \eqref{app.func.sob.useful} and \eqref{useful.next.est} with $\hilbertONE = L^2_{\vel}$ to deduce: 
$$
\| \Lambda^{k-\wM}  f \|_{ L^q_{\spa}  L^2_{\vel}}
\lesssim
\| \Lambda^{k+1-\wM}  f \|_{ L^{q'}_{\spa}  L^2_{\vel}}
\lesssim
\|   f \|_{ \dot{B}^{0,2}_{r'}  L^2_{\vel}}^{\theta}
\| \Lambda^{k+1}  f \|_{ \dot{B}^{0,2}_{2}  L^2_{\vel}}^{1-\theta},
$$
whenever $q \geq \frac{2n}{n-2}=2^*$, $q'=\frac{nq}{n+q} \geq 2$, $r' \geq 2$ are such that for
$\theta = \frac{\wM}{k+1}$, one has:
\begin{equation}
\label{q'r'}
\frac{1}{q'}
=\frac{\theta}{r'}+\frac{1-\theta}{2}
= \frac{\wM}{r'(k+1)}+\frac{k+1 - \wM}{2(k+1)}.
\end{equation}
Let us suppose for now that we can find such a $q$ and $r'$. 
Then, by doing an interpolation similar to \eqref{interp.w.last} we obtain that the previous upper bound is
$$\lesssim \|f \|_{ \dot{B}^{0,2}_{r'}  L^2_{\vel}}^{\theta}
\| w^{\left(\frac{-\gamma-2s}{2}\right) \frac{\tilde{\theta}}{1-{\tilde{\theta}}}}\Lambda^{k+1}  f \|_{ L^{2}_{\spa}  L^2_{\vel}}^{(1-\theta)(1-\tilde{\theta})}
\| \Lambda^{k+1}  f \|_{ L^{2}_{\spa}  \spaceL}^{(1-\theta)\tilde{\theta}},$$
for any $\tilde{\theta} \in (0,1)$.  At the same time we use \eqref{func.sob.old} to deduce that for $p \geq 2$:
$$
\| \Lambda^{\wM}  f \|_{ L^p_{\spa}  \spacen}
\lesssim
\|  f \|_{ L^2_{\spa}  \spacen}^{1-\theta'}
\| \Lambda^{k+1}  f \|_{ L^2_{\spa}  \spacen}^{\theta'},
$$
for 
\begin{equation}
\label{eq:theta'} 
\theta' = \frac{\frac{\Ndim}{2} - \frac{\Ndim}{p} + \wM}{k+1},
\end{equation}
provided that $\theta' \in (0,1)$. We will choose $\tilde{\theta} \in (0,1)$,  such that $1-\theta' =(1-\theta)\tilde{\theta} $. This is possible whenever $1-\theta'  < 1-\theta$ or equivalently $\theta'  > \theta$ which is automatic, provided $p>2$. 

We suppose that we can choose $q \in [2^*,\infty), r' \geq 2$ satisfying \eqref{q'r'}. Then, since $q \in (2,\infty)$, we can find $p > 2$ such that $\frac{1}{p}+\frac{1}{q}=\frac{1}{2}$. Furthermore, we suppose that $\theta'$ from \eqref{eq:theta'} belongs to $(0,1)$. Under these assumptions, we can use Young's inequality and argue as before to deduce that the given contribution is:
\begin{multline}\notag
\| \Lambda^{k-\wM}  f \|_{ L^q_{\spa}  L^2_{\vel}}
\| \Lambda^{\wM}  f \|_{ L^p_{\spa}  \spacen}
\lesssim 
\|f\|_{H^{K^*_n}_x L^2_v} 
\|\Lambda^{k+1} f\|_{L^2_x N^{s,\gamma}}
\\
+ 
\left(\| w^{\left(\frac{-\gamma-2s}{2}\right) \frac{\tilde{\theta}}{1-{\tilde{\theta}}}}\Lambda^{k+1} f \|_{ L^{2}_{\spa}  L^2_{\vel}}
+ \|f\|_{L^2_x N^{s,\gamma}} 
\right) \|\Lambda^{k+1} f\|_{L^2_x N^{s,\gamma}},
\end{multline}
where we used \eqref{func.besov} for the first term. We must choose
\begin{equation}
\label{L4'}
\ell' \geq \frac{-\gamma-2s}{2} \frac{\tilde{\theta}}{1-\tilde{\theta}}.
\end{equation}
Then this term satisfies the desired bound.

We now choose $r'$ for which all of the assumptions will be satisfied. If we take $r'\eqdef n$, from \eqref{q'r'},  we deduce that: $q' \in (2,n).$ Consequently, $q=\frac{nq'}{n-q'}>2^*>2$ and $q$ is finite. Thus, we can choose $p \in (2,\infty)$ such that $\frac{1}{p}+\frac{1}{q}=\frac{1}{2}$. We now explicitly compute $\frac{1}{q}$ from \eqref{q'r'} and the fact that $\frac{1}{q}=\frac{1}{q'}-\frac{1}{n}$:
\begin{equation}\label{qCOMP}
\frac{1}{q}=\frac{\theta}{n}+\frac{1-\theta}{2}-\frac{1}{n}
=
\frac{\wM}{\Ndim (k+1)}+\frac{k+1 - \wM}{2(k+1)}-\frac{1}{\Ndim}
=\left(\frac{1}{2}-\frac{1}{n}\right) \left(1-\frac{i}{k+1}\right).
\end{equation}
We substitute the middle expression into \eqref{eq:theta'} and use $\frac{1}{2}-\frac{1}{p}=\frac{1}{q}$  to deduce:
$$\theta' = \frac{\frac{i}{k+1}+\frac{n}{2}\big(\frac{k+1-i}{k+1}\big)-1+i}{k+1}.$$
Since $i \geq 1$, it follows that $\theta'>0$.
To check that $\theta' <1$, it suffices to show
\begin{equation} \notag
\frac{i}{k+1}+\frac{n}{2}\left(\frac{k+1-i}{k+1}\right)-1+i < k+1.
\end{equation}
which is equivalent to:
\begin{equation}
\label{ikbound}
\frac{n}{2}+i\left(\frac{1}{k+1}-\frac{n}{2(k+1)}+1\right)-1<k+1.
\end{equation}
By assumption \eqref{subcase1}, it follows that $\frac{1}{k+1}-\frac{n}{2(k+1)}+1>0$. Hence, we need to verify \eqref{ikbound} when $i=k-1$. In this case, the condition is equivalent to $2k+4>n$ which holds by \eqref{subcase1}.  

In order to find the precise value for the lower bound for $\ell'$ in \eqref{L4'}, we need to compute $\tilde{\theta}$ explicitly. In order to do this, using \eqref{qCOMP}, we have
$$\frac{\tilde{\theta}}{1-\tilde{\theta}}=\frac{2(k+1)}{\Ndim -2} -1
\le \frac{2 \dK}{\Ndim -2} -1.$$

Consequently, in \eqref{L4'}, we can take:
\begin{equation}
\label{L4}
\ell' \geq \frac{-\gamma-2s}{2} \left( \frac{2 \dK}{\Ndim -2} -1 \right).
\end{equation}
This grants the desired bound under \eqref{subcase2}.

We now consider the second subcase (and the last case) when
\begin{equation}
\label{subcase2}
k<\frac{n}{2}-1.
\end{equation}
Notice that this case is only needed when $\Ndim > 6$.  
Also recall that $\wM \in \{1, \ldots, k-1\}$ with $k\in \{2, \ldots, \dK-1\}$. 
In the following we recall that $\wJt$ from \eqref{notation.wJt} is the largest integer which is strictly less than $\frac{\Ndim}{2}$.
Let us take $q \geq \frac{2n}{n-2\wJt}$.
We note that for even $n$, $\tilde{j}=\frac{n+1}{2}-1=\frac{n-1}{2}$ and, for odd $n$, $\tilde{j}=\frac{n}{2}-1=\frac{n-2}{2}$. Consequently, the above condition on $q$ becomes:
\begin{equation}
\label{qbound}
q
\geq
\left\{
\begin{array}{ccc}
2n  & \text{for $n$ odd}, \\
n & \text{for $n$ even}.
\end{array}
\right.
\end{equation}
We use \eqref{LpBesov2}, \eqref{func.bes} and \eqref{func.interp} with $q' = \frac{\Ndim q}{\Ndim +q \wJt}$ and $\theta = \frac{\wM}{k+\wJt}$ to deduce that: 
$$
\| \Lambda^{k-\wM}  f \|_{ L^q_{\spa}  L^2_{\vel}}
\lesssim
\| \Lambda^{k-\wM}  f \|_{ \dot{B}^{0,2}_q  L^2_{\vel}}
\lesssim 
\| \Lambda^{k+\wJt-\wM}  f \|_{ \dot{B}^{0,2}_{q'}  L^2_{\vel}}
\lesssim
\|   f \|_{ \dot{B}^{0,2}_{r}  L^2_{\vel}}^{\theta}
\| \Lambda^{k+\wJt}  f \|_{ L^{2}_{\spa}  L^2_{\vel}}^{1-\theta}.
$$
We note that $q' \geq 2$ since $q \geq \frac{2n}{n-2\wJt}$, 
and $r\ge 2$ is obtained from the relation $\frac{1}{q'} = \frac{\theta}{r} + \frac{1-\theta}{2}$. 

By using \eqref{func.besov}, and an interpolation similar to \eqref{interp.w.last}, we obtain that the above expression is
$$
\lesssim
\|f \|_{ H^{K^*_n}_x L^2_{\vel}}^{\theta}
\| w^{\left(\frac{-\gamma-2s}{2}\right) \frac{\tilde{\theta}}{1-{\tilde{\theta}}}}\Lambda^{k+\wJt}  f \|_{ L^{2}_{\spa}  L^2_{\vel}}^{(1-\theta)(1-\tilde{\theta})}
\| \Lambda^{k+\wJt}  f \|_{ L^{2}_{\spa}  \spaceL}^{(1-\theta)\tilde{\theta}},$$
for any $\tilde{\theta} \in (0,1)$.  Simultaneously we use \eqref{func.sob.old} to obtain that for $p \geq 2$ we have
\begin{equation}
\label{theta'2}
\| \Lambda^{\wM}  f \|_{ L^p_{\spa}  \spacen}
\lesssim
\|  f \|_{ L^2_{\spa}  \spacen}^{1-\theta'}
\| \Lambda^{k+\wJt}  f \|_{ L^2_{\spa}  \spacen}^{\theta'},\
\quad 
\theta' = \frac{\frac{\Ndim}{2} - \frac{\Ndim}{p} + \wM}{k+\wJt}.
\end{equation}
We want to choose $1-\theta' =\tilde{\theta} (1-\theta)$ which requires that 
$1-\theta'  < 1-\theta$ or equivalently $\theta'  > \theta$, which follows if we take $p>2$.  Finally, given $q \in[\frac{2n}{n-2\wJt},\infty)$ as before, we can find $p > 2$ such that $\frac{1}{p}+\frac{1}{q}=2$. 

For such a pair $(p,q)$ and for $\theta'$ as defined in \eqref{theta'2}, we can see that $\theta' \in (0,1)$.
Indeed since $p > 2$, it follows that $\theta'>0$. On the other hand, in order to check that $\theta'=\frac{\frac{n}{q}+i}{k+\tilde{j}}<1$, since $i \leq k-1$, we must observe that
\begin{equation}
\label{subcase2bound}
\frac{n}{q}+(k-1) < k+\tilde{j}.
\end{equation}
From \eqref{notation.wJt}, then \eqref{subcase2bound} holds if $q > {\Ndim}/{\left\lceil \Ndim/2 \right\rceil}$, which is always the case by \eqref{qbound}.

By using the previous estimates, and arguing similarly as before, we obtain
\begin{multline}\notag
\| \Lambda^{k-\wM}  f \|_{ L^q_{\spa}  L^2_{\vel}} \| \Lambda^{\wM}  f \|_{ L^p_{\spa}  \spacen}
\\
\lesssim \left(\| f \|_{ H^{K^*_n}_x L^2_{\vel}}
+ \| w^{\left(\frac{-\gamma-2s}{2}\right) \frac{\tilde{\theta}}{1-{\tilde{\theta}}}}\Lambda^{k+\wJt}  f \|_{ L^{2}_{\spa}  L^2_{\vel}}+\|  f \|_{ L^2_{\spa}  \spacen}\right)  \| \Lambda^{k+\wJt}  f \|_{ L^2_{\spa}  \spacen}.
\end{multline}
Because of \eqref{subcase2},  \eqref{notation.wJt}, and $K \geq 2\ksob = 2 \lfloor \frac{n}{2}+1 \rfloor$, it follows that $k + \wJt \leq K$. Hence, we are done if we have \eqref{L4'} with $\tilde{\theta}$ as given in this part.

As before, we compute $\tilde{\theta}$ explicitly:
$$
\tilde{\theta}=\frac{1-\theta'}{1-\theta}=\frac{1-\frac{\frac{\Ndim}{q}+i}{k+\wJt}}
{1-\frac{i}{k+\wJt}}=1-\frac{\Ndim}{q(k+\wJt-i)} .
$$
Therefore, using \eqref{subcase2}, \eqref{qbound} and \eqref{notation.wJt}, we can choose $q=\Ndim$ for odd $\Ndim$ and we deduce:
$$
\frac{\tilde{\theta}}{1-\tilde{\theta}} = \frac{q}{\Ndim} \left(k+\wJt-i \right) - 1
=  \left(k+\left\lceil \Ndim/2 \right\rceil-i-1 \right) - 1
\le   \frac{\Ndim}{2}+\left\lceil \Ndim/2 \right\rceil  - 3.
$$
For even $\Ndim$, we take $q=2\Ndim$ and we deduce that:
$$
\frac{\tilde{\theta}}{1-\tilde{\theta}} = 
2 \left(k+\left\lceil \Ndim/2 \right\rceil-i-1 \right) - 1
\le \Ndim+ 2 \left\lceil \Ndim/2 \right\rceil  - 7.
$$
Recall that in this case, $\Ndim > 6$.   Hence, we recall the definition of $\ell_0^d$ in \eqref{ell0d} and deduce that in this case, we need to take:
\begin{equation}
\label{L5}
\ell' \geq \frac{-\gamma-2s}{2} \ell_0^d.
\end{equation} 
By using \eqref{L1}, \eqref{L2}, \eqref{L3'}, \eqref{L4}, \eqref{L5}, and the definition of $\PwNz$ in \eqref{wNz.def'}, the lemma now follows. 
\end{proof}

\section{Linear decay in Besov spaces}
\label{secAPP:linear}

For the linearized Boltzmann equation \eqref{Boltz}, by dropping the non-linear term we obtain the  Cauchy problem for the linear Boltzmann equation \eqref{ls}  when $\sourceG = 0$:
\begin{gather}
 \partial_t f + \SB f
=
0,
\quad
f(0, x, v) = f_0(x,v), \quad \SB \eqdef \vel\cdot\na_x + \FL.
\label{BoltzLIN}
\end{gather}
Then as in \eqref{ls.semi} we can represent solutions to \eqref{BoltzLIN} with the solution operator:
\begin{equation}
    \solU(t)=\semiG(t)\solU_0, \quad \semiG(t)\eqdef e^{-t\SB}.
    \label{BoltzLIN.semi}
\end{equation}
In this section we will establish the large time decay rates for the linear Boltzmann equation \eqref{BoltzLIN}.  In the first part of this section, we obtain Besov space decay estimates for general initial data belonging to an appropriate Besov space.

In the second part of this section, we study the hard potential case \eqref{kernelP} and we obtain improved decay estimates for initial data which belongs to an appropriate Besov space, but which is also microscopic. We will see that, in this case, we obtain an additional decay factor of $t^{-\frac{1}{2}}$. The key to obtaining the additional decay will be a detailed understanding  the spectral properties of the Fourier transform of the linearized Boltzmann operator for small frequencies. Our analysis of these spectral properties is motivated by the work of Ellis and Pinsky \cite{MR0609540}.

\subsection{Linear decay rates in Besov spaces}
We are interested in obtaining decay estimates for the general linear problem \eqref{BoltzLIN}. 

\begin{theorem}\label{thm.decay.lin}  Suppose that $m , \ALTsig \in \R$ with $m+\ALTsig > 0$, $1 \le p \le \infty$ and $\wN\in \R$.  Smooth solutions to 
\eqref{BoltzLIN} satisfy, uniform in $t \ge 0$, the large time decay estimate
\begin{equation} \notag 
\|   w^{\wN} \semiG(t)\solU_0 \|_{\dot{B}^{m,p}_2 L^2_{\vel}}
\lesssim  \|  w^{\wN+\sigma} \solU_0 \|_{\dot{B}^{m,p}_2 L^2_{\vel} \cap \dot{B}^{-\ALTsig,\infty}_2 L^2_{\vel}} 
(1+t)^{-(m+\ALTsig)/2}.
\end{equation}
This holds if $ \|  w^{\wN+\sigma} \solU_0 \|_{\dot{B}^{m,p}_2 L^2_{\vel} \cap \dot{B}^{-\ALTsig,\infty}_2 L^2_{\vel}}< \infty$.  Now for the soft potentials \eqref{kernelPsing}  we need $\sigma > -(m + \ALTsig)(\gamma+2s)$, and for the hard potentials \eqref{kernelP} we can take $\sigma=0$. 
\end{theorem}

To prove Theorem \ref{thm.decay.lin}, we will use the following Lyapunov functional constructed in \cite[Theorem 2.3]{SdecaySOFT}:

\begin{theorem}\label{thm.tfli}
Fix $\wN\in \R$.  Let $\solU(t, \spa, \vel)$ be the solution to the Cauchy problem \eqref{BoltzLIN}. Then there is a weighted time-frequency functional $\CE_\wN(t,\xi)$
such that
\begin{equation}\label{thm.tfli.1}
    \CE_\wN(t,\xi) \approx \nsm w^\ell \hat{\solU}(t,\xi)\nsm_{L^2_{\vel}}^2,
\end{equation}
where for any $t\geq 0$ and $\xi\in \threed$ we have 
\begin{equation}\label{thm.tfli.2}
\partial_t \CE_{\wN} (t,\xi) +\la \left( 1 \wedge |\xi |^2\right)  \nsm w^{\wN}  \hat{\solU}(t,\xi)\nsm_{L^2_{\gamma+2s}}^2
\le 0.
\end{equation}
We use the notation $1 \wedge |\xi|^2 \eqdef\min\{1,|\xi|^2\}$.
\end{theorem}

\begin{proof}[Proof of Theorem \ref{thm.decay.lin}]  Using Theorem \ref{thm.tfli}, as in  \cite[Eqn (2.21)]{SdecaySOFT}, we have
\begin{equation}\label{boundH}
\CE_\wN(t,\xi)\lesssim e^{- \la \left( 1 \wedge |\xi|^2\right) t } \CE_\wN(0,\xi).
\end{equation}
The bound \eqref{boundH} only holds for the hard-potentials \eqref{kernelP}.  For the soft-potentials \eqref{kernelPsing}, it follows from  \cite[just below eqn (2.21)]{SdecaySOFT} that we have the estimate
\begin{equation}\label{boundS}
\CE_\wN(t,\xi)\lesssim \left(\frac{ t \left( 1 \wedge |\xi|^2\right)}{\sigma}   + 1\right)^{-\sigma} \CE_{\wN+\sigma'}(0,\xi).
\end{equation}
This estimate \eqref{boundS} holds for any $\sigma > 0$ with $\sigma' = -\sigma(\gamma+2s)>0$.  Now because of the degeneracy of the soft potentials in \eqref{kernelPsing}, the upper bound \eqref{boundS} loses a weight of order $\sigma'$ on the initial data.  Alternatively, the hard potential case \eqref{kernelP} does not lose a weight on the initial data in the time-frequency estimate \eqref{boundH}.\footnote{Notice the changing of the definition of the weight $w$ in \eqref{weigh} from \cite{SdecaySOFT}.}

In the following we will prove Theorem \ref{thm.decay.lin} for the soft potential  \eqref{kernelPsing} case of \eqref{boundS}.  However notice that the hard potential \eqref{kernelP} estimate \eqref{boundH} is better than \eqref{boundS}.  Thus the proof below will clearly apply in both situations.

We now recall the smooth function $\varphi_j(\xi)$ which is supported on $| \xi | \approx 2^j$ from Section \ref{app.besov.h}.
Then multiplying \eqref{boundS} by $\varphi_j^2(\xi)$ and integrating over $\xi \in \threed_\xi$ we obtain from \eqref{boundS}, \eqref{thm.tfli.1} and the Plancherel theorem that
\begin{equation}\label{boundSint}
\| w^{\wN} (\Delta_j \solU )(t) \|_{L^2_\spa L^2_\vel}
\lesssim 
\left(\frac{ t \left( 1 \wedge 2^{2j}\right)}{\sigma}   + 1\right)^{-\sigma/2} 
\| w^{\wN+\sigma'} \Delta_j \solU_0 \|_{L^2_\spa L^2_\vel}.
\end{equation}
If $j \ge 0$, we conclude from \eqref{boundSint}, for any $\sigma > 0$, that 
\begin{equation}\notag
\left(2^{mj}\| w^{\wN} (\Delta_j \solU )(t) \|_{L^2_\spa L^2_\vel}  \right)
\lesssim 
\left( t  + 1\right)^{-\sigma/2} 
\left(
2^{mj} \| w^{\wN+\sigma'} \Delta_j \solU_0 \|_{L^2_\spa L^2_\vel}
\right).
\end{equation}
If $j < 0$, we alternatively conclude from \eqref{boundSint} that 
\begin{multline}\notag
\left(
2^{mj}\| w^{\wN} (\Delta_j \solU )(t) \|_{L^2_\spa L^2_\vel}  
\right)
\\
\lesssim 
t^{-(m+\ALTsig)/2}
 \left( \sqrt{ t } 2^{j} \right)^{(m+\ALTsig)}  \left( \left( \sqrt{ t } 2^{j} \right)^{2}   + 1\right)^{-\sigma/2} 
\|  w^{\wN+\sigma'} \solU_0 \|_{\dot{B}^{-\ALTsig, \infty}_{2} L^2_{\vel}} .
\end{multline}
We complete our proof of Theorem \ref{thm.decay.lin}  by taking the $\ell_j^p$ norm of both sides of the last two inequalities.  
In particular we notice that for $\sigma > (m+\ALTsig )$ then
\begin{equation}\label{sumEXAMPLE}
\left\|  \left( \sqrt{ t } 2^{j} \right)^{(m+\ALTsig)}  \left( \left( \sqrt{ t } 2^{j} \right)^{2}   + 1\right)^{-\sigma/2} \right\|_{\ell_j^p}  \lesssim  1,
\end{equation}
for any $p \in [1,\infty]$ and the upper bound does not depend upon $t\ge 0$.
\end{proof}

\subsection{Faster Linear Decay rates}
\label{secAPP:AppendixC}
If we assume that the initial data is microscopic as in \eqref{form.p} and  \eqref{microscopicf0}, we will see that it is possible to obtain a better decay bound in the linear problem. From the analysis, we will see that the key to using the additional information is to understand the behavior of the degenerate macroscopic part of the solution in the low spatial frequency regime. As a result of a precise analysis of the spectral properties of this operator, given by Proposition \ref{eigenvalues}, we can obtain an additional factor of the frequency (c.f. the factor $\kappa$ in \eqref{eigenvalue2}), which will result in a better estimate.
The arguments in this subsection are restricted to the hard potential case \eqref{kernelP} due to the degeneracy of the operators needed in order to obtain the spectral decomposition (c.f. \eqref{EV3.bob}).

Given $\xi \in \mathbb{R}^n$, let us look at the following operator:
\begin{equation}\label{FT.SB}
\widehat{\SB}(\xi) \eqdef 2\pi \rmi v \cdot \xi + \FL.
\end{equation}
We define $\kappa\eqdef|\xi|$, and, assuming $\xi \neq 0$, we let $\omega \eqdef \frac{\xi}{|\xi|}$.
The following fact then holds for the eigenvalues of $\widehat{\SB}(\xi)$ when $\xi$ sufficiently close to zero:

\begin{proposition}
\label{eigenvalues}
There exists $\kappa_0>0$ and smooth radial functions $\zeta_j=\zeta_j(\xi)$ such that $\zeta_j(\xi)\in C^{\infty}(\{\xi \in \mathbb{R}^n, 0<|\xi| \leq \kappa_0\})$ for $j=1,\ldots,\Ndim+2$ and
\begin{enumerate}
\item[i)] Every $\zeta_j(\xi)$ is an eigenvalue of $\widehat{\SB}(\xi)$.
\item[ii)] The $\zeta_j(\xi)$ have the asymptotic expansion:
\begin{equation}
\label{eigenvalue1}
\zeta_j(\xi)=i\zeta^{(1)}_j \kappa + \zeta^{(2)}_j \kappa^2 + \mathcal{O}(\kappa^3),\,\,\mbox{as}\,\,\kappa \rightarrow 0.
\end{equation}
Here $\zeta^{(1)}_j \in \mathbb{R}$ and $\zeta^{(2)}_j>0$.
\item[iii)] Let us denote by $\FP_j(\xi)$ the eigenprojection corresponding to the eigenvalue $\zeta_j(\xi)$. Then, assuming that $\kappa_0>|\xi|>0$, one has:
\begin{equation}
\label{eigenvalue2}
\FP_j(\xi)=\FP^{(0)}_j(\omega)+\kappa \FP^{(1)}_j(\xi)
\end{equation}
and the eigenvalues $\zeta_j(\xi)$ are semisimple, in the sense that they don't give rise to a generalized eigenspace.

Moreover, $\FP^{(0)}_j(\omega)$ are orthogonal projections on $L^2_v$ which satisfy:
\begin{equation}
\label{eigenvalue4}
\FP=\sum_{j=1}^{\Ndim+2}\FP^{(0)}_j(\omega).
\end{equation}
Finally, the $\FP^{(1)}_j(\xi)$ are uniformly bounded on $L^2_v$ for $|\xi| \leq \kappa_0$.
\end{enumerate}
\end{proposition}

Let us point out that this type of result  was first established by Ellis and Pinsky \cite{MR0609540} in the setting of hard spheres. These arguments carry over to the setting of the non cut-off hard potential Boltzmann equation. We refer the reader to \cite{MR0609540} for details of the proof. We will now give an outline of the argument to obtain the existence of the eigenvalues.  We would like to solve the eigenvalue equation
\begin{equation}
\label{EV1.bob}
\left(\widehat{\SB}(\xi)-\zeta\right)\phi=(2 \pi \rmi v \cdot \xi + L - \zeta) \phi=0.
\end{equation}
We add to \eqref{EV1.bob} the macroscopic projection $\FP$ as
\begin{equation}
\label{EV2.bob}
\left(\widehat{\SB}(\xi)+\FP-\zeta\right)\phi=\FP  \phi.
\end{equation}
It can be shown that $\left(\widehat{\SB}(\xi)+\FP-\zeta\right)$ is invertible with bounded inverse on $L^2_{\vel}$, for sufficiently small $\zeta$ and $|\xi|$. Hence, \eqref{EV2.bob} can be rewritten as
\begin{equation}
\label{EV3.bob}
\phi= \left(\widehat{\SB}(\xi)+\FP-\zeta\right)^{-1}\FP  \phi.
\end{equation}
This equation \eqref{EV3.bob} says that if we know the $\Ndim+2$ coefficients of $\FP  \phi$ then we can calculate $\phi$ itself.  Thus we take $\FP$ of both sides of \eqref{EV3.bob} to obtain
\begin{equation}
\label{EV4.bob}
\FP\phi= \FP \left(\widehat{\SB}(\xi)+\FP-\zeta\right)^{-1}\FP  \phi.
\end{equation}
This grants a system of $\Ndim+2$ equations with $\Ndim +2$ unknowns, the macroscopic components \eqref{coef.p.def}, with parameter $\zeta$.  One now expands out this system of equations and does a comparison of coefficients of the velocity basis vectors to obtain the exact form of this system for the macroscopic components from \eqref{coef.p.def}. The smoothness properties of $\zeta_j(\xi)$ are deduced by the use of the implicit function theorem.

The following lemma will be useful in proving the additional decay.

\begin{lemma} 
\label{projections}
If we choose $\kappa$ sufficiently small, then there exists $C>0$ such that
$$\sum_{j=1}^{n+2} |\FP_j(\xi) f|_{L^2_{\vel}}^2 \geq C |\FP f|_{L^2_{\vel}}^2.$$
\end{lemma}
\begin{proof}
We notice that by construction, as in \cite{MR0609540}, we have
\begin{equation}
\label{Pj11Pf}
\FP^{(1,1)}_j(\xi) \FP f=\FP^{(1,1)}_j(\xi) f,
\end{equation}
and additionally $(\FI - \FP) \FP^{(1,1)}_j(\xi)=0$ and $\FP \FP^{(1,2)}_j(\xi)=0.$
Hence for $C_1,C_2>0$:
\begin{multline}\notag
|\FP_j(\xi) f|_{L^2_{\vel}}^2 = |\FP^0_j(\omega) f + \kappa \FP^{(1,1)}_j(\xi) f|_{L^2_{\vel}}^2 +  \kappa^2 |\FP^{(1,2)}_j(\xi)f|_{L^2_{\vel}}^2
\\
\geq C_1 |\FP^0_j(\omega) f|_{L^2_{\vel}}^2-C_2 \kappa^2 |\FP^{(1,1)}_j(\xi) f|_{L^2_{\vel}}^2
\\
=
C_1 |\FP^0_j(\omega) f|_{L^2_{\vel}}^2-C_2 \kappa^2 |\FP^{(1,1)}_j(\xi)\FP f|_{L^2_{\vel}}^2.
\end{multline}
The final equality follows from  \eqref{Pj11Pf}.

By Proposition \ref{eigenvalues}, it follows that the $\FP^{(1,1)}_j(\xi)$ are uniformly bounded on $L^2_{\vel}$ for $|\xi|$ sufficiently small, hence it follows that, for some $C_3>0$
\begin{equation}
\label{Pjfbound}
|\FP_j(\xi) f|_{L^2_{\vel}}^2 \geq C_1 |\FP^0_j(\omega) f|_{L^2_{\vel}}^2 - C_3 \kappa^2 |\FP f|_{L^2_{\vel}}^2.
\end{equation}
We now sum \eqref{Pjfbound} in $j=1,\ldots,\Ndim+2$ to deduce that, for $\kappa$ sufficiently small:
\begin{multline}\notag
\sum_{j=1}^{n+2} |\FP_j(\xi) f|_{L^2_{\vel}}^2 \geq C_1 \sum_{j=1}^{n+2} |\FP^0_j(\omega) f|_{L^2_{\vel}}^2 - C_3 (n+1) \kappa^2 |\FP f|_{L^2_{\vel}}^2
\\
= C_1 |\FP f|_{L^2_{\vel}}^2 - C_3 (n+1) \kappa^2 |\FP f|_{L^2_{\vel}}^2 \geq C |\FP f|_{L^2_{\vel}}^2.
\end{multline}
For the last equality, we used \eqref{eigenvalue4}.
\end{proof}

The next result establishes a relevant orthogonality property of $\widehat{\SB}(\xi)$

\begin{lemma}
\label{projectionPj}
Given $j=1,\ldots,n+2$ and $\FP_j(\xi)$, $\zeta_j(\xi)$ as in Proposition \ref{eigenvalues}, there exists $R_j(\xi)$ such that
$
\widehat{\SB}(\xi)=\zeta_j(\xi) \FP_j(\xi) + R_j(\xi)
$
and $\FP_j(\xi) R_j(\xi)^*=0$.
\end{lemma}

\begin{proof}
We define
$
R_j(\xi)\eqdef\widehat{\SB}(\xi)-\zeta_j(\xi)\FP_j(\xi).
$
Then since 
$
\widehat{\SB}(\xi)\FP_j(\xi)=\zeta_j(\xi) \FP_j(\xi)
$
we have that 
$
R_j(\xi)\FP_j(\xi) =0
$
or taking adjoints $\FP_j(\xi) R_j(\xi)^*=0$.  \end{proof}

As we will see, the above result will be useful when we want to separately study the evolution of each eigencomponent in time. In the end, we will obtain decay estimates coming from the precise asymptotics of the eigenvalues given by \eqref{eigenvalue1} in Proposition \ref{eigenvalues}. Putting all of these components together will allow us to obtain good decay estimates by using Lemma \ref{projections}.  We have

\begin{theorem}\label{linear.decay.A} 
Suppose that the initial condition $f_0$ in \eqref{BoltzLIN} satisfies \eqref{microscopicf0} and that we are in the hard potential case \eqref{kernelP}.  
Let $m$, $\ALTsig \in \R$ with $m+\ALTsig > 0$, $2\le p \le \infty$ and $\wN \ge 0$.  Then smooth solutions to 
\eqref{BoltzLIN} satisfy the large time decay estimate
\begin{equation} \notag 
\|   w^{\wN} \semiG(t)\solU_0 \|_{\dot{B}^{m,p}_{2} L^2_{\vel}}
\lesssim  \|  w^{\wN} \solU_0 \|_{\dot{B}^{m,p}_{2} L^2_{\vel} \cap \dot{B}^{-\ALTsig,\infty}_2 L^2_{\vel}}
(1+t)^{-(m+\ALTsig+1)/2}.
\end{equation}
This holds uniformly in $t \ge 0$ if $ \|  w^{\wN} \solU_0 \|_{\dot{B}^{m,p}_{2} L^2_{\vel} \cap \dot{B}^{-\ALTsig,\infty}_2 L^2_{\vel}}< \infty$.  
\end{theorem}

Notice that Theorem \ref{linear.decay.A} is more general than and directly implies Proposition \ref{linear.decay.m} from Section \ref{sec.decayNL}.  This follows from the definitions in Section \ref{secAPP:INTERP} if we take $p=2$.

\begin{proof}
We will prove Theorem \ref{linear.decay.A} in three steps.  In the first step, we study the case when $|\xi| \ge \kappa_0$ for any small $\kappa_0>0$.  In the second step we use the eigenvalue decomposition of $\widehat{\SB}$ from  \eqref{BoltzLIN} on $|\xi| \le \kappa_0$ for a small $\kappa_0>0$ as in Proposition \ref{eigenvalues} to obtain the decay of the macroscopic part of the solution.  Then in the last step we use an estimate from \cite{SdecaySOFT} to prove the decay of the microscopic part.

We recall the method used in the proof of Theorem \ref{thm.tfli}.  To begin, we notice that from \eqref{boundH}  we have that
\begin{equation}\label{boundSnew.int}
\| w^{\wN} (\Delta_k \solU )(t) \|_{L^2_\spa L^2_\vel}
\lesssim 
e^{-\zeta (1 \wedge 2^k)t} \| w^{\wN} \Delta_k \solU_0 \|_{L^2_\spa L^2_\vel}.
\end{equation}
Given $\kappa_0>0$ small choose $M>0$ such that $2^{2k} \ge \kappa_0$ whenever $k \ge -M$.  Define
$$
c_k \eqdef \| w^{\wN} (\Delta_k \solU )(t) \|_{L^2_\spa L^2_\vel}, 
\quad 
\| c_k \|_{\ell^p_{k\ge -M}}  \eqdef \left( \sum_{k\ge -M} | c_k|^p \right)^{1/p}, \quad 1 \le p < \infty,
$$
with an obvious modification when $p=\infty$.
Then from \eqref{boundSnew.int} and the above we deduce that
\begin{equation}\label{final.bound.largeM}
\| 2^{mk} c_k \|_{\ell^p_{k\ge -M}}
\lesssim 
\left(1+t\right)^{-j/2} 
\| w^{\wN}  \solU_0 \|_{\dot{B}^{m,p}_2 L^2_\vel}.
\end{equation}
This will hold for any $j> 0$.

For the second step, let us consider the Fourier transform of the problem \eqref{BoltzLIN}:
\begin{gather}
 \partial_t \widehat{f} + \widehat{\SB} \widehat{f}
=
0,
\quad
\widehat{f}(0, \xi, v) = \widehat{f_0}(\xi,v).
\label{FourierBoltzLIN'}
\end{gather}
Here we recall \eqref{FT.SB}.  
We apply the Littlewood-Paley Projection $\Delta_k$ to obtain:
\begin{gather}
 \partial_t \widehat{\Delta_k f} + \widehat{\SB} \widehat{\Delta_k f}
=
0,
\quad
\widehat{\Delta_k f}(0, \xi, v) = \widehat{\Delta_k f_0}(\xi,v).
\label{LPFourierBoltzLIN'}
\end{gather}
We will further take the complex conjugate of \eqref{LPFourierBoltzLIN'}, use that 
$$
\overline{\widehat{\SB}} = \widehat{\SB}^* =  \overline{\zeta_j(\xi)} \FP_j(\xi) + R_j(\xi)^*,
$$
and use Lemma \ref{projectionPj} when we apply $\FP_j(\xi)$ to the result to deduce:
\begin{equation}
\label{ODE}
\partial_t \FP_j(\xi) \overline{\widehat{\Delta_k f}(\xi)} 
+
 \overline{\zeta_j(\xi)} \FP_j(\xi) \overline{ \widehat{\Delta_k f}(\xi) }=0.
\end{equation}
Moreover, we integrate \eqref{ODE} in time and use \eqref{eigenvalue2} from Proposition \ref{eigenvalues} to deduce that for $| \xi | $ sufficiently small
$$
\FP_j(\xi) \overline{\widehat{\Delta_k f}(t,\xi,\vel)}=| \xi | e^{-\overline{\zeta_j(\xi)} t} 
\FP^{(1)}_j(\xi) \overline{\widehat{\Delta_k f_0}(\xi,\vel)}.
$$
We note that, in this way, \emph{we have gained an extra factor of $|\xi|$}.
Consequently, since by Proposition \ref{eigenvalues}, $\FP^{(1)}_j(\xi)$ is bounded on $L^2_{\vel}$, it follows that
$$
\int d\vel ~ |\FP_j(\xi) \overline{\widehat{\Delta_k f}(t,\xi,\vel)}|^2 \lesssim e^{-2 \text{Re}\left( \zeta_j(\xi) \right) t} | \xi |^2 
\int d\vel ~ |\widehat{\Delta_k f_0}(\xi,\vel)|^2.
$$
Notice that $\text{Re}\left( \zeta_j(\xi) \right) = \text{Re}\left( \overline{\zeta_j(\xi)} \right)$.
By \eqref{eigenvalue1} of Proposition \ref{eigenvalues} the above is
$$
\lesssim e^{-\zeta_j^{(2)} |\xi|^2 t} |\xi|^2 \int d\vel ~ |\widehat{ \Delta_k f_0}(\xi,\vel)|^2.
$$
Here, the $\zeta_j^{(2)}>0$ no longer depend on $\xi$.  We also used that $|\xi|$ is sufficiently small.  
Then for all $j=1,\ldots,n+2$ and $\sigma>0$, and for all $|\xi|$ sufficiently small
$$
\nsm \FP_j(\xi) \overline{\widehat{\Delta_k f}(t,\xi)}\nsm_{L^2_{\vel}}^2 \leq \frac{C}{t} (1+|\xi|^2 t)^{-\sigma}
\nsm \widehat{\Delta_k f_0}(\xi)\nsm_{L^2_{\vel}}^2.
$$
Consequently, we obtain an additional factor of $\frac{1}{t}$.
We sum the above inequality in $j=1,\ldots,n+2$, we note that the projection $\FP$ is real, and we use Lemma \ref{projections} to deduce that, for $|\xi|$ sufficiently small
$$
\nsm \FP \widehat{\Delta_k f}(t,\xi)\nsm_{L^2_{\vel}}^2
\leq \frac{C}{t}(1+|\xi|^2 t)^{-\sigma} \nsm \widehat{\Delta_k f_0}(\xi,\vel)\nsm_{L^2_{\vel}}^2.
$$
Consequently, by Plancherel's Theorem, it follows that for $k$ sufficiently negative, i.e. for $k<-M$ for some $M>0$ large
\begin{equation}
\label{LittlewoodPaleyC}
\|\Delta_k \FP f(t)\|_{L^2_{\spa}L^2_{\vel}} \lesssim \frac{1}{\sqrt{t}}(1+t2^{2k})^{-\frac{\sigma}{2}}\|\Delta_k f_0\|_{L^2_{\spa}L^2_{\vel}}.
\end{equation}
Now, analogously to the proof of Theorem \ref{thm.decay.lin}, using the estimate \eqref{final.bound.largeM} when $k \geq -M$, and $\ell=0$ we deduce for $1\le p \le \infty$ the following uniform in $t\ge 0$ inequality
\begin{equation}
\label{macroscopicONE}
\|\FP f(t)\|_{\dot{B}^{m,p}_2 L^2_{\vel}} \lesssim (1+t)^{-\frac{\ALTsig+m+1}{2}} 
\|f_0\|_{\dot{B}^{m,p}_2 L^2_{\vel} \cap \dot{B}^{-\ALTsig,\infty}_2 L^2_{\vel}}.
\end{equation}
This gives us the bound on the macroscopic term.

We now estimate the microscopic term.
Let us recall that we are working in the hard potential case \eqref{kernelP}, and hence by the discussion after \cite[Equation (2.25)]{SdecaySOFT} 
\begin{multline}
\label{2.25}
|\{\FI-\FP\}\widehat{f}(t,\xi)|_{L^2_{\vel}}^2 
\lesssim e^{-\lambda t} |\{\FI-\FP\}\widehat{f_0}(\xi)|_{L^2_{\vel}}^2 
\\
+
\int_0^t ds \, e^{-\lambda(t-s)} |\xi|^2 |\FP \widehat{f}(s,\xi)|_{L^2_{\vel}}^2,
\end{multline}
Let us now consider the second term in the upper bound of \eqref{2.25}.  We multiply this term by the Littlewood-Paley projection $\varphi_k^2(\xi)$ (from Section \ref{app.besov.h}), also multiply it by $|\xi|^{2m}$, integrate over $\xi \in \threed_\xi$ and then sum the $\ell^p_{k\le -M}$ norm for $2\le p \le \infty$.  We use the $\ell^p_{k\le -M}$ to denote the standard $\ell^p_{k}$ norm, as defined in Section \ref{sec.notation}, except that $\ell^p_{k\le -M}$ is only summed over frequencies $k\le -M$ for some large $M>1$.  In the first step of this procedure, similar to \eqref{LittlewoodPaleyC}, we have, for all $\ALTsig>0$
\begin{equation}\notag
\left(2^{2(m+1)k}\|\Delta_k \FP f(s)\|_{L^2_{\spa}L^2_{\vel}}^2 \right)
\lesssim s^{-1-(m+1)-\ALTsig} \left( s 2^{2k} \right)^{m+1+\ALTsig} (1+s2^{2k})^{-\sigma}
\| f_0\|_{\dot{B}^{-\ALTsig, \infty}_{2}L^2_{\vel}}^2.
\end{equation}
Then similarly to \eqref{sumEXAMPLE} we have that
\begin{equation}
\label{LittlewoodPaleySecondC}
\left\|
\left(2^{(m+1)k}\|\Delta_k \FP f(s)\|_{L^2_{\spa}L^2_{\vel}} \right) \right\|_{\ell^p_{k\le -M}}^2
\lesssim (1+s)^{-1-(m+1)-\ALTsig}
\| f_0\|_{\dot{B}^{-\ALTsig, \infty}_{2}L^2_{\vel}}^2
\end{equation}
Here we use (WLOG) that $s\ge 1$.  

Now we plug the estimate \eqref{LittlewoodPaleySecondC} into \eqref{2.25} to obtain
\begin{multline}
\notag
\left\| \left( 2^{m k}\|\{\FI-\FP\}\Delta_k \widehat{f}(t)\|_{L^2_{\vel}L^2_{\spa}} \right) \right\|_{\ell^p_{k\le -M}}^2
\lesssim 
e^{-\lambda t} \|\{\FI-\FP\}f_0\|_{\dot{B}^{m,p}_2 L^2_{\vel}}^2
\\
+ 
\| f_0\|_{\dot{B}^{-\ALTsig, \infty}_{2}L^2_{\vel}}^2 
\int_0^t ds \, e^{-\lambda(t-s)}  (1+s)^{-1-(m+1)-\ALTsig}
\end{multline}
Here we have supposed that $p\ge 2$ and we initially took the $\ell^{p/2}_{k\le -M}$ norm of the dyadic estimates. 
Finally, we use \cite[Lemma A.1]{ZhuStrain2} to estimate the time integral and deduce that
\begin{multline}
\label{2.27}
\left\| \left( 2^{m k}\|\{\FI-\FP\}\Delta_k \widehat{f}(t)\|_{L^2_{\vel}L^2_{\spa}} \right) \right\|_{\ell^p_{k\le -M}}^2
\lesssim 
e^{-\lambda t} \|\{\FI-\FP\} f_0\|_{\dot{B}^{m,p}_2 L^2_{\vel}}^2
\\
+ 
\| f_0\|_{\dot{B}^{-\ALTsig, \infty}_{2}L^2_{\vel}}^2 
(1+t)^{-m-\ALTsig-2}.
\end{multline}
Collecting \eqref{2.27} with \eqref{macroscopicONE} and \eqref{final.bound.largeM} yields Theorem \ref{linear.decay.A} when $\wN =0$ 

Lastly, when $\wN>0$, we recall the estimate \cite[Equation (2.9)]{SdecaySOFT}:
\begin{multline}\label{macroWeightINEQ}
\frac{d}{dt}\nsm w^{\wN}\{\FI-\FP\}\hat{\solU}(t,\xi)\nsm^2_{L^2} +
\la \nsm w^{\wN}\{\FI-\FP\}\hat{\solU}(t,\xi)\nsm^2_{L^2_{\gamma+2s}}
\\
\le C_\lambda |\xi|^2 \nsm w^{-\sigma} \hat{\solU}(t,\xi)\nsm^2_{L^2_{\vel}}
+
 C \nsm\{\FI-\FP\}\hat{\solU}(t,\xi)\nsm_{L^2(B_{C})}^2.
\end{multline}
This estimate will hold for any large $\sigma>0$.  Here $L^2(B_{C})$ denotes the $L^2_{\vel}$ norm on $B_C$, the ball of radius $C>0$ centered at the origin.  This estimate \eqref{macroWeightINEQ} is actually slightly stronger that what is stated in  \cite[Equation (2.9)]{SdecaySOFT}, however following the proof of  \cite[Equation (2.9)]{SdecaySOFT} then  \eqref{macroWeightINEQ} can be directly deduced.

From \cite[Equation (2.11)]{SdecaySOFT} we also have 
\begin{equation}\label{macroWeightellINEQ}
\frac{d}{dt}\nsm \{\FI-\FP\}\hat{\solU}(t,\xi)\nsm^2_{L^2_{\vel}} +
\la \nsm \{\FI-\FP\}\hat{\solU}(t,\xi)\nsm^2_{L^2_{\gamma+2s}}
\\
\le C_\lambda |\xi|^2 \nsm \FP \hat{\solU}(t,\xi)\nsm^2_{L^2_{\vel}}.
\end{equation}
Taking a suitable linear combination of \eqref{macroWeightINEQ} and \eqref{macroWeightellINEQ} we obtain
\begin{multline}\notag 
\frac{d}{dt}
\left( \delta\nsm w^{\wN}\{\FI-\FP\}\hat{\solU}(t,\xi)\nsm^2_{L^2_{\vel}} +
\nsm \{\FI-\FP\}\hat{\solU}(t,\xi)\nsm^2_{L^2_{\vel}}
\right)
\\
+
\la 
\left( \delta\nsm w^{\wN}\{\FI-\FP\}\hat{\solU}(t,\xi)\nsm^2_{L^2_{\gamma+2s}} +
\nsm \{\FI-\FP\}\hat{\solU}(t,\xi)\nsm^2_{L^2_{\gamma+2s}}
\right)
\le C |\xi|^2 \nsm w^{-\sigma} \hat{\solU}(t,\xi)\nsm^2_{L^2_{\vel}},
\end{multline}
for some suitably small $\delta>0$, since $\nsm \FP  \hat{\solU}(t,\xi) \nsm_{L^2_{\vel}} \lesssim \nsm w^{-\sigma} \hat{\solU}(t,\xi)\nsm_{L^2_{\vel}}$.  
We use the fact that $\gamma+2s \geq0$ to deduce
\begin{multline}\label{macroWeightINEQnew'}
\frac{d}{dt}
\left(\delta \nsm w^{\wN}\{\FI-\FP\}\hat{\solU}(t,\xi)\nsm^2_{L^2_{\vel}} +
\nsm \{\FI-\FP\}\hat{\solU}(t,\xi)\nsm^2_{L^2_{\vel}}
\right)
\\
+
\la \left(\delta \nsm w^{\wN}\{\FI-\FP\}\hat{\solU}(t,\xi)\nsm^2_{L^2_{\vel}} +
\nsm \{\FI-\FP\}\hat{\solU}(t,\xi)\nsm^2_{L^2_{\vel}}
\right)
\le C |\xi|^2 \nsm w^{-\sigma} \hat{\solU}(t,\xi)\nsm^2_{L^2_{\vel}},
\end{multline}
Now we  use the integrating factor $e^{-\lambda t}$.  Then, as in \eqref{2.25}, it follows from \eqref{macroWeightINEQnew'}, when $\wN \ge 0$, that 
\begin{multline}
\notag
\nsm w^{\wN} \{\FI-\FP\}\widehat{f}(t,\xi) \nsm_{L^2_{\vel}}^2 
\lesssim 
e^{-\lambda t} \nsm w^{\wN} \{\FI-\FP\}\widehat{f_0}(\xi) \nsm_{L^2_{\vel}}^2
+
\\ \int_0^t ds \, e^{-\lambda(t-s)} |\xi|^2 \nsm w^{-\sigma}  \widehat{f}(s,\xi)\nsm_{L^2_{\vel}}^2,
\end{multline}
which holds for any $\sigma>0$.  In order to obtain the above inequality, we also used the fact that
$\delta \nsm w^{\wN}\{\FI-\FP\}\hat{\solU}(t,\xi)\nsm^2_{L^2_{\vel}} +
\nsm \{\FI-\FP\}\hat{\solU}(t,\xi)\nsm^2_{L^2_{\vel}} \approx  \nsm w^{\wN}\{\FI-\FP\}\hat{\solU}(t,\xi)\nsm^2_{L^2_{\vel}}$. From here we follow the argument from \eqref{2.25} to \eqref{2.27} and the explanations below it to obtain Theorem \ref{linear.decay.A} when $\wN \ge 0$.   \end{proof}


\begin{bibdiv}
\begin{biblist}



\bib{AMUXY3}{article}{
    author = {Alexandre, Radjesvarane},
        author = {Morimoto, Y.},
            author = {Ukai, Seiji},
        author = {Xu, Chao-Jiang},
        author = {Yang, Tong},
	title = {Boltzmann equation without angular cutoff in the whole space: Qualitative Properties of Solutions},
       journal={Arch. Rational Mech. Anal.},
   volume={202},
   date={2011},
   number={2},
   pages={599-661},
    eprint = {arXiv:1008.3442v2},
}

\bib{ArsenioMasmoudi}{article}{
   author={Arsenio, Diogo},
   author={Masmoudi, Nader},
   title={A new approach to velocity averaging lemmas in Besov spaces},
   journal={preprint},
   eprint = {arXiv:1206.6659},
   }
   
\bib{MR575897}{article}{
    author={Caflisch, Russel E.},
     title={The Boltzmann equation with a soft potential. I, II},
   journal={Comm. Math. Phys.},
    volume={74},
      date={1980},
    number={1, 2},
     pages={71\ndash 95, 97\ndash 109},
}

\bib{MR0052553}{article}{
    author={Calder\'{o}n, Alberto P.},
    author={Zygmund, Antoni},
    title={On the existence of certain singular integrals},
    journal={Acta Math.},
    volume={88},
    pages={85-139},
    review={MR0052553},
}





\bib{MR2420519}{article}{
   author={Duan, Renjun},
   title={On the Cauchy problem for the Boltzmann equation in the whole
   space: global existence and uniform stability in $L\sp 2\sb {\xi}(H\sp
   N\sb x)$},
   journal={J. Differential Equations},
   volume={244},
   date={2008},
   number={12},
   pages={3204--3234},
   issn={0022-0396},
}



\bib{arXiv:0912.1742}{article}{
   author={Duan, Renjun},
   author = {{Strain}, Robert~M.},
    title = {Optimal Time Decay of the Vlasov-Poisson-Boltzmann System in ${\mathbb{R}}^3$}, 
       journal={Arch. Rational Mech. Anal.},
   volume={199},
   date={2011},
   number={1},
   pages={291-328},
   eprint = {arXiv:0912.1742},
   doi={10.1007/s00205-010-0318-6},
}


\bib{DuanLiuXu}{article}{
   author={Duan, Renjun},
   author={Liu, Shuangxian},
   author={Xu, Jiang},
   title={Global well-posedness in spatially critical Besov space for the Boltzmann equation},
   journal={preprint},
   eprint = {arXiv:1310.2727},
   }


\bib{2010arXiv1006.3605D}{article}{
   author={Duan, Renjun},
   author = {{Strain}, Robert~M.},
    title = {Optimal Large-Time Behavior of the Vlasov-Maxwell-Boltzmann System in the Whole Space},
       journal={Comm. Pure Appl. Math.},
   volume={64},
   date={2011},
   number={11},
   pages={1497--1546},
   doi={10.1002/cpa.20381},
   eprint = {arXiv:1006.3605v2},
}


\bib{MR2325837}{article}{
   author={Duan, Renjun},
   author={Ukai, Seiji},
   author={Yang, Tong},
   author={Zhao, Huijiang},
   title={Optimal convergence rates for the compressible Navier-Stokes
   equations with potential forces},
   journal={Math. Models Methods Appl. Sci.},
   volume={17},
   date={2007},
   number={5},
   pages={737--758},
   doi={10.1142/S021820250700208X},
}

\bib{MR2357430}{article}{
   author={Duan, Renjun},
   author={Ukai, Seiji},
   author={Yang, Tong},
   author={Zhao, Huijiang},
   title={Optimal decay estimates on the linearized Boltzmann equation with
   time dependent force and their applications},
   journal={Comm. Math. Phys.},
   volume={277},
   date={2008},
   number={1},
   pages={189--236},
}




\bib{DYZ2011}{article}{
   author={Duan, Renjun},
   author={Yang, Tong},
   author={Zhao, Huijiang},
   title={Global Solutions to the Vlasov-Poisson-Landau System},
   journal={preprint},
   date={2011},
 	eprint = {arXiv:1112.3261v1},
}

\bib{MR1800316}{book}{
   author={Duoandikoetxea Zuazo, Javier},
   title={Fourier Analysis},
   publisher={American Mathematical Society, Graduate Series in Mathematics},
   volume={29}
   place={Providence, RI},
   date={2000},
   isbn={0-8218-2172-5},
}



\bib{MR0609540}{article}{
   author={Ellis, Richard S.},
   author={Pinsky, Mark A.},
   title = {The first and second fluid approximations to the linearized Boltzmann equation}      
   journal = {J. Math. Pures Appl.},
   volume = {54},
   number = {9},
   date = {1975},
   pages ={125--156}, 
}







				
\bib{MR1379589}{book}{
   author={Glassey, Robert T.},
   title={The Cauchy problem in kinetic theory},
   publisher={Society for Industrial and Applied Mathematics (SIAM)},
   place={Philadelphia, PA},
   date={1996},
   pages={xii+241},
   isbn={0-89871-367-6},
}
		
\bib{gsNonCut0}{article}{
   author={Gressman, Philip T.},
      author={Strain, Robert M.},
   title={Global classical solutions of the Boltzmann equation without angular cut-off},
      journal={J. Amer. Math. Soc.},
   volume={24},
   date={2011},
   number={3},
   pages={771-847},
   doi={10.1090/S0894-0347-2011-00697-8},
    eprint = {arXiv:1011.5441v1},
}





\bib{gsNonCutA}{article}{
   author={Gressman, Philip T.},
      author={Strain, Robert M.},
   title={Global classical solutions of the Boltzmann equation with long-range interactions},
   date={March 30, 2010},
    journal={Proc. Nat. Acad. Sci. U. S. A.},
       volume={107},
   number={13},
   pages={5744-5749},
	eprint={doi: 10.1073/pnas.1001185107}
}




\bib{gsNonCutEst}{article}{
   author={Gressman, Philip T.},
      author={Strain, Robert M.},
   title={Sharp anisotropic estimates for the Boltzmann collision operator and its entropy production},
      journal={Advances in Math.},
   volume={227},
   date={August 20, 2011},
   number={6},
   pages={2349-2384},
   doi={10.1016/j.aim.2011.05.005},
    eprint = {arXiv:1007.1276v1},
}





\bib{MR2000470}{article}{
   author={Guo, Yan},
   title={The Vlasov-Maxwell-Boltzmann system near Maxwellians},
   journal={Invent. Math.},
   volume={153},
   date={2003},
   number={3},
   pages={593--630},
   issn={0020-9910},
}

\bib{MR2095473}{article}{
   author={Guo, Yan},
   title={The Boltzmann equation in the whole space},
   journal={Indiana Univ. Math. J.},
   volume={53},
   date={2004},
   number={4},
   pages={1081--1094},
   issn={0022-2518},
}




\bib{GuoWang2011}{article}{
   author={Guo, Yan},
      author={Wang, Yanjin},
   title={Decay of dissipative equations and negative Sobolev spaces},
    journal={preprint},
   date={2011},
   pages={1--37},
   eprint = {arXiv:1111.5660v1},
}
		







\bib{MR1938147}{book}{
   author={Lemari{\'e}-Rieusset, P. G.},
   title={Recent developments in the Navier-Stokes problem},
   series={Chapman \& Hall/CRC Research Notes in Mathematics},
   volume={431},
   publisher={Chapman \& Hall/CRC, Boca Raton, FL},
   date={2002},
   pages={xiv+395},
   isbn={1-58488-220-4},
   doi={10.1201/9781420035674},
}
	




 

\bib{MR2322149}{article}{
   author={Mouhot, Cl{\'e}ment},
   author={Strain, Robert M.},
   title={Spectral gap and coercivity estimates for linearized Boltzmann
   collision operators without angular cutoff},
   journal={J. Math. Pures Appl. (9)},
   volume={87},
   date={2007},
   number={5},
   pages={515--535},
   issn={0021-7824},
       eprint = {arXiv:math.AP/0607495},
}




\bib{MR2259206}{article}{
   author={Strain, Robert M.},
   title={The Vlasov-Maxwell-Boltzmann system in the whole space},
   journal={Comm. Math. Phys.},
   volume={268},
   date={2006},
   number={2},
   pages={543--567},
   issn={0010-3616},
}

\bib{strainSOFT}{article}{
    author={Strain, Robert M.},
     title={Asymptotic Stability of the Relativistic {B}oltzmann Equation for the Soft-Potentials},
   journal={Comm. Math. Phys.},
   volume={300},
   date={2010},
   number={2},
   pages={529--597},
   eprint={arXiv:1003.4893v1}
      doi={10.1007/s00220-010-1129-1},
}


\bib{SdecaySOFT}{article}{
    author={Strain, Robert M.},
     title={Optimal time decay of the non cut-off Boltzmann equation in the whole space},
    journal={Kinetic and related models},
   volume={in press},
   date={2012},
   pages={1--31},
   eprint={arXiv:1011.5561v1}
}

		
\bib{MR2209761}{article}{
   author={Strain, Robert M.},
   author={Guo, Yan},
   title={Almost exponential decay near Maxwellian},
   journal={Comm. Partial Differential Equations},
   volume={31},
   date={2006},
   number={1-3},
   pages={417--429},
   issn={0360-5302},
}

\bib{MR2366140}{article}{
   author={Strain, Robert M.},
   author={Guo, Yan},
   title={Exponential decay for soft potentials near Maxwellian},
   journal={Arch. Ration. Mech. Anal.},
   volume={187},
   date={2008},
   number={2},
   pages={287--339},
   issn={0003-9527},
}

\bib{ZhuStrain}{article}{
   author={Strain, Robert M.},
   author={Zhu, Keya},
   title={Large-Time Decay of the Soft Potential relativistic Boltzmann equation in $\R^3_{\spa}$}
    journal={Kinetic and Related Models},
   volume={5},
   date={2012},
   number={2},
   pages={383-415},
   doi={10.3934/krm.2012.5.383},
 	eprint = {arXiv:1106.1579v1},
}


\bib{ZhuStrain2}{article}{
   author={Strain, Robert M.},
   author={Zhu, Keya},
   title={The Vlasov-Poisson-Landau System in $\R^3_{\spa}$}
   journal={Archive for Rational Mechanics and Analysis},
   year={2013},
   volume={210},
   number={2},
   pages={615-671},
   eprint = {arXiv:1202.2471v1},
}


\bib{MR2233925}{book}{
   author={Tao, Terence},
   title={Nonlinear Dispersive Equations, Local and Global Analysis},
   publisher={American Mathematical Society, Conference Board of the Mathematical Sciences},
   place={Providence, RI},
   date={2006},
   volume={106},
   isbn={0-8218-4143-2},
}

\bib{Tao2}{article}{
  author={Tao, Terence},
  title={Lecture notes for Math 254A, Time-frequency analysis},
  eprint={http://www.math.ucla.edu/~tao/254a.1.01w/}
}

\bib{Tao3}{article}{
  author={Tao, Terence},
  title={Lecture notes for Math 247A, Fourier analysis},
  eprint={http://www.math.ucla.edu/~tao/247a.1.06f/}
}







\bib{MR882376}{article}{
   author={Ukai, Seiji},
   title={Solutions of the Boltzmann equation},
   conference={
      title={Patterns and waves},
   },
   book={
      series={Stud. Math. Appl.},
      volume={18},
      publisher={North-Holland},
      place={Amsterdam},
   },
   date={1986},
   pages={37--96},
}


\bib{MR677262}{article}{
   author={Ukai, Seiji},
   author={Asano, Kiyoshi},
   title={On the Cauchy problem of the Boltzmann equation with a soft
   potential},
   journal={Publ. Res. Inst. Math. Sci.},
   volume={18},
   date={1982},
   number={2},
   pages={477--519 (57--99)},
   issn={0034-5318},
   review={\MR{677262 (84h:82048)}},
   doi={10.2977/prims/1195183569},
}

\bib{MR1942465}{article}{
    author={Villani, C{\'e}dric},
     title={A review of mathematical topics in collisional kinetic theory},
 booktitle={Handbook of mathematical fluid dynamics, Vol. I},
     pages={71\ndash 305},
          book={
 publisher={North-Holland},
     place={Amsterdam},
        },
      date={2002},
}

\bib{MR881519}{article}{
   author={Wiegner, Michael},
   title={Decay results for weak solutions of the Navier-Stokes equations on
   ${\bf R}^n$},
   journal={J. London Math. Soc. (2)},
   volume={35},
   date={1987},
   number={2},
   pages={303--313},
   issn={0024-6107},
   doi={10.1112/jlms/s2-35.2.303},
}


\bib{MR2764990}{article}{
   author={Yang, Tong},
   author={Yu, Hongjun},
   title={Optimal convergence rates of classical solutions for
   Vlasov-Poisson-Boltzmann system},
   journal={Comm. Math. Phys.},
   volume={301},
   date={2011},
   number={2},
   pages={319--355},
   issn={0010-3616},
   review={\MR{2764990 (2012d:35380)}},
   doi={10.1007/s00220-010-1142-4},
}


\end{biblist}
\end{bibdiv}
\end{document}